\documentclass[11pt]{amsart}
\usepackage{amsfonts,latexsym,amsthm,amssymb,amsmath,amscd,euscript,tikz, tikz-cd}
\usepackage[alphabetic, msc-links, bibtex-style, nobysame]{amsrefs}
\usepackage{stackengine}
\usepackage{framed}
\usepackage{xfrac}
\usepackage[makeroom]{cancel}
\usepackage{faktor}
\usepackage{braket}
\usepackage{pgf,tikz,pgfplots}
\usepgfplotslibrary{groupplots}
\pgfplotsset{compat=1.18}
\usepgfplotslibrary{fillbetween} 
\usepackage{mathrsfs}
\usetikzlibrary{arrows}
\definecolor{wrwrwr}{rgb}{0.3803921568627451,0.3803921568627451,0.3803921568627451}
\definecolor{rvwvcq}{rgb}{0.08235294117647059,0.396078431372549,0.7529411764705882}
\definecolor{mblue}{rgb}{0.2, 0.3, 0.8}
\definecolor{morange}{rgb}{1, 0.5, 0}
\definecolor{mgreen}{rgb}{0.1, 0.4, 0.2}
\definecolor{mred}{rgb}{0.5, 0, 0}
\definecolor{ForestGreen}{RGB}{34,139,34}
\usepackage{float}

\numberwithin{equation}{section}

\usepackage{lmodern}
\usepackage{alphabeta}
\usepackage{supertabular}
\usepackage{amssymb}
\usepackage{enumerate}
\usepackage{stmaryrd}
\usepackage{bbm}
\usepackage{mathtools}
\usepackage{setspace}
\usepackage{tikz,tikz-cd}
\usetikzlibrary{matrix,calc,positioning,arrows,decorations.pathreplacing,patterns,knots}

\newcommand{\la}{\langle}
\newcommand{\rg}{\rangle}

\newtheorem{theorem}{{Theorem}}[section]
\newtheorem*{theorem*}{Theorem}
\newtheorem{lemma}[theorem]{Lemma}
\newtheorem{proposition}[theorem]{Proposition}

\newtheorem{corollary}[theorem]{Corollary}
\newtheorem*{corollary*}{Corollary}

\theoremstyle{definition}
\newtheorem{definition}{Definition}
\newtheorem{remark}{Remark}

\usepackage{lmodern,url,enumerate,mathtools}
\usepackage[hmargin = 1in,vmargin=1in]{geometry}
\usepackage{graphicx}
\usepackage{subcaption}

\usepackage{hyperref}
    \hypersetup{colorlinks=true,citecolor=ForestGreen,linkcolor = blue,urlcolor =black,linkbordercolor={1 0 0}}

\newcommand{\ve}{\varepsilon}

\newcommand{\mr}[1]{{\rm #1}}
\newcommand{\mres}{\mathbin{\vrule height 1.6ex depth 0pt width
0.13ex\vrule height 0.13ex depth 0pt width 1.3ex}}

\newcommand{\cA}{\mathcal{A}}\newcommand{\cB}{\mathcal{B}}

\newcommand{\cE}{\mathcal{E}}
\newcommand{\cH}{\mathcal{H}}

\newcommand{\cL}{\mathcal{L}}
\newcommand{\cN}{\mathcal{N}}

\newcommand{\cR}{\mathcal{R}}

\newcommand{\cU}{\mathcal{U}}

\newcommand{\bE}{\mathbb{E}}
\newcommand{\bG}{\mathbb{G}}

\newcommand{\bN}{\mathbb{N}}

\newcommand{\bR}{\mathbb{R}}

\newcommand{\nc}{\newcommand}
\nc{\sn}{\mr{sn}}
\nc{\cn}{\mr{cn}}
\nc{\dn}{\mr{dn}}

\allowdisplaybreaks

\nc{\on}{\operatorname}
\nc{\p}{\partial}
\nc{\ol}{\overline}
\nc{\ul}{\underline}
\nc{\pa}{\partial}

\nc{\pb}{\partial_b}
\nc{\pc}{\partial_c}
\nc{\pd}{\partial_d}
\nc{\pe}{\partial_e}
\nc{\pf}{\partial_f}
\nc{\pg}{\partial_g}
\nc{\ph}{\partial_h}
\nc{\pari}{\partial_i}
\nc{\pj}{\partial_j}
\nc{\pk}{\partial_k}
\nc{\pl}{\partial_l}
\nc{\pell}{\partial_\ell}
\nc{\parm}{\partial_m}
\nc{\pn}{\partial_n}
\nc{\po}{\partial_o}
\nc{\pp}{\partial_p}
\nc{\pq}{\partial_q}
\nc{\pr}{\partial_r}
\nc{\ps}{\partial_s}
\nc{\pt}{\partial_t}
\nc{\pu}{\partial_u}
\nc{\pv}{\partial_v}
\nc{\pw}{\partial_w}
\nc{\px}{\partial_x}
\nc{\py}{\partial_y}
\nc{\pz}{\partial_z}

\numberwithin{equation}{section}
\makeatletter
\@addtoreset{equation}{section}
\makeatother

\usepackage{mathtools}
\mathtoolsset{showonlyrefs}

\author[Antonio De Rosa]{Antonio De Rosa}
\address{{\href{mailto:antonio.derosa@unibocconi.it}{antonio.derosa@unibocconi.it}} \hfill Department of Decision Sciences and BIDSA, Bocconi University}
\author[Benjy Firester]{Benjy Firester}
\address{{\href{mailto:benjyfir@mit.edu}{benjyfir@mit.edu}} \hfill Department of Mathematics, MIT}
\author[Raphael Tsiamis]{Raphael Tsiamis}
\address{
{\href{mailto:r.tsiamis@columbia.edu}{r.tsiamis@columbia.edu}}\hfill Department of Mathematics, Columbia University}

\title[Regularity of anisotropic stationary varifolds]{Regularity of varifolds with bounded anisotropic first variation} 

\date{\today}

\begin{document}

\begin{abstract}
We prove an $\varepsilon$-regularity theorem for $m$-varifolds with mean curvature in $L^p$, $p>m$, with respect to an anisotropic integrand satisfying the quadratic exposed condition: near sufficiently flat density-one points, such varifolds are representable as $C^{1,\alpha}$ graphs. We also exhibit the first examples of anisotropies satisfying the uniform scalar atomic condition that are not close to any ellipsoidal norm. These include the $\ell^q$ norms in every dimension and codimension, for explicit ranges of $q$, and a new class of axisymmetric anisotropies.
\end{abstract}

\maketitle

\section{Introduction}

We prove an $\varepsilon$-regularity theorem for $m$-varifolds with mean curvature in $L^p$, $p>m$, with respect to an anisotropic integrand satisfying the quadratic exposed condition: near sufficiently flat density-one points, such varifolds are representable as $C^{1,\alpha}$ graphs.
This result establishes an anisotropic analogue of Allard's celebrated regularity theorem in every dimension and codimension~\cite{allard-first-variation}.
Allard's regularity theorem is a fundamental link between the measure-theoretic and the classical theories of minimal surfaces.
Besides its intrinsic relevance, this $\varepsilon$-regularity theorem is an important input in compactness, blow-up, and min-max arguments throughout geometric analysis; see, for instance,~\cites{simon-gmt,delellis-allard}.

For anisotropic energies, several key ingredients of the isotropic theory are unavailable.
Allard observed that, unless the integrand is induced by an inner product, the usual monotonicity formula fails~\cite{allard-monotonicity}, so the mass ratios may not be monotone, and the blowup limits are not necessarily cones.
More importantly for regularity, the lower density estimates used in the isotropic argument cannot be obtained from monotonicity.
This difficulty arises in several other questions concerning anisotropic minimal surfaces, notably in the study of Plateau problems, min-max constructions, and geometric flows.
We refer to~\cites{derosa-survey,anisotropic-min-max,DePhilippisDeRosaLi2024,derosa-halavati-wang} for recent dvelopments.

In~\cite{allard-regularity}, Allard obtained an anisotropic regularity theorem for codimension-one varifolds, assuming a uniform lower bound for mass ratios and employing a barrier argument specialized to hypersurfaces.
Complementary recent work of Kolasi\'nski and Santilli proves quadratic flatness and countable $C^2$-rectifiability under bounded anisotropic mean curvature, and establishes locality and perpendicularity properties of the anisotropic mean curvature \cites{kolasinski-santilli,perpendicularity-and-locality}.

In arbitrary codimension, substantial work has centered around finding suitable ellipticity conditions for anisotropic integrands.
De Philippis, De Rosa, and Ghiraldin introduced the atomic condition and proved an Allard-type rectifiability theorem under locally bounded anisotropic first variation~\cite{derosa-ghiraldin}.
The atomic condition was subsequently proved to imply Almgren strict ellipticity and strict polyconvexity~\cites{derosa-kolasinski , derosa-lei-young }.
In their work on Allard regularity, De Rosa and Tione introduced a stronger condition, which they called the uniform scalar atomic condition (USAC)~\cite{derosa-tione-regularity}.
Under this condition, they established a Caccioppoli inequality and proved partial $C^{1,\alpha}$ regularity for Lipschitz graphs, with a later extension to boundary regularity~\cite{DRR}.
This interior regularity result was later extended to boundary points \cite{DRR}.
More recently, Jakimiuk, Kolasi\'nski, and Le\'sniak introduced the quadratic exposed condition (QEC), which retains the Caccioppoli property of USAC but allows a larger class of integrands.
The last remaining step in arbitrary codimension was how to obtain the graphical representation for varifolds, for which the corresponding step in Allard's proof fails in the absence of lower density estimates.

In place of a monotonicity formula, the anisotropic Michael-Simon inequality provides the key ingredient for obtaining the density estimates required for the regularity argument~\cite{michael-simon}.
Firester and Tsiamis~\cite{anisotropic-michael-simon} recently proved such an inequality in arbitrary dimension and codimension under a general spectral condition.
In particular, their result applies to convex anisotropic integrands in the hypersurface case, suitable perturbations of the area integrand in arbitrary dimension and codimension, and certain $\ell^p$ anisotropies.
Combined with Allard's work~\cite{allard-regularity}, it also yields the regularity of hypersurfaces with
bounded anisotropic mean curvature in every dimension.
Earlier, De Philippis and Pigati established an anisotropic Michael-Simon inequality for surfaces in $\bR^3$ and showed, for integrands satisfying the atomic condition, that the Michael-Simon inequality is equivalent to compactness of rectifiable varifolds with bounded anisotropic first variation~\cite{dephilippis-pigati}.

The purpose of this paper is to combine these advances and complete the Allard scheme for varifolds in arbitrary dimension and codimension with respect to an anisotropy that satisfies the QEC condition and a Michael-Simon inequality.

After the first version of this paper appeared, De Philippis, Gennaioli, Pigati, and Rindler \cite{galf} established the Michael-Simon inequality under the condition \textup{(AC1)}. 
Since QEC implies AC~\cite{atomic-condition}, and thus \textup{(AC1)}, it directly implies the relevant Michael-Simon inequality.

Our first main result is the anisotropic version of Allard's interior regularity theorem.
\begin{theorem}\label{thm:allard}
Consider an anisotropy $\Psi \in C^2(\bG(N,m),(0,\infty))$ satisfying the quadratic exposed condition \textup{QEC}, and let $p \in (m,\infty]$, $p\geq 2$. There exist constants $\ve_{\star}>0$, $\gamma_{\star} \in (0, \frac{1}{16}]$, and $\alpha_{\star} \in (0, \frac{1}{2}(1-\frac{m}{p}))$, depending on $\Psi,N,m,p$, and a non-decreasing modulus $\omega_{\star} : [0,\ve_{\star}] \to [0,\infty)$, with $\lim_{t \downarrow 0} \omega_{\star}(t) =0$, satisfying the following property.

Let $V$ be an integral $m$-varifold in an open neighborhood of $\bar{B}_{2r}(x_0)$,
with $x_0 \in \textup{spt}\,\|V\|$ and $\Psi$-mean curvature $H_{\Psi} \in L^p(B_{2r}(x_0),\|V\|)$.
Suppose that, for some $S_0 \in \bG(N,m)$,
\begin{equation}\label{eqn:main-thm-hypotheses}
    \begin{split}
    \delta(x_0,S_0,r) &:= \left| \frac{\|V\|(B_{2r}(x_0))} {\omega_m(2r)^m}-1 \right| + r^{1-\frac{m}{p}} \|H_{\Psi}\|_{L^p(B_{2r}(x_0), \|V\|)} + r^{-m} \int_{B_{2r}(x_0)} \| T_x V - S_0 \|^2 \, d \|V\|
    \end{split}
\end{equation}
satisfies $\delta(x_0,S_0,r) \leq \ve_{\star}$.
Then, after translating $x_0$ to $0$ and applying an orthogonal change of coordinates taking $S_0$ to $\bR^m \times \{ 0 \}$, there exists a $f \in C^{1,\alpha_{\star}}(B^m_{\gamma_{\star} r} ; \bR^{N-m})$ with $f(0)=0$ such that
\[
V \mres B_{\gamma_{\star} r}
=  v(\operatorname{graph}f,1)\mres B_{\gamma_{\star} r}.
\]
Moreover, $f$ satisfies the scale-invariant estimate
\[
r^{-1}\, \|f\|_{L^{\infty}(B^m_{\gamma_{\star} r})} + \|Df\|_{L^{\infty}(B^m_{\gamma_{\star} r})}
+ r^{\alpha_{\star}}\, [Df]_{C^{0,\alpha_{\star}}(B^m_{\gamma_{\star} r})}
\leq \omega_{\star}(\delta(x_0,S_0,r)).
\]
In particular, $\textup{spt}\,\|V\|$ is an embedded $C^{1,\alpha_{\star}}$ submanifold and $V$ has multiplicity one in a
neighborhood of every point $x$ where $\delta(x,S,r) \leq \ve_{\star}$ holds for some $S \in \bG(N,m)$ and some scale $r$.
\end{theorem}

In particular, Theorem~\ref{thm:allard} applies to every positive anisotropy in a $C^2$ neighborhood of the area integrand, since every such anisotropy satisfies QEC~\cites{derosa-tione-regularity,atomic-condition}.
Thus, the classical Allard theorem is stable under genuinely anisotropic perturbations in every dimension and codimension.
See Remark~\ref{rmk:on-allard} for further related observations and consequences. 

Our second main result concerns the scope of the atomic conditions.  For
$T\in\bG(N,m)$, let $\xi_T$ be a Euclidean unit simple $m$-vector spanning
$T$, and define, following \cite{derosa-tione-regularity}*{\S 6},
\[
    \Psi_q(T):=\|\xi_T\|_{\ell^q}
      =\left(\sum_{|I|=m}|(\xi_T)_I|^q\right)^{1/q}.
\]
\begin{theorem}\label{thm:ell-p-on-G(N,m)}
For $2-\frac{7}{10\min\{m,N-m\}}\leq q\leq2$, the $\ell^q$ anisotropy satisfies USAC on $\bG(N,m)$.
\end{theorem}
In particular, $\Psi_q$ satisfies QEC and AC.
For $q\neq2$, the norm $\Psi_q$ is not induced by an inner product. 
We thus obtain explicit, non-ellipsoidal examples of USAC and AC integrands in every dimension and codimension, beyond the previously known perturbative description of the class. 
In Remark \ref{rmk:when-usac-fails} we complement this result by proving that for $2 \leq m \leq N-2$ and every $q>4$, the $\ell^q$ anisotropy does not satisfy the scalar atomic conditions. 
This answers in the negative a question posed in \cite{derosa-tione-regularity}*{Remark 6.4}; see also \cite{atomic-condition}*{Remark 8.13(c)} for a discussion on this question.

Finally, in Section~\ref{sec:axisymmetric}, we construct families of anisotropic integrands based on an axisymmetric ansatz that satisfy USAC. 
Our main construction is the following.
\begin{theorem}\label{thm:axisymmetric-usac}
Let $2\leq m<N$ and fix a unit vector $e\in\bR^N$. 
Let $h\in C^2(\bR^2\setminus\{0\}, \bR_+)$ be an absolute norm, meaning $h(x,y)=h(|x|,|y|)$, $h$ is $1$-homogeneous, and $h$ is convex.
Define
\[
    \Psi_h(T):=h\bigl(|Te|,|T^\perp e|\bigr),
    \qquad T\in\bG(N,m).
\]
Set $H(\theta):=h(\cos(\theta),\sin(\theta) )$ for $0\leq\theta\leq\frac{\pi}{2}$. 
Then, $\Psi_h$ satisfies USAC if and only if
\[
H(\theta)+H''(\theta)>0 \quad\text{for every }\theta\in[0,\tfrac{\pi}{2}].
\]
Furthermore, if this condition holds, then $\Psi_h$ supports a Michael-Simon inequality.
\end{theorem}
In Corollary~\ref{cor:ha-integrands}, we provide examples of such integrands that satisfy USAC, and whose distance to any ellipsoidal area integrand is lower bounded. 
As such, they satisfy the hypotheses of Theorem~\ref{thm:allard} while not being perturbations of the area functional under any change of coordinates.

\subsection{Ideas and strategy of the proofs}

The proof of Theorem~\ref{thm:allard} is
inspired by Allard's regularity scheme
\cites{allard-first-variation,allard-regularity} and by the anisotropic
excess-decay argument for Lipschitz graphs developed in~\cite{derosa-tione-regularity}; see also the discussion in
\cite{gmt-differential-inclusions}*{\S 8}.
To identify the importance of the quadratic exposed condition (QEC), we associate with $\Psi$ the non-parametric graph integrand
\[
F_\Psi(A) :=\Psi(\operatorname{graph}A) \sqrt{\det(I_m+A^tA)}.
\]
By~\cite{atomic-condition}, QEC provides a Caccioppoli inequality that bounds the tilt-excess by the height-excess and the $L^2$ curvature.
Moreover, QEC implies a quantitative quasiconvexity inequality for $F_\Psi$ and, in particular, a local uniform Legendre-Hadamard inequality for $D^2F_\Psi$. 
On the other hand, the Michael-Simon inequality and the smallness of the scale-invariant $L^p$ norm of $H_{\Psi}$ produces a uniform lower mass bound $\|V\|(B_s(x)) \geq cs^m$ at every point $x \in \textup{spt} \, \|V\|$ and at every admissible scale.
In Section~\ref{sec:L2-L-infty}, we use this property to convert the a priori $L^2$ height bound into an $L^{\infty}$ bound, replacing in arbitrary codimension the
codimension-one barrier used in~\cite{allard-regularity}*{\S\S 3.4-3.5}.

Once the support lies in a sufficiently thin cylinder, we apply Allard's Lipschitz approximation for general parametric integrands~\cite{allard-regularity}*{\S 2.6}.  
This later version of the approximation
theorem, unlike the affine approximation in~\cite{allard-first-variation}*{\S 6.2}, does not require an anisotropic monotonicity formula.
The resulting Lipschitz graph agrees with $V$ outside an exceptional set with mass bounded by $Cr^m(E+\Lambda^2)$, where $E$ is the normalized squared height-excess and $\Lambda$ is the scale-invariant $L^p$ curvature quantity.
In particular, the non-graphical part becomes negligible after normalizing by $E_j^{\frac{1}{2}}$ and performing a blowup as $E_j \downarrow 0$.
Thus, after suitable translation and rescaling, the Lipschitz approximations converge to a weak solution $v$ of a constant-coefficient system, which is strongly elliptic by virtue of the Legendre-Hadamard condition for $D^2 F_{\Psi}$.
The graph of the affine polynomial of $v$ then provides the affine plane to iterate the decay of the height-excess at the next scale.

In Section~\ref{section:initiate-the-procedure}, we carry out the iteration following~\cite{derosa-tione-regularity}.
A key technical step in the iteration is to guarantee that the varifold mass remains in a fixed multiplicity-one regime across scales, which we obtain in Lemma~\ref{lemma:persistence-of-mass-one} as an application of the Michael-Simon inequality and the compactness theory developed in~\cite{dephilippis-pigati}.
Unlike in~\cite{derosa-tione-regularity}, the varifold is not a priori graphical, so special care is required to ensure that the varifold agrees with a multiplicity-one graph over a uniform neighborhood of $x_0$; this argument is carried out in~\hyperref[step-4-in-Allard]{Step 4} of the proof of Theorem~\ref{thm:allard}.

We next turn to the examples. 
To prove the USAC property of $\ell^q$, we work directly in Pl\"ucker coordinates for Theorem~\ref{thm:ell-p-on-G(N,m)}.
For the $\ell^q$ norm on $\bigwedge^m\bR^N$, the normalized differential determines an oblique projection $\Pi_T$, and the normalized USAC pairing reduces to the scalar inequality 
\[
K_q (T,S) := m - \textup{tr}(\Pi_T \Pi_S) \geq c_{N,m,q} ( m - \textup{tr}(TS)), \qquad m - \textup{tr}(TS) = \tfrac{1}{2} \|T-S\|^2.
\]
In Proposition~\ref{prop:kernel}, we establish this inequality for $2 - \frac{7}{10 \min \{ m,N-m \}} \leq q \leq 2$.

Finally, in Section~\ref{sec:axisymmetric}, we exhibit new $C^2$ anisotropic integrands satisfying USAC as well as a Michael-Simon inequality.
The construction of Theorem~\ref{thm:axisymmetric-usac} is obtained via an axisymmetric ansatz, which reduces USAC to a second-order differential inequality.
In Corollary~\ref{cor:ha-integrands}, we provide explicit examples that remain a definite distance from every ellipsoidal integrand.
These provide the first non-perturbative examples of USAC $C^2$ anisotropies.

\smallskip \noindent \textbf{Acknowledgments.} 
We are thankful to Guido De Philippis, Alessandro Pigati, Simon Brendle, and Toby Colding for helpful conversations.

ADR was funded by the European Union: the European Research Council (ERC), through StG ``ANGEVA'', project number: 101076411. Views and opinions expressed are however those of the authors only and do not necessarily reflect those of the European Union or the European Research Council. Neither the European Union nor the granting authority can be held responsible for them.
BF recognizes support from a Simons Dissertation Fellowship and a MathWorks Fellowship.
RT is supported by a Simons Dissertation Fellowship, an A.G.~Leventis Foundation Scholarship, and an Onassis Foundation Scholarship.

\smallskip \noindent \textbf{AI usage statement.}
We were assisted by AI in the proof of Proposition~\ref{prop:kernel}, which is used for Theorem~\ref{thm:ell-p-on-G(N,m)}.
We used ChatGPT 6 to find a range of exponents $q \in [ q_{N,m}, 2]$ for which Proposition~\ref{prop:kernel}, and hence USAC, holds for the $\ell^q$ anisotropy on $\bG(N,m)$.
The result was obtained after several interactions; we then reworked the argument and included several intermediate lemmas.

In addition, we used ChatGPT 6 to proofread the document after completing a preliminary version.
This did not alter the mathematical content, but it helped clarify the structure and presentation of Section~\ref{section:initiate-the-procedure}.
It also helped simplify the computations of Theorem~\ref{thm:axisymmetric-usac}.

The manuscript contains no AI-generated text.

\section{Preliminaries}\label{sec:preliminaries}

We recall some important properties of anisotropic elliptic integrands, following~\cites{derosa-ghiraldin , derosa-tione-regularity }.
Let $\Psi: \bG(N,m) \to (0,+\infty)$ be $C^1$.
We denote by $A^t$ the transpose of an endomorphism $A$ and identify $\bG(N,m)$ with the set of rank-$m$ orthogonal projections in $\bR^N$, namely 
\[
\bG(N,m) = \{ T \in \bR^{N \times N} : T = T^t, \; T^2 = T, \; \on{tr} T = m \}.
\]
Given an $m$-plane $T \in \bG(N,m)$, we define a tensor field $B_{\Psi}(T)$ by
\begin{equation}\label{eqn:BPsi(T)}
B_{\Psi}(T) = \Psi(T) T + T^{\perp} \, d \Psi(T) \, T, \qquad \text{where } \; d \Psi(T) := D \Psi(T) + D \Psi(T)^t.
\end{equation}
We observe that $B_{\Psi}(T) T = B_{\Psi}(T)$ and $\sup_{T \in \bG(N,m)} \| B_{\Psi}(T) - B_{\tilde{\Psi}}(T) \| \leq C_N \, \| \Psi - \tilde{\Psi} \|_{C^1}$ for any $\Psi, \tilde{\Psi}$.
In particular, a uniform $C^1$ bound on $\Psi$ implies a uniform bound on $B_{\Psi}$.
Moreover, 
\begin{equation}\label{eqn:B-psi-identity}
    \la B_{\Psi}(T) , L \rg = \Psi(T) \la T, L \rg + \la D \Psi(T), T^{\perp} LT + (T^{\perp} LT)^t \rg, \qquad \forall \; L \in \bR^{N \times N},
\end{equation}
where $D \Psi(T)$ is the differential of $\Psi$, extended to a $C^1$ function in a neighborhood of $\bG(N,m) \subset \bR^{N \times N}$.
We also use the dual integrand $\Psi^*(P) := \Psi(P^{\perp})$ for $P \in \bG(N,N-m)$, given by
\[
B_{\Psi^*} (S^{\perp}) = \Psi(S) S^{\perp} - S \, d \Psi(S) \, S^{\perp}, \qquad B_{\Psi}(S)^t B_{\Psi^*}(S^{\perp}) = 0, \qquad \text{for } \; S \in \bG(N,m).
\]
We will also use repeatedly the fact that $\Psi \in C^2(\bG(N,m))$ implies that the association $T \mapsto B_{\Psi}(T)$ is Lipschitz.
We write $c_{\Psi}>0$ for a sufficiently small positive number and $C_{\Psi} < + \infty$ for a sufficiently large positive number, which change from one step to another and depend only on $N,\Psi$.

Let $V = v(M,\theta)$ be a rectifiable $m$-varifold in an open set $\Omega \subset \bR^N$, namely a Radon measure on $\Omega \times \bG(N,m)$.
We write $\|V\|$ for the weight measure, which satisfies $\|V\|(A) := V(A \times \bG(N,m))$ for every Borel set $A \subseteq \Omega$.
If $\Theta^m(\|V\|,x_0) \geq \theta_0$ a.e., then we can choose the rectifiable set $M$ to satisfy $\theta \geq \theta_0$ $\cH^m$-a.e., and hence $\cH^m(M) \leq \theta_0^{-1} \|V\|(\Omega)$.

The anisotropic first variation of $V$ with respect to the flow generated by a Lipschitz vector field $g \in \textup{Lip}_c(\Omega;\bR^N)$ is computed in~\cite{derosa-ghiraldin}*{Lemma A.2} as
\begin{equation}\label{eqn:first-variation}
[\delta_{\Psi} V] (g) = \int \la B_{\Psi}(T), D g(x) \rg \, dV(x,T) =
\int_{M} \la B_{\Psi} (T_x M), Dg(x) \rg \, \theta \, d \cH^m(x).
\end{equation}
Given an open set $U$, we say that $V = v(M,\theta)$ has $\Psi$-mean curvature in $L^p_{\textup{loc}}$ if there exists a map $H_{\Psi} \in L^p_{\textup{loc}}( M \cap U, \bR^N ; \theta\cH^m \mres M)$ with the property that
\[
[\delta_{\Psi} V](g) = \int_{\Omega \times \bG(N,m)} \la B_{\Psi}(T), Dg \rg \, dV(x,T) = - \int_{\Omega} \la H_{\Psi}(x), g(x) \rg \, d \|V\|(x)
\]
for $g \in \textup{Lip}_c(U,\bR^N)$.
We say that $V$ is $\Psi$-stationary if this expression is zero for all $g$.

\begin{definition}\label{def:all-the-usac}
    We say that $\Psi$ satisfies the scalar atomic condition SAC if $\la B_{\Psi}(T), B_{\Psi^*}(S^{\perp}) \rg > 0$ for every $T \neq S \in \bG(N,m)$, and the uniform scalar atomic condition USAC if
    \begin{equation}\label{eqn:usac}
        \la B_{\Psi}(T), B_{\Psi^*}(S^{\perp}) \rg \geq c_{\Psi} \, \|T - S\|^2, \qquad \text{for every } \; S,T \in \bG(N,m)
    \end{equation}
    for some $c_{\Psi}>0$.
We say that $\Psi$ satisfies the quadratic exposed condition QEC (cf.~\cite{atomic-condition}*{Definition 6.2}) if there exist a constant $\lambda_{\Psi}>0$ and a map $\cN:\bG(N,m)\to\bR^{N\times N}$ with $\|\cN(T)\|=1$ for every $T$, such that, for every $S,T\in\bG(N,m)$, we have
    \[
    \cN(T)\,T=0 \qquad  \text{and}\qquad \la B_{\Psi}(S),\cN(T)\rg \geq \lambda_{\Psi}\|S-T\|^2.
    \]
Finally, $\Psi$ satisfies the atomic condition AC if every probability
measure $\lambda$ on $\bG(N,m)$ satisfies
$\textup{rank} \int_{\bG(N,m)} B_{\Psi}(T) \, d\lambda(T) \geq m$, with equality
if and only if $\lambda = \delta_{T_0}$ for some $T_0 \in \bG(N,m)$.
\end{definition}
For a $C^2$ anisotropy satisfying QEC, \cite{atomic-condition}*{Lemma 6.4 and Theorem 6.8} allow us, after possibly decreasing $\lambda_{\Psi}$, to choose $\cN \in C^{0,1}$ satisfying
\[
\cN(T)B_{\Psi}(T)^t = B_{\Psi}(T)^t \cN(T) = 0 \qquad \text{for every  } T \in \bG(N,m).
\]
We fix such a Lipschitz choice whenever $\Psi$ is of class $C^2$.

\subsection{QEC and quasiconvexity}

By the results of De Rosa-Tione~\cite{derosa-tione-regularity}, every $\Psi$ in a $C^2$ neighborhood of the area integrand satisfies USAC; the class of USAC integrands is $C^2$-open; and USAC implies the atomic condition of~\cite{derosa-ghiraldin} as well as quasiconvexity.
We show here that QEC also implies the quantitative quasiconvexity of~\cite{derosa-tione-regularity}*{Proposition 3.13}; this property is also described in~\cite{atomic-condition}*{Remark 9.16}.
As remarked therein, this step allows us to apply the regularity theory of~\cite{derosa-tione-regularity} for merely QEC, rather than USAC, integrands.
To formulate this result, we first recall the passage from varifolds to the non-parametric setting as in~\cite{derosa-tione-regularity}*{\S 2.3} via the associated non-parametric graph integrand
\begin{equation}\label{graphical}
    F_{\Psi}(A) := \Psi ( \textup{graph} \, A) \sqrt{\det (I_m + A^t A)}, \qquad \cA(A) := \sqrt{\det(I_m + A^t A)}.
\end{equation}
\begin{proposition}\label{prop:quasiconvexity}
Let $\Psi \in C^2(\bG(N,m), (0,\infty))$ be a QEC anisotropy with associated map $\cN \in C^{0,1}$.
Then, there exists a constant $\alpha = \alpha(\Psi,\cN)>0$ such that
\begin{equation}\label{eqn:quasiconvex}
\int_{\Omega} ( F_{\Psi}(A+D \varphi)) - F_{\Psi}(A))\, dx \geq \alpha \int_{\Omega} ( \cA(A+D \varphi) - \cA(A)) \, dx
\end{equation}
for every bounded open set $\Omega \subset \bR^m$, every $A \in \bR^{(N-m) \times m}$, and every $\varphi \in C_c^1(\Omega, \bR^{N-m})$.

Moreover, for every $R>0$, there exists a $c_{\Psi,\cN,R}>0$ such that 
\begin{equation}\label{eqn:local-LH}
D^2 F_{\Psi}(A)[L,L] \geq c_{\Psi,\cN,R} \|L\|^2
\end{equation}
whenever $\|A\| \leq R$ and $L$ is a rank-one matrix.
In particular, $F_{\Psi}$ satisfies the local uniform Legendre-Hadamard condition.
\end{proposition}
\begin{proof}
In view of the observation following Definition~\ref{def:all-the-usac}, the exposing vector field $N$ provided by the QEC condition may be taken to be Lipschitz.
Thus, we can apply~\cite{atomic-condition}*{Theorem 7.4} to choose $\alpha>0$ so that $\Psi_\alpha=\Psi-\alpha$ is positive and satisfies QEC, hence AC. 
Consequently, ~\cite{derosa-lei-young}*{Theorem 1.16} produces a convex, positively one-homogeneous function $G_{\alpha}$ on $\bigwedge^m \mathbb{R}^N$ whose restriction to unit simple $m$-vectors agrees with the even lift of $\Psi_{\alpha}$.
For $A \in \mathbb{R}^{(N-m) \times m}$, we define the operator
\[
Z(A) := \bigwedge_{i=1}^m (e_i + A e_i), \qquad \text{so that } \; G_{\alpha} (Z(A)) = F_{\Psi}(A) - \alpha \, \cA(A).
\]
By construction, the components of $Z(A)$ are minors of $A$, including the constant minor.
Hence,
\[
\int_{\Omega} Z(A+D \varphi) \, dx = |\Omega| \, Z(A), \qquad \text{for every } \; \varphi \in C_c^1(\Omega, \bR^{N-m}).
\]
Applying Jensen's inequality, we therefore obtain
\begin{align*}
    \int_{\Omega} (F_{\Psi}(A+D \varphi) - \alpha \, \cA(A+D \varphi)) \, dx &= \int_{\Omega} G_{\alpha}(Z(A+D\varphi)) \, dx \\
    &\geq |\Omega| \, G_{\alpha}(Z(A)) = |\Omega| (F_{\Psi}(A) - \alpha \, \cA(A)).
\end{align*}
This establishes the property~\eqref{eqn:quasiconvex}.
In particular, $F_{\Psi} - \alpha \cA$ is quasiconvex, and hence rank-one convex.
Since it is $C^2$, its Hessian is non-negative on rank-one matrices.

In the next step, we use these properties to prove~\eqref{eqn:local-LH}. 
Let $L = u \otimes v$ be a rank-one matrix.
Because the function $Z(A+tL)$ is affine in $t$, the convexity of $G_{\alpha}$ and the above discussion yield
\[
D^2 F_{\Psi}(A) [L,L] \geq \alpha \, D^2 \cA(A) [L,L].
\]
On the other hand, we compute that
\begin{align*}
    D^2 \cA(A)[L,L] = \cA(A) \la (I_{N-m} + AA^t)^{-1} u, u \rg \cdot \la (I_m + A^tA)^{-1} v,v \rg.
\end{align*}
Combining this properties, for $\|A\| \leq R$, we obtain
\[
D^2 F_{\Psi}(A) [L,L] \geq \alpha (1+R^2)^{-2} \| L\|^2.
\]
This establishes~\eqref{eqn:local-LH} and completes the proof of our assertion.
\end{proof}

Thanks to~\cite{atomic-condition}*{Proposition 6.12}, the Caccioppoli inequality of De Rosa-Tione \cite{derosa-tione-regularity}*{Proposition 4.3} extends to QEC integrands in the following form.
\begin{proposition}\label{prop:caccioppoli}
Let $\Psi \in C^2(\bG(N,m))$ be an integrand satisfying QEC and consider a rectifiable $m$-varifold $V$ in $B_{2r}(x_0)$ with $\Psi$-mean curvature $H_{\Psi} \in L^2( \|V\|)$.
Then, for every $S \in \bG(N,m)$ and every $a \in \bR^N$, it holds that
\[
c_{\Psi}\int_{B_r(x_0)} \|T_x V - S \|^2 \, d \|V\| \leq r^{-2} \int_{B_{2r}(x_0)} \textup{dist}(x,a+S)^2 \, d \|V\| + r^2 \int_{B_{2r}(x_0)} |H_{\Psi}|^2 \, d \|V\| 
\]
for $c_{\Psi}>0$ depending on $N,m$ and $\Psi$ and its QEC constants.
Moreover, for $M \in (0,\infty)$, there is a $c_{\Psi,M}>0$ such that if $H_{\Psi} \in L^p(\|V\|)$ for $p> m$ and $p \geq 2$ and $\|V\|(B_{2r}(x_0)) \leq Mr^m$, then 
\begin{align*}
& c_{\Psi,M} \, r^{-m} \int_{B_r(x_0)} \|T_x V - S\|^2 \, d\|V\| \\
&\quad\leq r^{-m-2} \int_{B_{2r}(x_0)} \textup{dist}(x,a+S)^2 \, d\|V\| + \Bigl( r^{1- \frac{m}{p}} \|H_{\Psi} \|_{L^p(B_{2r}(x_0), \|V\|)} \Bigr)^2 .
\end{align*}
\end{proposition}

\subsection{An \texorpdfstring{$L^2$}{L2}-to-\texorpdfstring{$L^\infty$}{Linfinity} height bound}\label{sec:L2-L-infty}
In what follows, unless stated otherwise, we assume that $\Psi \in C^2(\bG(N,m), (0,\infty))$ satisfies QEC.
We now define the key quantities used in the iteration scheme for the proof of Theorem~\ref{thm:allard}.
\begin{definition}\label{def:height-excess}
We consider an integral $m$-varifold $V$ defined in a neighborhood of $\bar{B}_{2r}(x_0)$.
Given an affine $m$-plane $A = a+S$, where $a \in \bR^N$ and $S \in \bG(N,m)$, we define
\begin{equation}\label{eqn:EVSx0r-LambdaVx0r}
    \begin{split}
    E(V,A,x_0,r) &:= r^{-m-2} \int_{B_r(x_0)} \textup{dist}(x,A)^2 \, d\|V\|(x) , \\
    \Lambda(V,x_0,r) &:= r^{1 - \frac{m}{p}} \| H_{\Psi} \|_{L^p (B_r(x_0), \|V\|)}.
    \end{split}
\end{equation}
\end{definition}

A key step in Allard's regularity argument is an $L^2$-to-$L^{\infty}$ estimate for the height, cf.~\cite{allard-regularity}*{\S\S3.4-3.5}. 
For an $m$-varifold $V$ satisfying a mass bound $\|V\|(B_r(x_0)) \leq M r^m$, we obtain
\begin{equation}\label{eqn:holder-bound-for-H-Psi}
r^{2-m} \int_{B_r} |H_{\Psi}|^2 \, d\|V\| \leq C(M) \, \Lambda(V,x_0,r)^2
\end{equation}
for $p \geq 2$, using H\"older's inequality.

The anisotropic Michael-Simon inequality provides the required lower mass ratios in arbitrary codimension and yields the following simple estimate.
\begin{lemma}\label{lemma:density-and-height-bound}
There are constants $\ve_0(N,m,p,\Psi), c_0(N,m, \Psi)>0$ such that the following properties hold for every integral $m$-varifold $V$ in an open neighborhood of $\bar{B}_{2r}(x_0)$ satisfying $\Lambda(V,x_0,2r) \leq \ve_0$.
\begin{enumerate}[(i)]
    \item For every $x \in \textup{spt} \, \|V\| \cap B_{2r}(x_0)$ and $0<s\leq \min \{ \frac{r}{2}, r- \frac{|x-x_0|}{2}\}$,
    we have $\|V\|(B_s(x)) \geq c_0 s^m$.
    \item For every
$\lambda \in (0,2)$, there is a constant $C_{\lambda} = C(N,m,\Psi,\lambda)< \infty$ with the following property.
    For every affine $m$-plane $A \subset \bR^N$ such that $E(V,A,x_0,2r) \leq \ve_0$, we have
    \[
    r^{-1} \sup_{x \in \textup{spt} \, \|V\| \cap B_{\lambda r}(x_0)} \textup{dist}(x,A)
    \leq C(N,m,\Psi,\lambda) \, E(V,A,x_0,2r)^{\frac{1}{m+2}}.
    \]
\end{enumerate}
\end{lemma}
\begin{proof}
Because $\Psi \in C^2(\bG(N,m),(0,\infty))$ is QEC, it satisfies the (AC1) condition, and hence supports a Michael-Simon inequality, by~\cite{galf}*{Theorem 1.6}. 
The lower mass bound $(i)$ is a standard consequence of the functional Michael-Simon inequality, together with Simon's
ODE lemma, applied to $M(s) := \|V\|(B_s(x))$; see~\cite{dephilippis-pigati}*{Corollary 1.12$(i)$}.

For $(ii)$, fix $\lambda \in (0,2)$ and $x \in \textup{spt} \, \|V\| \cap B_{\lambda r}(x_0)$, and set $d := \textup{dist}(x,A)$ and $s := \frac{1}{2} \min \{ d, \frac{2-\lambda}{2}r \}$.
Then, $\textup{dist}(y,A) \geq d- s \geq \frac{d}{2}$ on the ball $B_s(x) \subset B_{2r}(x_0)$, so the lower bound $(i)$ yields
\[
\int_{B_{2r}(x_0)} \textup{dist}(y,A)^2 \, d\|V\|  \geq \tfrac{1}{4} d^2 \|V\|(B_s(x)) \geq c'_0 d^2 s^m. 
\]
Let $E := E(V, A, x_0, 2r)$ for brevity, so this bound implies $E r^{m+2} \geq c''_0 d^2 s^m$.
If $d \leq \frac{2-\lambda}{2}r$, then $s = \frac{d}{2}$ and $\frac{d}{r} \leq C E^{\frac{1}{m+2}}$ proves the claim.
If $d>\frac{2-\lambda}{2}r$, then $s = \frac{2-\lambda}{4} r$ and $\frac{d}{r} \leq CE^{\frac{1}{2}} \leq C E^{\frac{1}{m+2}}$.
Combining these cases and taking the supremum over $x \in \textup{spt} \, \|V\| \cap B_{\lambda r}(x_0)$ proves our assertion.
\end{proof}

\subsection{Lipschitz approximation} 
We use Allard's Lipschitz approximation for general parametric integrands~\cite{allard-regularity}*{\S 2.6}.
For background on the isotropic Lipschitz approximation argument, see~\cite{simon-gmt}*{Chapter 5, \S 2}.
For varifolds with $\Psi$-mean curvature in $L^p$, we can replace Allard's bound \S 2.6(5) for the Lebesgue measure of the non-graphical set by an error term that has the same quadratic scaling as the height-excess.
Given an $m$-plane $S \in \bG(N,m)$, we denote by $\pi_S$ and $S^{\perp}$ the orthogonal projections onto $S$ and $S^{\perp}$, respectively.
\begin{equation}\label{eqn:Zr-domain}
Z_r(x_0,S) := \{ x \in \bR^N : |\pi_S (x-x_0)| < \tfrac{5}{2}r , \; |S^{\perp} (x - x_0)| < r \}, \qquad Z_r(x_0,S) \Subset B_{3r}(x_0).
\end{equation}
Throughout, we will assume the following bound on the projected mass of the $m$-varifold $V$:
\begin{equation}\label{eqn:projected-mass-bound}
    \tfrac{3}{4} \omega_m (2r)^m < \|V\| \bigl( Z_r \cap \pi_S^{-1}(B^S_{2r} (\pi_S x_0)) \bigr) = \| V \mres Z_r \| \bigl( \pi_S^{-1}(B^S_{2r}(\pi_S x_0)) \bigr) < \tfrac{5}{4} \omega_m(2r)^m
\end{equation}
where $B^S_{2r}(\pi_S x_0) = \{ z \in S : |z - \pi_S x_0| < 2r \}$ is a $2r$-ball on the $m$-dimensional plane $S$.
After rotating by an orthogonal map $R$ and translating, we can make $S = \bR^m \times \{ 0 \}$ and $x_0 = 0$; we then write $B^m_{2r} := B^S_{2r}(\pi_S x_0)$ for brevity.
Whenever making such an orthogonal change of coordinates, we also replace the varifold $V$, the anisotropy $\Psi$, and its stress tensor $B_{\Psi}$ by the rotated varifold $V$ and the pullback $\Psi_R$ with stress tensor $B_{\Psi_R}$, where
\begin{equation}\label{eqn:replace-Psi-R}
\Psi_R(T) := \Psi(R^t TR), \qquad B_{\Psi_R}(T) := R \, B_{\Psi}(R^t TR) \, R^t, \qquad \tilde{V} := R_{\#} V.
\end{equation}
We observe that $H_{\Psi_R}(Rx) = R \, H_{\Psi}(x)$, so all the relevant $L^p$ norms are preserved.

\begin{proposition}\label{prop:lipschitz-approximation}
Fix $\sigma \in (0,1)$, $L>0$, and $M<+\infty$.
There exists an $\ve_0(\Psi,m,N,\sigma,L,M)>0$ with the following property.
After a rotation and a translation, write $\bR^N = \bR^m \times \bR^{N-m}$, $S = \bR^m \times \{ 0 \}$, $x_0=0$, and $Z_r := Z_r(x_0,S)$.
Let $V$ be an integral $m$-varifold in an open neighborhood of $B_{3r}$ with $\Psi$-mean curvature $H_{\Psi} \in L^2( B_{3r} , \|V\|)$ and $\|V\|(B_{3r}) \leq M r^m$.
Suppose that~\eqref{eqn:projected-mass-bound} holds, that $\sup_{ \textup{spt} \, \|V\| \cap B_{11 r/4}(x_0) } |S^{\perp}x| \leq \ve_0 r$, and
\begin{align*}
    & r^{-m} \int_{Z_r \times \bG(N,m)} \| T - S\|^2 \, dV(x,T) + r^{2-m} \int_{Z_r} |H_{\Psi}|^2 \, d\|V\| \leq \ve_0
\end{align*}
where $Z_r$ is defined in~\eqref{eqn:Zr-domain}.
Then, there exists a Borel set $K \subset B^m_{\sigma r}$ and an $L$-Lipschitz map $f: B^m_{\sigma r} \to \bR^{N-m}$ such that $V$ agrees in $B_{2r}$ with the multiplicity-one graph of $f$ over $K$, namely
\[
V \mres ( B_{2r} \cap \pi_S^{-1}(K)) = v ( \{ (z,f(z)) : z \in K \}, 1 ),
\]
and such that
\begin{align*}
    & \int_{B^m_{\sigma r}} |Df|^2 + \cL^m ( B^m_{\sigma r} \setminus K) + \|V\| \left( \bigl(B_{2r}\cap\pi_S^{-1}(B_{\sigma r}^m)\bigr) \setminus \{(z,f(z)):z\in K\} \right) \\ &\qquad \leq C \, \Bigl[ \int_{Z_r \times \bG(N,m)} \|T-S\|^2 \, d V(x,T) + r^2 \int_{Z_r} |H_{\Psi}|^2 \, d \|V\| \Bigr].
\end{align*}
\end{proposition}
\begin{proof}
We retain Allard's argument in~\cite{allard-regularity}*{\S 2.6}, modifying only its final covering argument to exploit the assumption that $\delta_{\Psi} V$ is absolutely continuous with $L^2$ density.
We regard $W := V \mres Z_r$ as a varifold in the cylinder $\cU_r := \{ x \in \bR^N : |\pi_S x| < \frac{5}{2} r\}$, extending by zero outside $Z_r$, where
\begin{equation}\label{eqn:Zr-containment}
Z_r := \{ x \in \bR^N : |\pi_S x| < \tfrac{5}{2} r, \; |S^{\perp}x| < r \}, \qquad
\bar{Z}_r \subset \bar{B}_{\sqrt{29} r/2} \Subset B_{11r/4}  \Subset B_{3r}.
\end{equation}
The assumption $\sup_{\textup{spt} \, \|V\| \cap B_{11r/4}} |S^{\perp} x| \leq \ve_0 r$ implies that $\textup{spt} \, \|W\| \subset \{ |S^{\perp} x| \leq \ve_0 r\}$ in $\cU_r$.
We can choose $\chi \in C_c^{\infty}(S^{\perp};[0,1])$ with $\chi=1$ on $B^{S^{\perp}}_{r/2}$ and $\chi=0$ outside $B^{S^{\perp}}_{3r/4}$.
For any $g \in C^1_c(\cU_r; \bR^N)$, the field $\tilde{g}(x) := \chi (S^{\perp} x) g(x)$ extends by zero to a compactly supported $C^1$ vector field in $B_{3r}$.
On the portion of $\textup{spt} \, \|V\|$ where $\tilde{g}$ or its derivative can contribute, the height hypothesis implies that
\[
\chi(S^{\perp} x) = 1, \qquad D(\chi \circ S^{\perp})(x) =0,
\]
because $\{ 0 < \chi < 1 \} \cap \cU_r$, lies in $Z_r \subset B_{11r/4}$ and avoids $\textup{spt} \, \|V\|$.
Therefore,
\begin{align*}
    [\delta_{\Psi} W](g) &= \int \la B_{\Psi}(T), Dg(x) \rg \, dW(x,T) = \int \la B_{\Psi}(T), D \tilde{g}(x) \rg \, d V(x,T) \\
    &= - \int \la H_{\Psi}, \tilde{g} \rg \, d \|V\| = - \int \la H_{\Psi}, g \rg \, d \|W\|, \qquad \text{for every } \; g \in \textup{Lip}_c(\cU_r ; \bR^N).
\end{align*}
Thus, as in~\cite{allard-regularity}*{\S 2.2}, projected test fields $g$ with compact horizontal support are admissible, the cutoff produces no additional first variation term, hence
$|\delta_{\Psi} W| = |H_{\Psi}| \, \|W\|$ as measures in $\cU_r$.
We will apply Allard's construction to $W$; we first check that the hypotheses entering the construction are satisfied after decreasing $\ve_0$.
Since $\|W\|(Z_r) \leq M r^m$, the Cauchy-Schwarz inequality gives
$$
r^{-m} \int_{Z_r \times \bG(N,m)} \|T-S\| \, dW \leq M^{\frac{1}{2}} \Bigl( r^{-m} \int_{Z_r \times \bG(N,m)} \|T-S\|^2 \, d V \Bigr)^{\frac{1}{2}},
$$
and
\begin{align*}
r^{1-m} |\delta_{\Psi} W|(Z_r) &= r^{1-m} \int_{Z_r} |H_{\Psi}| \, d \|W\| \leq r^{1-m} \, \|W\|(Z_r)^{\frac{1}{2}} \Bigl( \int_{Z_r} |H_{\Psi}|^2 \, d\|W\| \Bigr)^{\frac{1}{2}} \\
&\leq M^{\frac{1}{2}} \Bigl( r^{2-m} \int_{Z_r} |H_{\Psi}|^2 \, d \|V\| \Bigr)^{\frac{1}{2}}.
\end{align*}
Thus, both the normalized tilt and the normalized first variation appearing in Allard's hypotheses can be made arbitrarily small by choosing $\ve_0<1$ small.
By~\eqref{eqn:projected-mass-bound}, we have $\frac{\|W\|(\pi_S^{-1}(B^m_{2r}))}{\omega_m(2r)^m} \in ( \frac{3}{4}, \frac{5}{4})$.
Therefore, the integer selected in Allard's construction is $1$, and we can apply the scale-invariant autonomous version of~\cite{allard-regularity}*{\S 2.6} to obtain a Borel set $K_0 \subset B^m_{2r}$ and an $L$-Lipschitz map $f_0: B^m_{2r} \to \bR^{N-m}$ such that $W$ agrees over $K_0$ with the multiplicity-one graph of $f_0$.
In particular, $\cL^m$-a.e. $z \in K_0$ has $T_{(z,f_0(z))}W = T_{(z,f_0(z))} \textup{graph}(f_0)$.
The Lipschitz estimate follows from~\cite{allard-regularity}*{\S 2.6} by considering the graph over $K_0$ and using Kirszbraun's extension theorem~\cite{maggi}*{Theorem 7.2}.

It remains to improve the estimate for the complement of $K_0$.
As above, we use $W = V \mres Z_r$ in what follows.
For every $z \in B^m_{2r} \setminus K_0$, Allard's construction provides a critical radius $\rho(z) \leq Cr$ so that either the tilt-excess bound fails, or the first-variation bound fails; these are alternatives (19) and (20) in Allard's notation.
At the same radius,~\cite{allard-regularity} produces a local mass upper bound.
Concretely, up to fixed constants, one of the following alternatives holds on $Q_z := \pi_S^{-1} (\bar{B}^S_{\rho(z)}(z)) \cap \cU_r$:
\[
(i):\quad \gamma\rho(z)^m \leq \int_{Q_z\times\bG(N,m)}\|T-S\|\,dW(x,T),
 \qquad\text{or}\qquad (ii):\quad \gamma\rho(z)^m \leq\rho(z)|\delta_\Psi W|(Q_z),
\]
where $\gamma>0$ is a fixed constant depending only on the parameters chosen in the Lipschitz approximation theorem.
At the same radius, Allard's local integrality estimate (21) gives
\begin{equation}\label{eqn:allard-mass-bound}
\|W\|(Q_z) + \cL^m ( B_{\rho(z)}(z)) \leq C \rho(z)^m.
\end{equation}
In the first case $(i)$, the Cauchy-Schwarz inequality and the bound~\eqref{eqn:allard-mass-bound} give
\begin{align*}
    \gamma^2 \rho(z)^{2m} &\leq \|W\| (Q_z) \int_{Q_z} \|T-S\|^2 \, d W \leq C \rho(z)^m \int_{Q_z} \|T-S\|^2 \, dW.
\end{align*}
Rearranging terms, we deduce that
    \begin{equation}
        \rho(z)^m \leq C' \int_{Q_z} \|T-S\|^2 \, dW. \label{eqn:case-I-bound}
    \end{equation}
In the second case $(ii)$, we combine the equality $|\delta_{\Psi}W| = |H_{\Psi}| \, \|W\|$ with the Cauchy-Schwarz inequality and the local mass bound~\eqref{eqn:allard-mass-bound} to obtain
\begin{align*}
    \gamma^2 \rho(z)^{2m} &\leq \rho(z)^2 \Bigl( \int_{Q_z} |H_{\Psi}| \, d \|W\| \Bigr)^2 \leq \rho(z)^2 \|W\| (Q_z) \int_{Q_z} |H_{\Psi}|^2 \, d \|W\| \\
    &\leq C \rho(z)^{m+2} \int_{Q_z} |H_{\Psi}|^2 \, d \|W\|.
\end{align*}
Combining this inequality with the bound $\rho(z) \leq Cr$, we arrive at
\begin{equation}\label{eqn:rhoMHpsiBound}
\rho(z)^m \leq C' r^2 \int_{Q_z} |H_{\Psi}|^2 \, d\|W\|.
\end{equation}
We may therefore combine the bounds~\eqref{eqn:allard-mass-bound}, \eqref{eqn:case-I-bound}, and~\eqref{eqn:rhoMHpsiBound} in all cases to estimate
\begin{align*}
\bigl( (\pi_S)_{\#} \|W\| + \cL^m \bigr) ( \bar{B}_{\rho(z)}(z)) \leq C \, \Bigl[ \int_{Q_z} \|T-S\|^2 \, d \|W\| + r^2 \int_{Q_z} |H_{\Psi}|^2 \, d \|W\| \Bigr].
\end{align*}
As in~\cite{simon-gmt}*{Ch.~5, Lemma 2.6}, we can apply the Besicovitch covering theorem for the family $\{ B_{\rho(z)}(z) : z \in B_{2r}^m \setminus K_0 \}$ and use $V \mres Z_r = W$ and $Q_z = \pi_S^{-1} (\bar{B}^S_{\rho(z)}(z)) \cap \cU_r$ to bound
\begin{equation}\label{eqn:after-besicovitch}
    \begin{split}
    (\cL^m + (\pi_S)_{\#} \|W\|) \, (B_{2r}^m\setminus K_0)  
    &\leq C \Bigl[ \int_{Z_r \times \bG(N,m)} \|T-S\|^2 \, dV + r^2 \int_{Z_r} |H_\Psi|^2 \, d\|V\| \Bigr].
    \end{split}
\end{equation}
We now set $K = K_0 \cap B_{\sigma r}^m$ and $f = f_0|_{B_{\sigma r}^m}$, which defines an $L$-Lipschitz function with $B^m_{\sigma r} \setminus K = B^m_{\sigma r} \setminus K_0$.
Over $K$, the varifold $W$ agrees with the multiplicity-one graph of $f$.
For $\ve_0<1$, the height assumption gives $|S^{\perp}x|<r$ for $x \in \textup{spt} \, \|V\| \cap B_{2r}$, hence $V \mres B_{2r} = W \mres B_{2r}$ and
\begin{equation}\label{eqn:mult-1-graph}
V \mres (B_{2r} \cap \pi_S^{-1}(K)) = v ( \{(z,f(z)) : z \in K \}, 1).
\end{equation}
Therefore, the $\|V\|$-mass of the region $(B_{2r} \cap \pi_S^{-1}(B^m_{\sigma r}) ) \setminus \{ ( z,f(z)) : z \in K \}$ is bounded by $(\pi_S)_{\#} \|W\|(B^m_{\sigma r} \setminus K)$.
Combining this bound with~\eqref{eqn:after-besicovitch}, we arrive at
\begin{equation}\label{eqn:non-graphical-measure-bound}
    \begin{split}
    &\cL^m (B^m_{\sigma r} \setminus K) + \|V\| \left( (B_{2r} \cap \pi_S^{-1}(B^m_{\sigma r}) \setminus \{ (z,f(z)) : z \in K \} \right) \\
    & \leq C \Bigl[ \int_{Z_r \times \bG(N,m)} \|T-S\|^2 \, dV + r^2 \int_{Z_r} |H_{\Psi}|^2 \, d \|V\| \Bigr].
    \end{split}
\end{equation}
It remains to control the Dirichlet energy of $f$ over this region.
At $\cL^m$-a.e.~ $z \in K$, the map $f$ is differentiable and the tangent plane to $V$ at $(z,f(z))$ is the graph of $Df(z)$.
On the bounded set $\{ |A| \leq L \}$, the quantity $|A|^2$ is uniformly equivalent to the squared distance between $\textup{graph} \, A$ and $S = \bR^m \times \{ 0 \}$.
Indeed, for unit-length vector $a \in \bR^m$, the unit vector $w := \frac{(a, Df(z) a)}{\sqrt{1 + |(Df) a|^2}}$ lies in $T_{(z,f(z))} V = \textup{graph} (Df(z))$, so we can use $|Df| \leq L$ to bound
\[
(1 + |(Df(z)) \, a|^2)^{- \frac{1}{2}} \, |(Df(z)) \, a| = \textup{dist}(w,S) \leq \|T_{(z,f(z))} V - S \|, \]
from which we deduce that
\begin{equation}
    |Df(z)|^2 \leq C_L \| T_{(z,f(z))} \textup{graph} \, f - S \|^2. \label{eqn:Df(z)-CL-graph}
\end{equation}
Next, using the multiplicity-one graph representation~\eqref{eqn:mult-1-graph} over $K$, we obtain
\begin{align*}
    \int_K |Df|^2 \, dz &\leq C_{m,L} \int_K \| T_{(z,f(z))} V - S \|^2 \, \sqrt{\det(I_m + (Df)^t Df)} \, dz \\
    &= C_{m,L} \int_{ \{ (z,f(z)) : z \in K \} } \|T_x V - S\|^2 \, d \cH^m \leq C_{m,L} \int_{Z_r \times \bG(N,m)} \|T-S\|^2 \, dV
\end{align*}
because the graph Jacobian is at least one.
On the complement of $K$, Rademacher's theorem and the Lipschitz bound imply that $|Df| \leq C_m L$ a.e. 
Therefore,
\[
\int_{B_{\sigma r}^m \setminus K} |Df|^2 \, dz \leq L^2 \cL^m ( B_{\sigma r}^m \setminus K). 
\]
The desired estimate follows from combining the preceding estimates for $|Df|^2$ over $K$ and $B^m_{\sigma r} \setminus K$ with the bound~\eqref{eqn:non-graphical-measure-bound}.
This completes the proof.
\end{proof}

\begin{lemma}\label{lemma:tilt-excess}
Fix $\sigma \in (0,1)$, $L>0$, and $M < + \infty$.
Suppose that $p>m$ and $p \geq 2$.
There exist constants $\ve_0>0$ and $C<+\infty$, depending on $\Psi, m,N,p, \sigma, L,M)$ with the following property.

Let $V$ be an integral $m$-varifold in a neighborhood of $\bar{B}_{4r}(x_0)$ with $\|V\|(B_{3r}(x_0)) \leq M r^m$ and $H_{\Psi} \in L^p(B_{3r}(x_0), \|V\|)$.
Let $A$ be an affine $m$-plane with direction $S \in \bG(N,m)$.
Suppose that $r^{-1} \textup{dist}(x_0,A) \leq \ve_0$, that the mass bound~\eqref{eqn:projected-mass-bound} holds relative to $S \in \bG(N,m)$, and that
\[
E + \Lambda^2 \leq \ve_0, \qquad \text{where} \quad E := E(V, A, x_0, 3r), \qquad \Lambda := \Lambda(V,x_0,3r).
\]
Then, after translating and rotating so that $x_0=0$ and $S = \bR^m \times \{ 0 \}$, there exists a Borel set $K \subset B^m_{\sigma r}$ and an $L$-Lipschitz map $f: B^m_{\sigma r} \to \bR^{N-m}$ such that $V \mres ( B_{2r} \cap \pi_S^{-1}(K))$ agrees with the multiplicity-one graph of $f$ over $K$, as well as
\begin{align*}
    & \int_{B^m_{\sigma r}} |Df|^2 + \cL^m ( B^m_{\sigma r} \setminus K) + \|V\| \bigl( \bigl(B_{2r}\cap\pi_S^{-1}(B_{\sigma r}^m)\bigr) \setminus \{(z,f(z)):z\in K\} \bigr) \leq C r^m \, (E + \Lambda^2).
\end{align*}
Moreover, writing $A = a+S$, where $a \in S^{\perp}$, the function $f$ can be chosen so that
\begin{align*}
r^{-1} \|f -a\|_{L^{\infty} (B^m_{\sigma r})} &\leq C \, E^{\frac{1}{m+2}}, \qquad \text{and} \qquad r^{-m-2} \int_{B^m_{\sigma r}} |f-a|^2 \leq C (E + \Lambda^2).
\end{align*}
\end{lemma}
\begin{proof}
We retain the terms $E,\Lambda$ and apply the above normalization to make $(x_0,S) = (0, \bR^m \times \{ 0 \})$ and $A = a+S$ with $a \in S^{\perp}$, so $|a| = \textup{dist}(0,A) \leq \ve_0 r$.
Since $A = a+S$, for $x \in \textup{spt} \, \|V\| \cap B_{11r/4}$ we can decrease $\ve_0$, bound $|S^{\perp} x|$, and then apply Lemma~\hyperref[lemma:density-and-height-bound]{\ref{lemma:density-and-height-bound}~$(ii)$} with outer ball $B_{3r}$ to obtain
\[
|S^{\perp}x| \leq |S^{\perp} x - a| + |a| = \textup{dist}(x,A) + \textup{dist}(0,A), \qquad \sup_{\textup{spt} \, \|V\| \cap B_{11r/4}} \textup{dist}(x,A) \leq C \, E^{\frac{1}{m+2}} r \]
upon taking $R =\frac{3}{2}r$ and $\lambda = \frac{11}{6}$ there, namely with inner ball $B_{11r/4}$.
We conclude from this that
\begin{equation}\label{eqn:affine-height-bound}
\sup_{\textup{spt} \, \|V\| \cap B_{11r/4}} |S^{\perp} x| \leq ( C E^{\frac{1}{m+2}} + \ve_0) r
\end{equation}
for a constant $C = C(\Psi, m,N, \sigma, L,M)$.
Thus, the assumption of a uniform height bound in Proposition~\ref{prop:lipschitz-approximation} is satisfied after decreasing $\ve_0$, and hence, for $Z_r := Z_r(x_0,S)$ defined in~\eqref{eqn:Zr-domain},
\begin{equation}\label{eqn:V-mres=V-mres-Zr}
V \mres B_{2r} = ( V \mres Z_r) \mres B_{2r}, \qquad Z_r := \{ |\pi_S x| < \tfrac{5}{2} r, \; |S^{\perp} x| < r \}.
\end{equation}
Indeed, $x \in \textup{spt} \, \|V\| \cap B_{2r}$ has $|\pi_S x| \leq |x| < 2r < \frac{5r}{2}$, and $|S^{\perp}x|<r$ by~\eqref{eqn:affine-height-bound}.

We work with the restricted varifold $W := V \mres Z_r$ in what follows.
The projected mass hypothesis~\eqref{eqn:projected-mass-bound} gives $\|W\|(\pi_S^{-1}(B^m_{2r})) \in [\frac{3}{4}, \frac{5}{4}] \omega_m(2r)^m$.
To estimate the tilt-excess on $Z_r$, we use~\eqref{eqn:Zr-containment}, so $\bar{Z}_r \Subset B_{3r}$, to cover $Z_r$ by finitely many balls $\{ B_{\lambda_0 r}(x_i) \}_{i=1}^J$, where $\lambda_0>0$ and $J<+\infty$ depend only on $N$, such that $B_{2 \lambda_0 r}(x_i) \subset B_{3r}$ for every $i$.
Moreover, $\|V\|(B_{2 \lambda_0 r}(x_i)) \leq M r^m \leq M \lambda_0^{-m} ( \lambda_0 r)^m$.
By the Caccioppoli inequality of Proposition~\ref{prop:caccioppoli} with affine comparison plane $A$, we have
\begin{align*}
    &(\lambda_0 r)^{-m} \int_{B_{\lambda_0 r}(x_i)} \|T_x V-S\|^2 \, d\|V\|(x) \\
    &\leq C(\lambda_0 r)^{-m-2} \int_{B_{2 \lambda_0 r}(x_i)} \textup{dist}(x,A)^2 \, d\|V\|(x) + C \bigl[ (\lambda_0 r)^{1- \frac{m}{p}} \|H_{\Psi}\|_{L^p(B_{2 \lambda_0 r}(x_i), \|V\|)} \bigr]^2.
\end{align*}
We see that this quantity is bounded by $C(E+\Lambda^2)$ upon using~\eqref{eqn:holder-bound-for-H-Psi} and
\[
\int_{B_{3r}} \textup{dist}(x,A)^2 \, d\|V\|(x) = (3r)^{m+2} E \qquad \text{and} \qquad \Lambda = (3r)^{1- \frac{m}{p}} \|H_{\Psi}\|_{L^p(B_{3r},\|V\|)}.
\]
We can therefore sum over the finite cover to obtain
\begin{equation}\label{eqn:tilt-from-height-bound}
r^{-m} \int_{Z_r \times \bG(N,m)} \| T - S\|^2 \, d V \leq C ( E + \Lambda^2).
\end{equation}
Combining~\eqref{eqn:holder-bound-for-H-Psi}, ~\eqref{eqn:projected-mass-bound},  ~\eqref{eqn:affine-height-bound}, ~\eqref{eqn:tilt-from-height-bound}, and decreasing $\ve_0$, we deduce that all the hypotheses of Proposition~\ref{prop:lipschitz-approximation} are satisfied.
This produces a Borel set $K \subset B^m_{\sigma r}$ and an $L$-Lipschitz map $f: B^m_{\sigma r} \to S^{\perp}$ such that $V \mres ( B_{2r} \cap \pi_S^{-1}(K))$ agrees with the multiplicity-one graph of $f$ over $K$, and
\begin{align*}
    & \int_{B^m_{\sigma r}} |Df|^2 + \cL^m ( B^m_{\sigma r} \setminus K) + \|V\| \bigl( \bigl(B_{2r}\cap\pi_S^{-1}(B_{\sigma r}^m)\bigr) \setminus \{(z,f(z)):z\in K\} \bigr) \leq C r^m \, (E + \Lambda^2).
\end{align*}
Here we used that $W = V \mres Z_r$ and $V$ agree in $B_{2r}$, so the $\|W\|$-mass agrees with the $\|V\|$-mass.

It remains to establish the two estimates involving $f-a$.
Since $(z,f(z)) \in \textup{spt} \, \|V\| \cap B_{2r}$ for $z \in K$, the bound~\eqref{eqn:affine-height-bound} gives $|f(z)-a| = \textup{dist}( (z,f(z)),A) \leq C E^{\frac{1}{m+2}}r$.
Applying Kirszbraun's theorem~\cite{maggi}*{Theorem 7.2}, we can take a Lipschitz extension of $f-a$, compose it with the nearest-point projection of $S^{\perp}$ onto the closed ball of radius $C E^{\frac{1}{m+2}} r$, and translate back by $a$.
Adding back $a$, we obtain a $L$-Lipschitz extension that preserves the values on $K$, so the Lipschitz constant and Dirichlet energy do not increase.
This proves the bound $r^{-1} \| f - a \|_{L^{\infty}(B^m_{\sigma r})} \leq C E^{\frac{1}{m+2}}$.

Finally, to obtain the bound on $\| f-a\|_{L^2(B^m_{\sigma r})}$, we use the area formula on $K$ to bound
\[
\int_K |f-a|^2 \,  dz\leq \int_{ \{ (z,f(z)) : z \in K \} } \textup{dist}(x,A)^2 \, d \cH^m(x) \leq C E r^{m+2}
\]
by the graphical representation over $K$.
On the complement of $K$, we use the measure bound
\[
\int_{B^m_{\sigma r} \setminus K} |f-a|^2 \,dz \leq C E^{\frac{2}{m+2}} r^2 \cL^m (B^m_{\sigma r} \setminus K) \leq C E^{\frac{2}{m+2}}(E+\Lambda^2) r^{m+2}.
\]
The desired bound follows from $E \leq \ve_0$ after decreasing $\ve_0$.
This completes the proof.
\end{proof}

\begin{corollary}\label{cor:Lipschitz-approximation}
Suppose that the hypotheses of Lemma~\ref{lemma:tilt-excess} hold, and let $f$ and $K$ be the function and Borel set produced therein.
After translating, rotating, and rescaling so that $(x_0,r,S) = (0,1, \bR^m \times \{ 0 \})$, set $E := E(V,A,0,3)$ and $\Lambda := \Lambda(V,0,3)$.
Then, for $\varphi \in C_c^1(B^m_{\sigma} ; \bR^{N-m})$, we have
\begin{align*}
\Bigl| \int_{B^m_{\sigma}} \la D F_{\Psi}(Df) - DF_{\Psi}(0) , D \varphi \rg \Bigr| \leq C(E+\Lambda^2) \, \|D \varphi \|_{L^{\infty}} + C \, \Lambda \| \varphi \|_{L^{\infty}}.
\end{align*}
\end{corollary}
\begin{proof}
For $\varphi \in C_c^1(B^m_{\sigma} ; \bR^{N-m})$, we seek to apply Proposition~\ref{prop:lipschitz-approximation} and Lemma~\ref{lemma:tilt-excess}.
Let us write
\[
\cB := ( B_2 \cap \pi_S^{-1}(B^m_{\sigma})) \setminus \{ (z,f(z)) : z \in K \}
\]
for the non-graphical region whose $\|V\|$-mass we estimate in those results.

The argument of Lemma~\ref{lemma:tilt-excess} implies that the part of $\textup{spt} \, \|V\|$ lying over $\textup{spt} \, \varphi$ is contained in a fixed thin vertical cylinder, provided that $\ve_0$ and $E$ are sufficiently small.
We choose a smooth vertical cutoff function $\chi$ equal to $1$ on $B_{1/2}^{S^\perp}$ and $0$ outside $B_{3/4}^{S^\perp}$, and form the vector field
\[
g(z,z') := \chi(z) ( 0, \varphi(z)).
\]
Because $\sigma<1$, $\textup{spt} \, g \Subset B_2$ is compact.
Moreover, the height estimate on $B_{11/4}$ implies that 
\[
D \chi = 0, \qquad \text{on } \; \textup{spt} \, \|V\| \cap \{ |g| \neq 0 \} \cap \{ Dg \neq 0 \}.
\]
Consequently, testing the vertical variation on the multiplicity-one graph over $K$ produces the term $\int_K \la DF_{\Psi}(Df) , D \varphi \rg$.
Moreover, the contribution to this integral coming from the non-graphical region $\cB$ is bounded by $C \| D \varphi \|_{L^{\infty}} \|V\|(\cB)$.
On the other hand, since $DF_{\Psi}(Df)$ is uniformly bounded for $|Df| \leq L$, we can estimate the contribution away from $K$ by
\[
\Bigl| \int_{B^m_{\sigma} \setminus K} \la D F_{\Psi}(Df) - DF_{\Psi}(0), D \varphi \rg \Bigr| \leq C \|D \varphi\|_{L^{\infty}} \cL^m ( B^m_{\sigma} \setminus K).
\]
Because $\varphi \in C_c^1(B^m_{\sigma}; \bR^{N-m})$ and $DF_{\Psi}(0)$ is constant, we have $\int_{B^m_{\sigma}} \la D F_{\Psi}(0), D \varphi \rg = 0$.

Finally, we combine the above terms and use the first variation identity~\eqref{eqn:first-variation} for the vector field to estimate the above terms, bounding $H_{\Psi}$ via~\eqref{eqn:holder-bound-for-H-Psi}.
Combining these properties, we arrive at
\begin{align*}
& \Bigl| \int_{B^m_{\sigma}} \la DF_{\Psi}(Df) - DF_{\Psi}(0), D \varphi \rg \Bigr| \\
&\qquad \leq C \|D \varphi\|_{L^{\infty}} \, [ \cL^m (B^m_{\sigma} \setminus K) + \|V\| ( \cB) \,] + C \| \varphi\|_{L^{\infty}} \| H_{\Psi} \|_{L^p(B_2, \|V\|)}.
\end{align*}
The claimed bound follows from~\eqref{eqn:holder-bound-for-H-Psi} and the estimates of Lemma~\ref{lemma:tilt-excess}.
\end{proof}

\section{Proof of the main theorem}\label{section:initiate-the-procedure}

We now prove Theorem~\ref{thm:allard}.
The main ingredient of the proof is the iteration proving the decay of the excess; we follow the strategy of~\cite{derosa-tione-regularity}*{Proposition 4.5}, using the results of Section~\ref{sec:preliminaries}.
The only differences are that the Lipschitz map is obtained from Proposition~\ref{prop:lipschitz-approximation} and agrees with \(V\) only outside a set of quadratic size, and that in the final step we use Proposition~\ref{prop:caccioppoli} together with Lemma~\ref{lemma:density-and-height-bound} to transfer the affine approximation from the graph back to the varifold.
A key new step is the estimate~\eqref{eqn:estimate-non-graphical-region} for the varifold mass of the non-graphical portion $\cB$, namely the part not covered by the Lipschitz graph of Proposition~\ref{prop:lipschitz-approximation}.

Recall the notation of~\eqref{eqn:EVSx0r-LambdaVx0r} for the excess and the rescaled $L^p$ norm of the $\Psi$-mean curvature.

\begin{proposition}\label{prop:one-step-iteration}
Fix $\sigma \in (0,1)$ and $L,M<\infty$ as in Lemma~\ref{lemma:tilt-excess}.
Then, there exists a $C = C(\Psi,\cN,m,N,p,\sigma,L,M)<\infty$ such that for every $0<\tau<\frac{\sigma}{100}$, we can find an $\ve_0$, depending on $(\Psi,\cN,m,N,p,\sigma,L,M,\tau)$, with the following property.

Let $V$ be an $m$-varifold in an open neighborhood of $B_{4r}(x_0)$ satisfying the hypotheses of Lemma~\ref{lemma:tilt-excess} relative to an affine $m$-plane $A$ with direction $S \in \bG(N,m)$ and
\[
E := E(V,A,x_0,3r) \leq \ve_0, \qquad \Lambda := \Lambda(V,x_0,3r) \leq E.
\]
Then, there exists an affine $m$-plane $A'$ with direction $S' \in \bG(N,m)$ and such that
\begin{align}
    \|S'-S\| \leq C E^{\frac{1}{2}}, \qquad r^{-1} \textup{dist}(x_0,A') &\leq r^{-1} \textup{dist}(x_0,A) + C E^{\frac{1}{2}}, \qquad \text{and} \label{eqn:planes-close-to-S}  \\
    E(V,A',x_0,3 \tau r) &\leq C \tau^2 E. \label{eqn:quadratic-excess-decay}
\end{align}
\end{proposition}
\begin{proof}
Arguing by contradiction, suppose that for some fixed $\tau \in (0, \frac{\sigma}{100})$ the assertion fails.
By scaling and translating, fix $(x_0,r)=(0,1)$ and obtain integral varifolds $\tilde{V}_j$ and affine planes $\tilde{A}_j$ of directions $S_j$ such that
$\Lambda( \tilde{V}_j,0,3) \leq E(\tilde{V}_j, \tilde{A}_j,0,3) \to 0$, but the conclusion fails. 
We next rotate these configurations: consider maps $R_j \in O(N)$ so that $R_j S_j R_j^t = S := \bR^m \times \{ 0 \}$ and, after passing to a subsequence, extract a limit $R_j \to R_{\infty}$ because $O(N)$ is compact.
This also yields $\Psi_j \to \Psi_{\infty}$ in $C^2(\bG(N,m))$, where $\Psi_j(T) := \Psi(R^t_j T R_j)$.
The rotated varifolds $V_j := (R_j)_{\#} \tilde{V}_j$ and rotated anisotropies $\Psi_j$ satisfy the properties~\eqref{eqn:replace-Psi-R} and have
\begin{equation}\label{eqn:Ej-Lambdaj}
E_j := E(V_j, A_j,0,3) \to 0, \qquad \Lambda_j := \Lambda (V_j,0,3) \leq E_j, \qquad A_j := R_j \tilde{A}_j = a_j + S,
\end{equation}
but there exists no affine plane satisfying the asserted $O(E_j^{1/2})$ closeness to $S$ from~\eqref{eqn:planes-close-to-S} that gives the decay~\eqref{eqn:quadratic-excess-decay}, with $C = C(\Psi, m,N, \sigma, L,M)$ chosen sufficiently large below.
We may assume that $E_j>0$, else if $E_j=0$, then $\Lambda_j=0$ and the conclusion holds trivially with $A'=A_j$.

Applying Lemma~\ref{lemma:tilt-excess} to $V_j$, we obtain $L$-Lipschitz maps $f_j : B^m_{\sigma} \to \bR^{N-m}$ and Borel sets $K_j \subset B^m_{\sigma}$ such that $V_j$ agrees over $K_j$ with the multiplicity-one graph of $f_j$, and
\begin{equation}\label{eqn:extract-a-limit-1}
\int_{B^m_{\sigma}} |Df_j|^2 + \cL^m(B^m_{\sigma} \setminus K_j) + \|V_j\|(\cB_j) \leq C E_j,
\end{equation}
where $\cB_j$ denotes the non-graphical portion of $V_j$ in $B_2$, namely
\[
\cB_j := (B_2 \cap \pi_S^{-1}(B^m_{\sigma}) ) \setminus \{ (z, f_j(z)) : z \in K_j \}.
\]
Here, we used the bound $\Lambda_j \leq C E_j$ to absorb the bound on the left-hand side into $E_j$.
Moreover, 
\begin{equation}\label{eqn:extract-a-limit-2}
\int_{B^m_{\sigma}} |f_j - a_j|^2 \leq C E_j, \qquad \|f_j - a_j \|_{L^{\infty}(B^m_{\sigma})} \leq C E_j^{\frac{1}{m+2}} \to 0.
\end{equation}
We consider the rescaled functions $v_j := \frac{f_j - a_j}{E_j^{1/2}}$.
By the properties~\eqref{eqn:extract-a-limit-1} and~\eqref{eqn:extract-a-limit-2}, the sequence is bounded in $W^{1,2}(B^m_{\sigma} ; \bR^{N-m})$.
Thus, we can pass to a subsequence and extract a limit
\begin{equation}\label{eqn:L2-convergence}
v_j \rightharpoonup v \quad \text{weakly in } \; W^{1,2}(B^m_{\sigma}), \qquad v_j \to v \quad \text{strongly in } \; L^2(B^m_{\sigma}).
\end{equation}
We claim that the function $v$ is a weak solution of the linear system
\begin{equation}\label{eqn:solve-linear-system}
\int_{B^m_{\sigma}} D^2 F_{\Psi_{\infty}}(0) \, [Dv , D \varphi] = 0, \qquad \text{for every } \; \varphi \in C_c^1(B^m_{\sigma} ; \bR^{N-m}).
\end{equation}
Indeed, let us consider the rescaled non-parametric integrand 
\begin{align*}
F_j(X) := E_j^{-1} \bigl[ F_{\Psi_j}( E_j^{1/2} X) - F_{\Psi_j}(0) - E_j^{1/2} \, \la D F_{\Psi_j}(0), X \rg \bigr], \\
\text{so } \quad DF_j(Dv_j) = E_j^{- \frac{1}{2}} \, \bigl[ DF_{\Psi_j} (Df_j) - DF_{\Psi_j}(0) \bigr].
\end{align*}
We apply Corollary~\ref{cor:Lipschitz-approximation} and divide the resulting inequality by $E^{1/2}_j$ to obtain
\[
\Bigl| \int_{B^m_{\sigma}} \la DF_j(Dv_j), D \varphi \rg \Bigr| \leq C \frac{E_j + \Lambda_j^2}{E_j^{1/2}} \|D \varphi \|_{L^{\infty}} + C \frac{\Lambda_j}{E_j^{1/2}} \| \varphi \|_{L^{\infty}}.
\]
The assumption~\eqref{eqn:Ej-Lambdaj} implies that the right-hand side tends to zero.
On the other hand,
\[
DF_j(Dv_j) = \int_0^1 D^2 F_{\Psi_j} (t Df_j) [Dv_j] \, dt
\]
Since $Df_j \to 0$ strongly in $L^2$ due to~\eqref{eqn:extract-a-limit-1}, and $|Df_j| \leq L$ for each $j$, up to a subsequence we find $Df_j \to 0 $ a.e.
The continuity of $D^2 F_{\Psi_{\infty}}$ together with the convergence $\Psi_j \to \Psi_{\infty}$ in $C^2(\bG(N,m))$, as well as the dominated convergence theorem, implies that $\int_0^1 D^2 F_{\Psi_j} (t Df_j) \, dt \to D^2 F_{\Psi_{\infty}}(0)$ strongly in $L^q$, for every finite $q$.
Since $Dv_j \rightharpoonup Dv$ in $L^2$ and $D \varphi \in L^{\infty}$, the claim~\eqref{eqn:solve-linear-system} follows from combining the above estimates and sending $j \to \infty$.

By Proposition~\ref{prop:quasiconvexity}, $D^2 F_{\Psi_{\infty}}(0)$ satisfies the uniform Legendre-Hadamard condition.
Moreover, the constants $\alpha = \alpha_{\cN,\Psi}>0$ and $c_{\Psi,\cN,R}>0$ obtained therein are independent of $j$, because each rotated anisotropy $\Psi_j$ is again QEC with the rotated Lipschitz field $\cN_j(T) = R_j \cN(R^t_j T R_j) R_j^t$.
Also, the $C^2$-orbit of $\Psi$ under $O(N)$ is compact, so we obtain a uniform positive lower bound and deduce the uniform continuity of the Hessians.
Thus,~\eqref{eqn:solve-linear-system} is a constant-coefficient strongly elliptic system, and classical interior estimates~\cite{giaquinta-martinazzi}*{\S 5.2} yield
\begin{equation}\label{eqn:v-interior-estimates}
    \sup_{B^m_{\sigma/2}} |D^2 v| + |v(0)| + |Dv(0)| \leq C \, \|v\|_{W^{1,2}(B^m_{\sigma})} \leq C.
\end{equation}
We combine Taylor's theorem for the linear expansion of $v$ with the bound~\eqref{eqn:v-interior-estimates} to find, for $\tau < \frac{1}{20} \sigma$,
\begin{align*}
& |v(x) - v(0) - Dv(0) \, x| \leq C |x|^2 \|v \|_{L^2(B_{\sigma/2}^m)}, \\
& \qquad \implies \int_{B^m_{4 \tau}} |v-\ell|^2 \leq C \tau^{m+4} \qquad \text{for } \; \ell(z) := v(0) + Dv(0) z.
\end{align*}
Let $A'_j$ be the graph of the affine map $z \mapsto a_j+ E_j^{1/2} \, \ell(z)$ with direction $S'_j$, so~\eqref{eqn:v-interior-estimates} gives
\begin{equation}\label{eqn:planes-stay-close}
\|S'_j - S\| \leq C E_j^{\frac{1}{2}}, \qquad \textup{dist}(0,A'_j) \leq |a_j| + C E_j^{\frac{1}{2}} = \textup{dist}(0,A_j) + C E_j^{\frac{1}{2}}.
\end{equation}
Next, we consider the graphical portion of $V_j$ and estimate
\begin{align*}
& \textup{dist} \bigl( (z, f_j(z)), A'_j \bigr) \leq |f_j(z) -a_j- E_j^{\frac{1}{2}} \, \ell(z)| = E_j^{\frac{1}{2}} \, |v_j(z) - \ell(z)|, \\
&\qquad \implies E_j^{-1} \int_{ \{ (z, f_j(z)) : z \in K_j \} \cap B_{3\tau} } \textup{dist} (x, A'_j)^2 \, d \|V_j \|(x) \leq C \int_{B^m_{4 \tau}} |v_j - \ell|^2
\end{align*}
using the area formula and the uniform Lipschitz bound.
The convergence~\eqref{eqn:L2-convergence} implies
\begin{equation}\label{eqn:limsup-Ej-tau}
\limsup_{j \to \infty} E_j^{-1} \int_{ \{ (z,f_j(z)) : z \in K_j \} \cap B_{3\tau} } \textup{dist}(x, A'_j)^2 \, d \|V_j \| \leq C \tau^{m+4}.
\end{equation}
It remains to estimate the non-graphical portion $\cB_j$.
Combining the inequality~\eqref{eqn:extract-a-limit-1} with the height bound of Lemma~\ref{lemma:density-and-height-bound}, we obtain on $B_{3 \tau} \subset B_{3/2}$ the bound
\[
E_j^{-1} \|V_j \|(\cB_j) \leq C, \qquad \implies \qquad \sup_{\textup{spt} \, \|V_j \| \cap B_{3 \tau}} \textup{dist}(x,A_j) \leq C E_j^{\frac{1}{m+2}} \to 0.
\]
Moreover, on $B_{3 \tau}$, we have $\sup_{\textup{spt} \, \|V_j \| \cap B_{3 \tau}} \textup{dist}(x,A'_j) \to 0$ as well, due to
\[
\sup_{x \in B_{3 \tau}} |\textup{dist}(x, A'_j) - \textup{dist}(x,A_j)| \leq C E_j^{\frac{1}{2}}, \qquad \text{where} \quad E_j \to 0.
\]
Combining these steps, we arrive at
\begin{equation}\label{eqn:estimate-non-graphical-region}
E_j^{-1} \int_{\cB_j \cap B_{3\tau}} \textup{dist} (x, A_j')^2 \, d \|V_j \|(x) \leq E_j^{-1} \|V_j \|(\cB_j) \sup_{\textup{spt} \, \|V_j \| \cap B_{3 \tau}} \textup{dist}(x, A_j')^2 \to 0.
\end{equation}
Combining this bound with~\eqref{eqn:limsup-Ej-tau}, we deduce that
\[
    \limsup_{j \to \infty} E_j^{-1} \int_{B_{3\tau}} \textup{dist} (x, A'_j)^2 \, d\|V_j\| \leq C \tau^{m+4},
\]
and hence $\limsup_{j \to \infty} E_j^{-1} E(V_j, A'_j, 0, 3 \tau) \leq C' \tau^2$ for some $C' = C(\Psi, m,N,\sigma,L,M)$.
Choosing the constant $C$ in the statement larger than $C'$ gives a contradiction.
This proves our assertion.
\end{proof}

In Allard's argument~\cite{allard-regularity}*{\S 2.5}, it is proved that, for small excess and first variation relative to the local scale, the projected density remains close to a single integer throughout the iteration across scales.
Consequently, after decreasing the initial smallness parameter, the mass hypothesis required for the Lipschitz approximation is preserved at every scale.
We establish this step now.

We will specify the constant $M := \omega_m 6^m$ in the assumptions of Propositions~\ref{prop:lipschitz-approximation} through~\ref{prop:one-step-iteration}.

\begin{lemma}\label{lemma:persistence-of-mass-one}
Let $\Psi \in C^1(\bG(N,m),(0,\infty))$ satisfy AC, and let $p>m$ and $\theta \in (0,\frac{1}{4})$.
There is a $\tau_p = \tau_p(m,\theta)>0$ so that for $\tau \in (0, \frac{1}{4})$, there exists an $\ve_0 = \ve_0(\Psi,m,N,p,\theta,\tau)$ with the following property.
Let $V$ be an integral $m$-varifold in a neighborhood of $\bar{B}_{2r}(x_0)$ with $x_0 \in \textup{spt} \, \|V\|$, $\Lambda(V,x_0,2r) \leq \ve_0$, and, for some $S \in \bG(N,m)$,
\begin{align*}
    \left| \frac{\|V\|(B_{2r}(x_0))}{\omega_m (2r)^m} - 1 \right| \leq \tau_p, \qquad r^{-m} \int_{B_{2r}(x_0)} \|T_x V - S \|^2 \, d \|V\|(x) \leq \ve_0,
\end{align*}
Then, the following properties hold for every $y \in \textup{spt} \, \|V\| \cap B_{\theta r}(x_0)$, every affine $m$-plane $A = a'+S'$ whose direction $S' \in \bG(N,m)$ satisfies $r^{-1} \textup{dist}(x_0,A) + \|S'-S \| \leq \ve_0$, and every $j \in \{1,2,3\}$: 
\begin{align}
    (\theta r)^{-1} \sup_{x \in \textup{spt} \, \|V\| \cap B_{3\theta r}(y)} \textup{dist}(x, A) + (\theta r)^{-1} \sup_{\textup{spt} \, \|V\| \cap B_{3 \theta r}(y)} \textup{dist}(x,y+S') &< \tau, \label{eqn:plane-and-affine-plane-converge}\\
\Biggl| \frac{ \|V\| \bigl( B_{3 \theta r}(y) \cap \{ x : |\pi_{S'}(x-y)| < 2 \theta r \} \bigr) }{\omega_m (2 \theta r)^m} - 1 \Biggr| < \tau \qquad \text{and} \qquad \Biggl| \frac{\|V\|(B_{j \theta r}(y))}{\omega_m (j \theta r)^m} - 1 \Biggr| &< \tau .\label{eqn:v-3-theta-r}
\end{align}
\end{lemma}
\begin{proof}
By translating and rescaling, we can assume that $(x_0, r) = (0,1)$.
Also, fix a $\tau_p < \frac{1}{10} (1 - 10 \theta^2)^{\frac{m}{2}}$.
We argue by contradiction: fix $\tau, \theta \in (0, \frac{1}{4})$ and suppose that the assertion is false for every $\ve_0>0$.
Then, there exists a sequence of integral $m$-varifolds $V_k$ in a neighborhood of $\bar{B}_2$, points $y_k \in \textup{spt} \, \|V_k \| \cap B_{\theta}$, and affine $m$-planes $A_k = a_k + S'_k$ such that, for some $S_k \in \bG(N,m)$,
\begin{align}
0 \in \textup{spt} \, \|V_k \|, \qquad \bigl| (\omega_m 2^m)^{-1} \|V_k \|(B_2) - 1 \bigr| \leq \tau_p, \label{eqn:contradiction-assumptions} \\
\int_{B_2 \times \bG(N,m)} \|T-S_k\|^2 \, dV_k(x,T) \to 0, \qquad 
\textup{dist}(0,A_k) + \| S'_k - S_k \| \to 0,
\label{eqn:convergence-of-planes}
\end{align}
while $\|H_{\Psi}(V_k) \|_{L^p (B_2, \|V_k\|)} \to 0$, but one of the conclusions of~\eqref{eqn:plane-and-affine-plane-converge},\eqref{eqn:v-3-theta-r} fails for $V_k, y_k, A_k, S'_k$.
We normalize $A_k = a_k + S'_k$ with $a_k \in (S'_k)^{\perp}$, after replacing $a_k$ by $\tilde{a}_k := (I - S'_k) a_k$.
Then, 
\begin{equation}\label{eqn:dist-x-Ak}
    |a_k| = \textup{dist}(0,A_k) \to 0, \qquad \textup{dist}(x,A_k) = |(I - S'_k) (x - a_k)|.
\end{equation}
The assumed mass bound gives $\sup_k \|V_k \|(B_2) < \infty$, and H\"older's inequality gives, for $K \Subset B_2$,
\[
|\delta_{\Psi} V_k|(K) \leq \int_K |H_{\Psi,k}| \, d\|V_k\| \leq \|H_{\Psi}(V_k)\|_{L^p(B_2, \|V_k\|)} \| V_k \|(B_2)^{1 - \frac{1}{p}} \to 0.
\]
Since $\Psi$ satisfies AC and admits a Michael-Simon inequality,~\cite{dephilippis-pigati}*{Proposition 1.9 and Corollary 1.11} give compactness of integer-rectifiable varifolds: after passing to a subsequence, we obtain $V_k \rightharpoonup V$, for $V$ an integral $m$-varifold in $B_2$.
Moreover, the uniform density bound of Lemma~\hyperref[lemma:density-and-height-bound]{\ref{lemma:density-and-height-bound}~$(i)$} and the Hausdorff convergence result from~\cite{dephilippis-pigati}*{Corollary 1.12(ii)} implies that $0 \in \textup{spt} \, \|V\|$, since $0 \in \textup{spt} \, \|V_k\|$ for each $k$.
The upper mass bound passes to the limit, so $\|V\|(B_2) \leq (1+\tau_p) \omega_m 2^m$.
By the compactness of $\bG(N,m)$, after passing to a further subsequence, we have $S_k \to S \in \bG(N,m)$.
Thus, we also obtain $S'_k \to S$, due to
\[
\| S'_k - S \| \leq \| S'_k - S_k \| + \| S_k - S \| \to 0.
\]
We next identify the limiting varifold.
For every non-negative $\zeta \in C_c(B_2)$, we can use the properties~\eqref{eqn:contradiction-assumptions}, ~\eqref{eqn:convergence-of-planes} and the varifold convergence to obtain $\int \zeta(x) \|T-S\|^2 \, dV(x,T)=0$.
Thus, $T_x V = S$ for $\|V\|$-a.e. $x \in B_2$.
Likewise, we find
\[
[\delta_{\Psi} V](g) = \lim_{k \to \infty} [\delta_{\Psi} V_k](g) = 0, \qquad \forall \; g \in C_c^1(B_2;\bR^N),
\]
hence $V$ is $\Psi$-stationary.
For $e \in S$, we have $B_{\Psi}(S)^t e = \Psi(S) e$, and testing with $g = e \varphi$ gives $\int \la e, D \varphi \rg \, d\|V\| = 0$ for every $\varphi \in C_c^1(B_2)$.
Therefore, $\partial_e \|V\| = 0$ as distributions for every $e \in S$.
After identifying $S \simeq \bR^m = \textup{span} \{ e_1, \dots, e_m \}$, the distributional identities $\partial_{e_i} \|V\| = 0$ for $1 \leq i \leq m$ imply that $\|V\|$ is locally invariant under translations in the direction of $S$: for every cylinder $U \times W \Subset ( S \times S^{\perp}) \cap B_2$, there is a Radon measure $\nu$ on $W$ such that $\|V\| \mres (U \times W) = (\cH^m \mres U) \otimes \nu$.
Since $\|V\|$ is $m$-rectifiable and integral, we have $\|V\| = ( \cH^m \mres S ) \otimes \nu $ where the transverse measure $\nu$ is purely atomic.
The atoms have integer weights and are locally finite.
Thus, $V$ is locally a sum of integer-multiplicity affine planes parallel to $S$, $V \mres B_2 = \sum_{\alpha} q_{\alpha} \, |S+a_{\alpha}| \mres B_2$ for $q_{\alpha} \in \bN$ and $a_{\alpha} \in S^{\perp}$, so $\textup{dist}(0,S+a_{\alpha}) = |a_{\alpha}|$.
Since $0 \in \textup{spt} \, \|V\|$, we have $a_{\alpha}=0$ for some $\alpha$, so the plane $S$ itself contributes to this sum with multiplicity $q_0 \geq 1$, producing a term $q_0 \omega_m 2^m \leq (1+\tau_p) \omega_m 2^m$ for the mass of $V$; hence $q_0 = 1$.
To see that no other translate $S+a_{\alpha}$ can meet $B_{5 \theta}$, observe that
\[
|a_{\alpha}| < 5 \theta \implies (\omega_m 2^m)^{-1} \,\cH^m ( (S+a_{\alpha}) \cap B_2) = ( 1 - \tfrac{1}{4} |a_{\alpha}|^2 )^{\frac{m}{2}} \geq (1 - 10 \theta^2)^{\frac{m}{2}}.
\]
Together with the contribution of the multiplicity-one plane $S$, this would give $\|V\|(B_2) \geq (1+c_{\theta}) \omega_m 2^m$, contradicting the mass bound for $\tau_p < \frac{1}{10} (1 - 10 \theta^2)^{\frac{m}{2}}$.
Thus, 
\begin{equation}\label{eqn:v-mres-b-4-theta}
    V \mres B_{5 \theta} = |S| \mres B_{5 \theta}.
\end{equation}
After passing to a further subsequence, we obtain $y_k \to y$ for some $y \in \bar{B}_{\theta}$, and $y \in S$ due to the local Hausdorff convergence of supports and~\eqref{eqn:v-mres-b-4-theta}.
Moreover, since $B_{3 \theta}(y_k) \subset B_{4 \theta}$ for every $k$,
\begin{equation}\label{eqn:dist(x,P)-to-zero}
\sup_{x \in \textup{spt} \, \|V_k \| \cap B_{3 \theta}(y_k)} \textup{dist}(x,S) \to 0,
\end{equation}
by the local Hausdorff convergence of supports in $\bar{B}_{4 \theta}$ and~\eqref{eqn:v-mres-b-4-theta}.
We now verify the three conclusions of the Lemma.
First, since $y_k \in \textup{spt} \, \|V_k\|$, ~\eqref{eqn:dist(x,P)-to-zero} also gives $\textup{dist}(y_k,S) \to 0$.
Together with $S'_k \to S$, for $x \in \textup{spt} \, \|V_k \| \cap B_{3 \theta}(y_k)$, we have
\begin{align*}
    \textup{dist}(x, y_k + S'_k) &= |(I-S'_k)(x -y_k)| \leq |(I-S)(x-y_k)| + \|S'_k - S \| \, |x-y_k| \\
    &\leq \textup{dist}(x,S) + \textup{dist}(y_k,S) + 3 \theta \, \|S'_k - S \|.
\end{align*}
Therefore, $\sup_{\textup{spt} \, \|V_k \| \cap B_{3 \theta}(y_k)} \textup{dist}(x, y_k + S'_k) \to 0$.
Likewise, using~\eqref{eqn:dist-x-Ak},
\[
\textup{dist}(x, A_k) = |(I-S'_k)(x-a_k)| \leq \textup{dist}(x, S) + |x| \, \| S'_k - S \| + |a_k| \to 0
\]
by using $x \in B_{4 \theta}$, ~\eqref{eqn:convergence-of-planes}, ~\eqref{eqn:dist-x-Ak}, and~\eqref{eqn:dist(x,P)-to-zero}.
Therefore, $\sup_{\textup{spt} \, \|V_k \| \cap B_{3 \theta} (y_k)}\textup{dist}(x, A_k) \to 0$.
Since $\theta>0$ is fixed, this property, together with~\eqref{eqn:dist(x,P)-to-zero}, implies the conclusion claim~\eqref{eqn:plane-and-affine-plane-converge}.

Next, we have $V_k \rightharpoonup |S|$ in $B_{5 \theta}$ and $\|S\|(\partial B_{\rho}(y)) = 0$ for $\rho \leq 3 \theta$.
Since $y_k \to y$, for every sufficiently small $\delta>0$ and all large $k$, we use the containment $B_{j \theta-\delta}(y) \subset B_{j \theta}(y_k) \subset B_{j \theta+\delta}(y)$ and weak convergence to obtain
\[
\omega_m (j \theta-\delta)^m \leq \liminf_{k \to \infty} \|V_k \|(B_{j \theta}(y_k)), \qquad \limsup_{k \to \infty} \|V_k \| (B_{j \theta}(y_k)) \leq \omega_m (j \theta+\delta)^m
\]
for a.e.~ $\delta>0$.
Sending $\delta \downarrow 0$, we deduce that $\|V_k \|(B_{j \theta}(y_k)) \to \omega_m (j \theta)^m$ for $j \in \{1,2,3\}$.

Finally, on the rotating planes $S'_k$, write $C_k := B_{3 \theta}(y_k) \cap \{ x : |\pi_{S'_k}(x-y_k)| < 2 \theta \}$.
The functions $x \mapsto |x-y_k|$ and $x \mapsto |\pi_{S'_k}(x-y_k)|$ converge uniformly on compact subsets of $B_{5 \theta}$ to $x \mapsto |x-y|$ and $x \mapsto |\pi_S(x-y)|$.
Since $y \in S$, we have $|\pi_S(x-y)| =|x-y|$ on $S$, thus
\[
S \cap B_{3 \theta}(y) \cap \{ |\pi_S(x-y)| < 2 \theta \} = S \cap B_{2 \theta}(y)
\]
and this set has $m$-dimensional measure $\omega_m (2 \theta)^m$.
The boundaries defined by $\{ |x-y| = 3 \theta \} \cup \{ |\pi_S(x-y) | = 2 \theta \}$ have zero $\|S\|$-measure, hence the same inner-outer approximation argument used above shows that $\|V_k\|(C_k) \to \omega_m(2 \theta)^m$.
Therefore~\eqref{eqn:v-3-theta-r} also holds for all sufficiently large $k$.
This contradicts the choice of the sequence and completes the proof.
\end{proof}

\begin{lemma}\label{lemma:restart}
Fix some $0 < \theta < \min \{ \frac{1}{4},
\frac{\sigma}{100}\}$, let $\tau_p(m,\theta)$ be as in Lemma~\ref{lemma:persistence-of-mass-one}, and choose some $0 < \tau_0 < \min \{ \frac{1}{20}, \tau_p(m,\theta) \}$ and $0 < \eta < \frac{1}{20} \tau_0$.
There exists some $\ve>0$, depending on these choices, such that, after decreasing the smallness constants in Lemma~\ref{lemma:tilt-excess} and~\ref{lemma:persistence-of-mass-one}, the following holds.
Suppose that the hypotheses required to apply Lemma~\ref{lemma:tilt-excess} hold at a point $x \in \textup{spt} \, \|V\|$ and a scale $r$, relative to an affine plane $A$ with direction $S$, and moreover
\[
\left| \frac{\|V\| (B_{2r}(x))}{\omega_m(2r)^m} - 1 \right| \leq \tau_0.
\]
Suppose that $\tilde{A}$ is an affine plane with direction $\tilde{S}$ satisfying 
\[
r^{-1} \textup{dist}(x, \tilde{A}) + \| \tilde{S} - S \| \leq \ve \qquad \text{and} \qquad E(V, \tilde{A}, x, 3 \theta r) + \Lambda (V,x, 3 \theta r)^2 \leq \ve.
\]
Then, the hypotheses of Lemmas~\ref{lemma:tilt-excess} and~\ref{lemma:persistence-of-mass-one} hold at $x$ and at scale $\theta r$, relative to $\tilde{A}$ and $\tilde{S}$, with
\[
\left| \frac{\| V\|(B_{2 \theta r}(x))}{\omega_m (2 \theta r)^m} - 1 \right| < \eta < \tau_0.
\]
\end{lemma}
\begin{proof}
We first verify the hypotheses of Lemma~\ref{lemma:persistence-of-mass-one} at the scale $r$.
The assumption of near-unit mass is part of the hypotheses.
Moreover, applying Proposition~\ref{prop:caccioppoli} on a finite covering of $B_{2r}(x)$ by balls whose doubles are contained in $B_{3r}(x)$, we obtain
\begin{equation}\label{eqn:first-caccioppoli}
r^{-m} \int_{B_{2r}(x)} \| T_y V - S\|^2 \, d\|V\|(y) \leq C \, [ E(V,A,x,3r) + \Lambda (V,x,3r)^2 ].
\end{equation}
Thus, after decreasing the original smallness constant, the tilt-excess hypothesis of Lemma~\ref{lemma:persistence-of-mass-one} is satisfied.
The curvature hypothesis follows as well from the smallness assumed in Lemma~\ref{lemma:tilt-excess}, so we can apply Lemma~\ref{lemma:persistence-of-mass-one} with $y=x$ and with the comparison plane $\tilde{A}$ of direction $\tilde{S}$ to obtain
\begin{align*}
    ( \theta r)^{-1} \sup_{\textup{spt} \, \|V\| \cap B_{3 \theta r}(x)} \textup{dist}(y, \tilde{A}) + (\theta r)^{-1} \sup_{\textup{spt} \, \|V\| \cap B_{3 \theta r}(x)} \textup{dist}(y,x+\tilde{S}) &<\eta, \\
    \left| \frac{\|V\|(B_{3 \theta r}(x) \cap \{ y : |\pi_{\tilde{S}}(y-x)| < 2 \theta r \})}{\omega_m (2 \theta r)^m} - 1 \right| &< \eta, 
\end{align*}
and $\Bigl| \frac{\|V\| (B_{j \theta r}(x))}{\omega_m (j \theta r)^m} - 1 \Bigr| < \eta$ for $j=1,2,3$.
We now check the hypotheses at the new scale $\theta r$.
Since $x \in \textup{spt} \, \|V\|$, the assumed bound $E(V, \tilde{A}, x,3 \theta r) + \Lambda(V,x,3 \theta r)^2 \leq \ve$ and Lemma~\hyperref[lemma:density-and-height-bound]{\ref{lemma:density-and-height-bound}~$(ii)$}, applied with outer ball $B_{3 \theta r}(x)$, give $(3 \theta r)^{-1} \textup{dist}(x, \tilde{A}) \leq C \ve^{\frac{1}{m+2}}$.
The $j=3$ mass estimate also implies
\[
\|V\|(B_{3 \theta r}(x)) \leq (1+\eta) \omega_m (3 \theta r)^m \leq M  (\theta r)^m, \qquad \text{for } \; M = \omega_m 6^m.
\]
Next, if $\eta<1$, then on $\textup{spt} \, \|V\| \cap B_{3 \theta r}(x)$ we have $\textup{dist}(y,x+\tilde{S}) < \theta r$.
Consequently, on $\textup{spt} \, \|V\|$, 
\[
B_{3 \theta r}(x) \cap \{ y : |\pi_{\tilde{S}}(y-x)| < 2 \theta r\} = Z_{\theta r} (x, \tilde{S}) \cap \pi^{-1}_{\tilde{S}} \bigl(B^{\tilde{S}}_{2 \theta r} ( \pi_{\tilde{S}} x) \bigr).
\]
One inclusion follows from the preceding height bound $\textup{dist}(y,x+\tilde{S}) < \theta r$; for the converse,
\[
|\pi_{\tilde{S}}(y-x)| < 2 \theta r, \qquad |\tilde{S}^{\perp}(y-x)| < \theta r, \qquad \implies \qquad |y-x| < \sqrt{5} \theta r< 3 \theta r.
\]
Using $\eta < \frac{1}{20} \tau_0 < \frac{1}{400}$, we may therefore apply the estimate on the projected mass to obtain
\[
\tfrac{3}{4} \omega_m (2 \theta r)^m < \|V\| \bigl( Z_{\theta r} ( x, \tilde{S}) \cap \pi_{\tilde{S}}^{-1} ( B^{\tilde{S}}_{2 \theta r} ( \pi_{\tilde{S}} x))  \bigr) < \tfrac{5}{4} \omega_m ( 2 \theta r)^m.
\]
Together with the assumed smallness of $E(V, \tilde{A}, x,3 \theta r) + \Lambda (V,x, 3 \theta r)^2 \leq \ve$, this verifies all the hypotheses of Lemma~\ref{lemma:tilt-excess} at the scale $\theta r$.

It remains only to verify that Lemma~\ref{lemma:persistence-of-mass-one} can also be reapplied at this scale.
Since $\eta<\tau_0$, the $j=2$ estimate implies the hypothesis of near-unit mass, and $\Lambda(V,x,2 \theta r) \leq c_N \Lambda(V,x,3 \theta r)$, up to a fixed scaling factor, so this is sufficiently small.
Finally, another application of Proposition~\ref{prop:caccioppoli} as above, now on a fixed finite covering of $B_{2 \theta r}(x)$ by balls whose doubles lie in $B_{3 \theta r}(x)$, gives the analogue of~\eqref{eqn:first-caccioppoli}, with $(A,S,3r)$ replaces by $(\tilde{A}, \tilde{S}, 3 \theta r)$.
After decreasing the smallness constant once more, this is below the threshold in Lemma~\ref{lemma:persistence-of-mass-one}.
Thus, the hypotheses of Lemmas~\ref{lemma:tilt-excess} and~\ref{lemma:persistence-of-mass-one} hold at $x$ and at scale $\theta r$, relative to $\tilde{A}, \tilde{S}$, as claimed.
This completes the proof.
\end{proof}

Let us start the procedure of iterating the decay of the excess to prove Theorem~\ref{thm:allard}.
We generally follow the approach of~\cite{allard-first-variation}*{\S 8.17-8.19} as in Allard's proof of the Basic Regularity Lemma in~\cite{allard-regularity}*{\S 3.6}, but using the strategy of~\cite{derosa-tione-regularity}*{\S 4} towards a more direct proof.

\begin{proposition}\label{prop:initiate-the-procedure}
Let $C_0$ denote the constant in Proposition~\ref{prop:one-step-iteration}.
We choose $0 < \theta < \min \{ \frac{1}{4}, \frac{\sigma}{100}\}$ so small that $\max \{ C_0 \theta^2 , \theta^{1 - \frac{m}{p}} \} \leq \frac{1}{8}$ and set $K = 8 \theta^{-(m+2)}$.
After decreasing the initial smallness constant, suppose that the hypotheses of Lemmas~\ref{lemma:tilt-excess} and~\ref{lemma:persistence-of-mass-one} hold at $x_0$ and at scale $r_0$, relative to an affine plane $A_0$ with direction $S_0$, with the corresponding constants taken sufficiently small and $E_0 + K \Lambda_0 \leq \ve_0$.
Then, there are affine $m$-planes $A_k$ with directions $S_k$ such that, writing
\[
E_k := E(V, A_k, x_0, 3 r_k), \qquad \Lambda_k := \Lambda(V,x_0,3r_k), \qquad \text{for } \; r_k := \theta^k r_0,
\]
and defining $0 < \alpha < \frac{1}{2} (1- \frac{m}{p})$ by $\theta^{2 \alpha} = \frac{1}{4}$, it holds that 
\begin{equation}\label{eqn:iteration-step}
    E_k + K \Lambda_k \leq \theta^{2 \alpha k} (E_0 + K \Lambda_0), \qquad \| S_{k+1} - S_k \| \leq C E_k^{\frac{1}{2}},
\end{equation}
and all the hypotheses of Lemmas~\ref{lemma:tilt-excess} and~\ref{lemma:persistence-of-mass-one} hold at every scale $r_k$.
Moreover, the affine planes $A_k$ converge as $k \to \infty$ to the affine plane $x_0 + S^{(x_0)}$, for some $S^{(x_0)} \in \bG(N,m)$, with
\begin{equation}\label{eqn:Px0-Pk}
    \| S_k - S^{(x_0)} \| \leq C (E_0 + K \Lambda_0)^{\frac{1}{2}} \theta^{\alpha k}.
\end{equation}
\end{proposition}
\begin{proof}
Having chosen $\alpha>0$ by $\theta^{2 \alpha} = \frac{1}{4}$, we can prove~\eqref{eqn:iteration-step} in the form $E_k + K \Lambda_k \leq 4^{-k}(E_0 + K \Lambda_0)$.
Because we also imposed $\theta^{1- \frac{m}{p}} \leq \frac{1}{8}$, we can arrange $\alpha < \frac{1}{2} ( 1- \frac{m}{p})$ after decreasing $\theta$.
We prove the result by induction; it suffices to address the inductive step, supposing that $A_k$ has been constructed and all the required hypotheses hold at scale $r_k$.
Note that
\[
\Lambda_{k+1} = (3 r_{k+1})^{1 - \frac{m}{p}} \|H_{\Psi} \|_{L^p(B_{3 r_{k+1}}(x_0))} \leq \theta^{1 - \frac{m}{p}} \Lambda_k.
\]
We distinguish two cases.
If $\Lambda_k \leq E_k$, we can apply Proposition~\ref{prop:one-step-iteration} with $\tau =\theta$ to obtain an affine plane $A_{k+1}$ with direction $S_{k+1}$, such that
\[
E_{k+1} \leq C_0 \theta^2 \, E_k \qquad \text{and} \qquad \| S_{k+1} - S_k \| \leq C E_k^{\frac{1}{2}}.
\]
Consequently,
\begin{align*}
    E_{k+1} + K \, \Lambda_{k+1} \leq C_0 \theta^2 E_k + K \theta^{1 - \frac{m}{p}} \Lambda_k \leq \tfrac{1}{8} E_k + \tfrac{1}{8} K \Lambda_k \leq \tfrac{1}{4} (E_k + K \Lambda_k).
\end{align*}
This proves~\eqref{eqn:iteration-step} in the first case.
If instead $E_k < \Lambda_k$, we preserve the affine plane $A_{k+1} = A_k$ with $S_{k+1} = S_k$, hence the scaling estimate gives
\begin{align*}
    E_{k+1} = (3 \theta r_k)^{- m-2} \int_{B_{3 \theta r_k}(x_0)} \textup{dist}(y, A_k)^2 \, d \|V\|(y) \leq \theta^{-(m+2)} E_k < \theta^{-(m+2)} \Lambda_k.
\end{align*}
Thus, we may again bound
\[
E_{k+1} + K \Lambda_{k+1} \leq \theta^{-(m+2)} \Lambda_k + K \theta^{1 - \frac{m}{p}} \Lambda_k = \bigl( \tfrac{1}{8} + \theta^{1 - \frac{m}{p}} \bigr) K \Lambda_k \leq \tfrac{1}{4} K \Lambda_k
\]
which again implies~\eqref{eqn:iteration-step}.
It remains to verify that this construction can indeed be repeated.
Since $x_0 \in \textup{spt} \, \|V\|$, Lemma~\hyperref[lemma:density-and-height-bound]{\ref{lemma:density-and-height-bound}~$(ii)$} combined with Proposition~\ref{prop:one-step-iteration} gives
\[
r_k^{-1} \textup{dist} (x_0, A_{k+1}) + \| S_{k+1} - S_k \| \leq C E_k^{\frac{1}{m+2}} + C E_k^{\frac{1}{2}}
\]
in the first case.
In the second case, the bound on $r_k^{-1} \textup{dist}(x_0, A_{k+1}) + \| S_{k+1} - S_k \|$ is immediate because $A_{k+1} = A_k$.
The property~\eqref{eqn:iteration-step} implies that the right-hand side is uniformly as small as desired once $E_0 + K \Lambda_0 \leq \ve_0$ sufficiently small.
We may therefore apply Lemma~\ref{lemma:restart} with $\tilde{A} = A_{k+1}$, which shows that all the requirements on the mass, projected mass, height, and affine position to apply Lemmas~\ref{lemma:tilt-excess} and~\ref{lemma:persistence-of-mass-one} at scale $r_{k+1}$.
Moreover, the bound~\eqref{eqn:iteration-step} guarantees that the excess and the curvature remain small.
This completes the inductive step.

Finally, the bound~\eqref{eqn:iteration-step} implies that $E_k \leq \theta^{2 \alpha k}(E_0 + K \Lambda_0)$ and
\[
\sum_{k=0}^{\infty} \| S_{k+1} - S_k \| \leq C (E_0 + K \Lambda_0)^{\frac{1}{2}} \sum_{k=0}^{\infty} \theta^{\alpha k} < \infty.
\]
Thus, the sequence $\{ S_k \}$ converges to an $m$-plane $S^{(x_0)}$ with $\| S_k - S^{(x_0)} \| \leq C (E_0 + K \Lambda_0)^{\frac{1}{2}} \theta^{\alpha k}$.
This proves the inequality~\eqref{eqn:Px0-Pk}.
We also obtain
\[
r_k^{-1} \textup{dist}(x_0, A_k) \leq C E_k^{\frac{1}{m+2}} \leq C (E_0 + K \Lambda_0)^{\frac{1}{m+2}} \theta^{\frac{2 \alpha}{m+2} k } \to 0.
\]
Thus, the limiting plane is $x_0 + S^{(x_0)}$ and satisfies~\eqref{eqn:Px0-Pk}, as claimed.
\end{proof}

\begin{corollary}\label{cor:s-holder}
In the notation of Proposition~\ref{prop:initiate-the-procedure}, for every $s \in (0,r_0]$ we have
\begin{align*}
s^{-m} \int_{B_s(x_0)} \| T_y V - S^{(x_0)} \|^2 \, d \|V\|(y) &\leq C (E_0 + K \Lambda_0) \Bigl( \frac{s}{r_0} \Bigr)^{2 \alpha}, \\
s^{-1} \sup_{\textup{spt} \, \|V\| \cap B_s(x_0)} \textup{dist}(y, x_0 + S^{(x_0)} ) &\leq C (E_0 + K \Lambda_0)^{\frac{1}{m+2}} \Bigl( \frac{s}{r_0} \Bigr)^{\frac{2 \alpha}{m+2}}.
\end{align*}
\end{corollary}
\begin{proof}
At each radius $r_k$ in the induction of Proposition~\ref{prop:initiate-the-procedure}, the Caccioppoli inequality gives
\[
r_k^{-m} \int_{B_{r_k}(x_0) \times \bG(N,m)} \| T - S_k \|^2 \, d V(y,T) \leq C (E_k + \Lambda_k^2).
\]
The parallelogram equality shows that, for every plane $T$,
\begin{equation}\label{eqn:T-P(x)}
    \| T - S^{(x_0)} \|^2 \leq 2 \, \| T - S_k \|^2 + 2 \| S_k - S^{(x_0)} \|^2.
\end{equation}
Combining the above Caccioppoli inequality with the bounds~\eqref{eqn:Px0-Pk} and~\eqref{eqn:T-P(x)}, we deduce that
\[
r_k^{-m} \int_{B_{r_k}(x_0)} \|T - S^{(x_0)} \|^2 \, dV(y,T) \leq C (E_0 + K \Lambda_0) \theta^{2 \alpha k}.
\]
If $s \in (r_{k+1}, r_k]$, then $\frac{r_k}{s} \leq \theta^{-1}$, so the first claim bound follows from this inequality.
Regarding the second claimed bound, we use the height bound of Lemma~\hyperref[lemma:density-and-height-bound]{\ref{lemma:density-and-height-bound}~$(ii)$} on the affine plane $A_k$ together with the inequality $E_k \leq \theta^{2 \alpha k} (E_0 + K \Lambda_0)$ to obtain
\[
r_k^{-1} \sup_{\textup{spt} \, \|V\| \cap B_{r_k}(x_0)} \textup{dist}(y, A_k) \leq C (E_0 + K \Lambda_0)^{\frac{1}{m+2}} \theta^{\frac{2 \alpha}{m+2} k}.
\]
The same estimate applies to $\textup{dist}(x_0, A_k)$.
We combine with this bound with~\eqref{eqn:Px0-Pk} and 
\[
\textup{dist}(y, x_0 + S^{(x_0)}) \leq \textup{dist}(y, A_k) + \textup{dist}(x_0, A_k) + C r_k \, \| S_k - S^{(x_0)} \|
\]
to obtain the second claimed inequality, for $s \in (r_{k+1}, r_k]$.
The result for arbitrary radii $s \in (0,r_0]$ follows by the same inductive argument as in Proposition~\ref{prop:initiate-the-procedure}.
This completes the proof.
\end{proof}

\begin{corollary}\label{cor:recenter-uniformly}
After decreasing the initial smallness constant, there are fixed numbers $c \in (0,1)$ and $\rho \in (0,1)$ such that, for every $x \in \textup{spt} \, \|V\| \cap B_{c r_0}(x_0)$, all the hypotheses required to begin the preceding iteration hold at $x$ and scale $\rho r_0$, with constants independent of $x$.
\end{corollary}
\begin{proof}
The estimates involving the height-excess and mean curvature follow from the inclusion of balls: for $c, \rho$ fixed and sufficiently small, we can arrange $B_{3 \rho r_0}(x) \subset B_{r_0}(x_0)$, and hence
\begin{align*}
E(V, A_0, x, 3 \rho r_0) &\leq C_{\rho} E(V, A_0, x_0, 3r_0), \qquad \Lambda (V,x,3 \rho r_0) \leq C_{\rho} \Lambda (V,x_0, 3r_0), \\
(\rho r_0)^{-1} \textup{dist}(x, A_0) &\leq C_{\rho} E(V, A_0, x_0, 3r_0)^{\frac{1}{m+2}}
\end{align*}
using the height bound of Lemma~\hyperref[lemma:density-and-height-bound]{\ref{lemma:density-and-height-bound}~$(ii)$} at the original scale.
Moreover, the projected mass property~\eqref{eqn:projected-mass-bound} and multiplicity-one conditions of Lemma~\ref{lemma:persistence-of-mass-one} hold uniformly by Lemma~\ref{lemma:persistence-of-mass-one}. 
\end{proof}

\begin{proof}[Proof of Theorem~\ref{thm:allard}]
We complete the proof in five steps.

\smallskip \noindent \textbf{Step 1: Initialization.}
First, we fix the constants $\theta,\tau$ of Lemma~\ref{lemma:persistence-of-mass-one} and Proposition~\ref{prop:initiate-the-procedure}, with $\tau>0$ sufficiently small.
We first apply Lemma~\ref{lemma:persistence-of-mass-one} at $x_0$ and scale $r$, with $(y,A,S')=(x_0, x_0+S_0,S_0)$.
Setting $(r_0,A_0) := (\theta r, x_0 +S_0)$, we obtain the requirements of~\eqref{eqn:projected-mass-bound} and Lemma~\ref{lemma:persistence-of-mass-one} at scale $r_0$, together with $r_0^{-1} \sup_{\textup{spt} \, \|V\| \cap B_{3r_0}(x_0)} \textup{dist}(x, A_0)< \tau$.
Using the corresponding mass upper bound and the Caccioppoli inequality with the same finite covering argument as above, we find
\begin{align*}
    E(V, A_0, x_0, 3 r_0) \leq C \tau^2, \qquad \Lambda (V, x_0, 3r_0) \leq C \, \Lambda (V, x_0, 2r), \\
    r_0^{-m} \int_{B_{2r_0}(x_0)} \|T-S_0\|^2 \, d V \leq C \, [ E(V, A_0, x_0, 3r_0) + \Lambda (V, x_0, 3r_0)^2].
\end{align*}
Thus, by first choosing $\tau$ sufficiently small and then decreasing the smallness constant in the theorem statement, the hypotheses of Lemmas~\ref{lemma:tilt-excess} and~\ref{lemma:persistence-of-mass-one} hold at $x_0$ and scale $r_0$, relative to $A_0$ with direction $S_0$, and we can initiate the procedure of Proposition~\ref{prop:initiate-the-procedure}. 
We record the dependence of the resulting estimates on the original smallness parameter $\delta(x_0,S_0,r)$ at the end of the proof.

By Corollary~\ref{cor:recenter-uniformly}, after decreasing the initial smallness constant and shrinking the interior ball, there are fixed $c,\rho>0$ such that, for every $x \in \tilde{G} := \textup{spt} \, \|V\| \cap B_{2 c r_0}(x_0)$, the hypotheses of Lemmas~\ref{lemma:tilt-excess} and~\ref{lemma:persistence-of-mass-one} hold at $x$ and at scale $\rho r_0$, uniformly in $x$, relative to the initial affine plane $A_0$ with direction $S_0$.
We may therefore initialize the iteration of Proposition~\ref{prop:initiate-the-procedure} by setting $(r_0^{(x)}, A_0^{(x)}, S_0^{(x)}) := (\rho r_0, A_0, S_0)$.
Moreover, by the estimates of Corollary~\ref{cor:recenter-uniformly}, we have
\[
E(V, A^{(x)}_0, x, 3 r_0^{(x)}) +K \, \Lambda (V, x, 3r_0^{(x)}) \leq C_{\rho} (E_0 + K \, \Lambda_0).
\]
Thus, after decreasing the original smallness constant once more, the hypotheses of Proposition~\ref{prop:initiate-the-procedure} and Corollary~\ref{cor:recenter-uniformly} are satisfied uniformly for every $x \in \tilde{G}$.

\smallskip \noindent \textbf{Step 2: The plane field map $x \mapsto S^{(x)}$ is $C^{0,\alpha}$.}
For every $x\in \tilde{G}$, Proposition~\ref{prop:initiate-the-procedure} produces an $m$-plane $S^{(x)}$ and uniform constants such that, with $\cE := E_0 + K \Lambda_0$,
\begin{align}
    s^{-m} \int_{B_s(x) \times \bG(N,m)} \| T-S^{(x)} \|^2 \, d V(y,T) &\leq C \cE \Bigl( \frac{s}{r_0} \Bigr)^{2 \alpha},  \label{eqn:tilt-excess-uniform} \\
    s^{-1} \sup_{\textup{spt} \, \|V\| \cap B_s(x)} \textup{dist}(y, x+ S^{(x)}) &\leq C \cE^{\frac{1}{m+2}} \Bigl( \frac{s}{r_0} \Bigr)^{\frac{2 \alpha}{m+2}}. \label{eqn:distance-uniform}
\end{align}
Now, the map $x \mapsto S^{(x)}$ is $C^{0,\alpha}$: after shrinking the interior ball if necessary, we claim that
\begin{equation}\label{eqn:holder-planes}
    \| S^{(x)} - S^{(y)} \| \leq C \cE^{\frac{1}{2}} r_0^{-\alpha} |x-y|^{\alpha}, \qquad \forall \; x,y \in \tilde{G}.
\end{equation}
Indeed, the result is clear if $d = |x-y| \geq c r_0$, due to the uniform bound $\|S^{(x)} - S_0 \| + \| S^{(y)} - S_0 \| \leq C \cE^{\frac{1}{2}}$.
If $d< cr_0$, then using $B_{d/4}(y) \subset B_{2d}(x)$, the lower density bound of Lemma~\hyperref[lemma:density-and-height-bound]{\ref{lemma:density-and-height-bound}~$(i)$} shows that $\|V\|(B_{d/4}(y)) \geq c'd^m$.
Using the parallelogram inequality as in~\eqref{eqn:T-P(x)}, we can bound $ \frac{1}{2} \| S^{(x)} - S^{(y)}\|^2 \leq  \| T - S^{(x)} \|^2 +  \| T - S^{(y)}\|^2$.
Integrating on $B_{d/4}(y)$ then yields
\[
c' d^m \| S^{(x)} - S^{(y)} \|^2 \leq \int_{B_{d/4}(y)} \|T_z V-S^{(x)}\|^2 \, d \|V\|(z) + \int_{B_{d/4}(y)} \| T_z V - S^{(y)}\|^2 \, d\|V\|(z).
\]
The first integral is bounded using~\eqref{eqn:tilt-excess-uniform} and $B_{d/4}(y) \subset B_{2d}(x)$; the second integral is bounded using~\eqref{eqn:tilt-excess-uniform} on $B_{d/4}(y)$.
Therefore, $d^m \| S^{(x)} - S^{(y)} \|^2 \leq C \cE d^m ( \frac{d}{r_0})^{2 \alpha}$.
Cancelling $d^m$ and taking square roots proves the assertion~\eqref{eqn:holder-planes}.

\smallskip \noindent \textbf{Step 3: Graphicality.}
Next, after shrinking the constant $c$ further, we can arrange that $\sup_{x \in \tilde{G} } \| S^{(x)} - S_0 \| < \frac{1}{20}$, with $S_0$ the original reference direction, and that the right-hand side of~\eqref{eqn:distance-uniform} is $< \frac{1}{40}$ at each relevant scale.
For $x,y \in \tilde{G}$, we apply~\eqref{eqn:distance-uniform} with radius $2 |x-y|$ to find
\[
\textup{dist}(y, x +S^{(x)}) = |( I - S^{(x)}) (y-x)| \leq \tfrac{1}{20} |y-x|.
\]
We therefore obtain $|\pi_{S_0}(y-x)| \geq \frac{1}{2} |y-x|$, due to
\[
|S_0^{\perp}(y-x)| \leq |(I - S^{(x)})(y-x)| + \| S^{(x)} - S_0 \| \, |y-x| \leq \tfrac{1}{10} |y-x|.
\]
Thus, $\pi_{S_0}$ is injective on $\tilde{G}$ with Lipschitz inverse.
Having established all estimates in $\tilde{G} = \textup{spt} \, \|V\| \cap B_{2 cr_0}(x_0)$, we deduce that $G := \textup{spt} \, \|V\| \cap \bar{B}_{c r_0}(x_0)$ is a graph over $D = \pi_{S_0}(G)$.

\smallskip \noindent \textbf{Step 4: $D$ contains a ball.}\label{step-4-in-Allard}
We now prove that the domain $D$ contains a ball of fixed size: after decreasing the initial smallness constant, there exists a $\gamma_0 = \gamma_0(\Psi,N, m,p)>0$ such that $B^{S_0}_{\gamma_0 r_0}(\pi_{S_0}x_0)\subset D$.
Indeed, denote by $C_{\mathrm{bi}}$ the uniform Lipschitz constant of the inverse of $\pi_{S_0} : \tilde{G} \to \pi_{S_0}(\tilde{G})$ obtained in Step 3, and recall that every recentered iteration starts at
the scale $\rho r_0$, for $\rho>0$ the constant from Corollary~\ref{cor:recenter-uniformly}. 
Choose $\gamma_0>0$ so small that
\begin{equation}\label{eqn:choice-of-gamma-zero}
\gamma_0\leq\tfrac{\sigma\rho}{16}
\qquad\text{and}\qquad
2C_{\mathrm{bi}}\gamma_0+
\tfrac{32\gamma_0}{\sigma\theta}<c.
\end{equation}
Suppose, towards a contradiction, that there exists a $z\in B^{S_0}_{\gamma_0r_0}(\pi_{S_0}x_0)\setminus D$.
Since $D$ is compact, we may choose $w\in D$ such that $|z-w|=\textup{dist}(z,D)=:\delta>0$, and let $x\in G$ be the unique point satisfying $\pi_{S_0}x=w$.
Since $\pi_{S_0}x_0\in D$, we have $\delta\leq |z-\pi_{S_0}x_0|<\gamma_0r_0$, and hence
\[ 
|w-\pi_{S_0}x_0| \leq |w-z|+|z-\pi_{S_0}x_0| <2\gamma_0r_0.
\]
The bi-Lipschitz estimate from Step~3 therefore gives $|x-x_0| \leq C_{\textup{bi}} |w - \pi_{S_0} x_0| < 2 \, C_{\textup{bi}} \gamma_0 r_0$.

We now consider the sequence of scales $r_k^{(x)}:=\theta^k\rho r_0$ of the iteration centered at $x$. 
By the first inequality in~\eqref{eqn:choice-of-gamma-zero}, there is
a $k\geq0$ such that, writing $r_k:=r_k^{(x)}$, we have 
\[
\tfrac{\sigma\theta}{16}r_k < \delta \leq \tfrac{\sigma}{16}r_k,
\]
because successive radii have ratio $\theta$.
In particular, $r_k < \frac{16}{\sigma \theta} \delta \leq \frac{16 \gamma_0}{\sigma \theta} r_0$.
Combined with the bound $|x-x_0| < 2 C_{\textup{bi}} \gamma_0 r_0$, this implies that $B_{2r_k}(x) \subset B_{cr_0}(x_0)$.
We conclude from this that
\[
\textup{spt} \, \|V\| \cap B_{2r_k}(x) \subset G := \textup{spt} \, \|V\| \cap \bar{B}_{cr_0}(x_0).
\]
Let $A_k=A_k^{(x)}$  and $S_k=S_k^{(x)}$ be the affine plane and its direction produced by the iteration at the point $x$ and the scale $r_k$. 
By~\eqref{eqn:Px0-Pk}, the estimate of Step~2, and the fact that every recentered iteration starts with direction $S_0$, we may decrease the initial smallness constant so that $\|S_k - S_0 \| \leq \frac{1}{10}$ for every relevant center $x$ and every $k\geq0$.
Hence, the linear map $\pi_{S_0}|_{S_k} : S_k \to S_0$ has a $2$-Lipschitz inverse.
For $c_* := \frac{\sigma \theta}{64}$, the estimates of Proposition~\ref{prop:initiate-the-procedure} and Corollaries~\ref{cor:s-holder} and~\ref{cor:recenter-uniformly} allow us to decrease the original constant so that
\begin{equation}\label{eqn:small-holes}
\begin{split}
    r_k^{-1} \textup{dist}(x, A_k) + C E^{\frac{1}{m+2}}_k + [ 2 \omega_m^{-1} C(E_k+ \Lambda_k^2)]^{\frac{1}{m}} \leq c_*, \\
    \text{where} \qquad E_k := E(V, A_k, x, 3 r_k), \qquad \Lambda_k := \Lambda (V,x,3r_k),
\end{split}
\end{equation}
uniformly in $x$ and $k$.
Here, $C$ is chosen larger than the constants in
Lemma~\ref{lemma:tilt-excess}.

In the next step, we apply Lemma~\ref{lemma:tilt-excess} at $x$ and scale $r_k$, relative to $A_k$ and $S_k$.
This produces a Borel set $K_k\subset B^{S_k}_{\sigma r_k}$ and a Lipschitz map $f_k:B^{S_k}_{\sigma r_k}\to S_k^\perp$ whose graph agrees with $V$ over $K_k$. Writing $A_k-x=a_k+S_k$ with $a_k \in S^{\perp}_k$, we have $|a_k|=\textup{dist}(x,A_k)$.
The height estimate in Lemma~\ref{lemma:tilt-excess}, together with~\eqref{eqn:small-holes}, gives
\begin{equation}\label{eqn:fk-small-hole} 
\|f_k\|_{L^\infty(B^{S_k}_{\sigma r_k})} \leq |a_k|+\|f_k-a_k\|_{L^\infty(B^{S_k}_{\sigma r_k})} \leq c_\ast r_k.
\end{equation}
We define the point $\hat{z} := ( \pi_{S_0}|_{S_k})^{-1} (z - \pi_{S_0} x) \in S_k$.
Since $\pi_{S_0}x=w$, the choice of scale and the $2$-Lipschitz bound imply $|\widehat z| \leq2|z-w| = 2\delta \leq\frac{\sigma}{8}r_k$.
Therefore, $B^{S_k}_{c_\ast r_k}(\widehat z) \subset B^{S_k}_{\sigma r_k}$, and we claim that
\begin{equation}\label{eqn:hole-contained-in-bad-set}
B^{S_k}_{c_\ast r_k}(\widehat z)
\subset B^{S_k}_{\sigma r_k}\setminus K_k.
\end{equation}
Indeed, given $u\in K_k\cap B^{S_k}_{c_\ast r_k}(\widehat z)$, the corresponding graphical point $y := x + u + f_k(u)$ belongs to $\textup{spt}\|V\|\cap B_{2r_k}(x) \subset G$, so $\pi_{S_0}y\in D$. 
On the other hand, the definition of
$\widehat z$ and~\eqref{eqn:fk-small-hole} imply
\[
|\pi_{S_0}y-z| = \left| \pi_{S_0}|_{S_k}(u-\widehat{z}) + \pi_{S_0} f_k(u) \right| \leq |u-\widehat z|+|f_k(u)|< 2c_\ast r_k = \tfrac{\sigma \theta}{32} r_k < \delta
\]
due to the choice of $\delta$.
This contradicts $\textup{dist}(z,D)=\delta$ and proves
\eqref{eqn:hole-contained-in-bad-set}.
Now,~\eqref{eqn:hole-contained-in-bad-set} implies $\cL^m(B^{S_k}_{\sigma r_k}\setminus K_k) \geq \omega_m c_\ast^m r_k^m$, while Lemma~\ref{lemma:tilt-excess} combined with~\eqref{eqn:small-holes} gives
\[
\cL^m(B^{S_k}_{\sigma r_k}\setminus K_k) \leq Cr_k^m(E_k+\Lambda_k^2) \leq \tfrac{1}{2}\omega_m c_\ast^m r_k^m,
\]
which is a contradiction. 
Thus, no such point $z$ exists, and $B^{S_0}_{\gamma_0 r_0} (\pi_{S_0} x_0) \subset D$ follows.

\smallskip \noindent \textbf{Step 5: Conclusion of the proof.}
Choose a universal constant $\gamma_\star>0$ satisfying $\gamma_{\star} \leq \frac{\theta}{2} \min \{ c, \gamma_0\}$.
Since $r_0 = \theta r$, we have $B_{\gamma_{\star} r}(x_0) \subset B_{cr_0}(x_0)$ and
\[
B^{S_0}_{\gamma_\star r}(\pi_{S_0}x_0) \subset B^{S_0}_{\gamma_0r_0}(\pi_{S_0}x_0) \subset D.
\]
Hence, $G$ is the graph of a Lipschitz map $f$ over
$B^{S_0}_{\gamma_\star r}(\pi_{S_0}x_0)$, with one-to-one projection onto $S_0$. 
After translating $x_0$ to the origin and rotating $S_0$ onto $\bR^m\times\{0\}$, we can view $f$ as a map $f: B^m_{\gamma_{\star} r} \to \bR^{N-m}$.
For $\|V\|$-a.e.~ $x$, integrality gives $\Theta^m (\|V\|,x) \in \bN$, while iterating Lemma~\ref{lemma:persistence-of-mass-one} gives $\frac{3}{4} < \frac{\|V\|(B_s(x))}{\omega_m s^m} < \frac{5}{4}$ at arbitrarily small scales.
Passing to the density gives $\Theta^m (\|V\|,x) \in [ \frac{3}{4}, \frac{5}{4}]$, so $\Theta^m ( \|V\|,x) = 1$ for $\|V\|$-a.e.~ $x$ and $V = v( \textup{graph} \, f,1)$ in a fixed smaller neighborhood.

For $x = (z,f(z))$, the limiting plane $S^{(x)}$ is close to $S_0$, hence it is uniquely of the form $S^{(x)} = \textup{graph} \, L_z$, for a linear map $L_z : S_0 \to S_0^{\perp}$.
Using~\eqref{eqn:holder-planes}, we obtain $\| L_z - L_w \| \leq C \cE^{\frac{1}{2}} \bigl( \frac{|z-w|}{r_0} \bigr)^{\alpha}$.
Also, $f$ is differentiable at a.e.~ $z$, and $T_{(z,f(z))} V = \textup{graph} \, Df(z)$ because the varifold is multiplicity one.
On the other hand,~\eqref{eqn:tilt-excess-uniform} implies that the unique approximate tangent plane at this point is $S^{(z,f(z))}$, hence $Df(z) = L_z$ for a.e.~ $z$.
By~\eqref{eqn:holder-planes}, the right-hand side $z \mapsto L_z$ is $C^{0,\alpha}$.
Since $f$ is Lipschitz and its weak derivative agrees a.e.~ with the continuous function $L_z$, it follows that $f \in C^{1,\alpha}$, $Df(z) = L_z$ for every $z$, and $[Df]_{C^{0,\alpha}} \leq C r_0^{- \alpha} \cE^{\frac{1}{2}}$.
Finally, because $S^{(x)}$ is uniformly close to $S_0$, while $f( \pi_{S_0}x_0) = S_0^{\perp} x_0 = 0$ after translation, we obtain
\begin{equation}\label{eqn:c0alpha-bound}
\| Df \|_{L^{\infty}} \leq C \cE^{\frac{1}{2}}, \qquad \| f\|_{L^{\infty}} \leq C r \| Df \|_{L^{\infty}} \leq C r \cE^{\frac{1}{2}}
\end{equation}
because $D$ contains a ball.
Then,~\eqref{eqn:holder-planes} gives the $C^{0,\alpha}$ seminorm.

We finally record the dependence on the original smallness parameter.
The number $\theta$ is fixed once and for all in terms of the structural data, with $\max \{ C_0 \theta^2, \theta^{1 - \frac{m}{p}} \} \leq \frac{1}{8}$ and $K := 8 \theta^{-(m+2)}$ so $\theta, K$ are independent of $\delta (x_0, S_0, r)$ in the notation of Theorem~\ref{thm:allard}.
For $t \in [0,\ve_{\star}]$, we define
\[
\eta(t) := \sup\left\{ E(V,x_0+S_0,x_0,3\theta r) + K\Lambda(V,x_0,3\theta r): \delta(x_0,S_0,r)\leq t \right\},
\]
where the supremum is taken over all configurations satisfying the hypotheses of Theorem~\ref{thm:allard}. 
For every $\tau>0$, Lemma~\ref{lemma:persistence-of-mass-one} provides a $d_{\tau} > 0$ such that $E(V, x_0+S_0, x_0, 3 \theta r) \leq C \tau^2$ whenever $\delta(x_0, S_0, r) \leq d_{\tau}$.
Moreover, $\Lambda(V, x_0, 3 \theta r) \leq C_{\theta} \delta (x_0, S_0, r)$.
Consequently, 
\[
\eta(t) \leq C\tau^2 + K C_{\theta} t \qquad \text{for } \; t \leq d_{\tau}.
\]
Since $\tau>0$ is arbitrary, this proves $\lim_{t\downarrow0}\eta(t)=0$; also, $\eta(t)$ is non-decreasing by definition.

We now set $(r_0, A_0) := (\theta r, x_0 + S_0)$.
For every point $x$ at which we recenter the argument, as in Corollary~\ref{cor:recenter-uniformly}, let $(r_0^{(x)}, A_0^{(x)}) := (\rho r_0, A_0) = (\rho \theta r, A_0)$.
The estimates therein give
\[
E (V, A_0^{(x)}, x, 3r_0^{(x)}) \leq C_{\rho} \, E(V, A_0, x_0, 3 r_0), \qquad \Lambda (V,x, 3 r_0^{(x)}) \leq C_{\rho} \, \Lambda (V, x_0, 3r_0).
\]
Hence, uniformly over all the recentered iterations that apply Proposition~\ref{prop:initiate-the-procedure}, we find
\[
\cE_x := E(V, A_0^{(x)} , x, 3r_0^{(x)}) + K \, \Lambda (V, x, 3r_0^{(x)}) \leq C_{\rho} \, \eta ( \delta (x_0, S_0, r))
\]
because the estimates of Corollaries~\ref{cor:s-holder} and~\ref{cor:recenter-uniformly} are uniform under recentering at $x \in \textup{spt} \, \|V\| \cap B_{cr_0}(x_0)$.
Hence, using the relations~\eqref{eqn:holder-planes} and~\eqref{eqn:c0alpha-bound} with right-hand side involving $\cE_x^{\frac{1}{2}}$, we conclude
\[
r^{-1} \| f \|_{L^{\infty}} + \| Df \|_{L^{\infty}} + r^{\alpha_{\star}} [Df]_{C^{0,\alpha_{\star}}} \leq C \, \sup_{x \in \tilde{G}} \cE_x^{\frac{1}{2}} \leq C \, \eta(\delta(x_0, S_0, r))^{\frac{1}{2}}.
\]
Consequently, after increasing the constant if necessary,
the conclusion of the theorem holds with $\omega_{\star}(t) = C \eta(t)^{\frac{1}{2}}$.
This completes the proof.
\end{proof}

\begin{remark}\label{rmk:on-allard}
We collect here some observations on Theorem~\ref{thm:allard}.
\begin{enumerate}[(i)]
    \item After identifying $V = v(\textup{graph} \, f,1)$ in $B_{\gamma_{\star} r}$, testing $\delta_{\Psi} V$ with vertical vector fields shows that
    \[
    \textup{div} ( DF_{\Psi}(Df)) = \cA(Df) \, \pi_{\bR^{N-m}} H_{\Psi}(z,f(z))
    \]
    where $F_{\Psi}$ and $\cA$ are defined in~\eqref{graphical}.
    For small $\|Df\|_{L^{\infty}}$, QEC makes this a uniformly elliptic system; thus, if $\Psi \in C^{2,\beta}$ for some $\beta \in (0,1)$, standard interior $W^{2,p}$ estimates give $f \in W^{2,p}_{\textup{loc}} \hookrightarrow C^{1, 1- \frac{m}{p}}$. 
    If $H_{\Psi} \in L^{\infty}$, the endpoint regularity is $f \in W^{2,q}_{\textup{loc}}$ for every $q<\infty$, hence $f \in C^{1,\gamma}_{\textup{loc}}$ for every $\gamma<1$.
    More generally, if $\Psi \in C^{k+2,\beta}$ and $z \mapsto H_{\Psi}(z,f(z))$ is $C^{k,\beta}_{\textup{loc}}$, then Schauder theory yields $f \in C^{k+2,\beta}_{\textup{loc}}$.
    See~\cite{giaquinta-martinazzi} for the relevant elliptic estimates for systems.
    \item The $C^2$ regularity required of the anisotropy cannot be lowered to $C^{1,\sigma}$ for any $\sigma \leq 1$ within the present argument.
    The Caccioppoli inequality of Proposition~\ref{prop:caccioppoli}  requires the Lipschitz continuity of $B_{\Psi}$, hence of $D \Psi$; this forces $\Psi \in C^{1,1}$.
    The stronger $C^2$ assumption is used in the non-parametric formulation of the problem via $F_{\Psi}$, to ensure that the differential $D^2 F_{\Psi}$ satisfies the Legendre-Hadamard condition of Proposition~\ref{prop:quasiconvexity} pointwise and exists at every reference plane produced by the iteration of Proposition~\ref{prop:initiate-the-procedure}.
   
\end{enumerate}
\end{remark}

\section{Examples of atomic conditions}

We now exhibit the first examples of anisotropic integrands not close to the area functional satisfying the various atomic conditions, and notably USAC.

\subsection{Atomic conditions for \texorpdfstring{$\ell^q$}{ellq} norms}

In this section, we prove Theorem~\ref{thm:ell-p-on-G(N,m)}: we will show that the $\ell^q$ anisotropies are USAC in every dimension and codimension, for $q \in [ q_{N,m}, 2]$.
Following~\cite{derosa-tione-regularity}*{\S 6}, the $\ell^q$ anisotropy $\Psi_q$ is defined on $\bG(N,m)$ by
\begin{equation}\label{eqn:psi-p-norm}
    \Psi_q(T) := \Bigl( \sum_{ |I| = m } |(\xi_T)_I|^q \Bigr)^{\frac{1}{q}}
\end{equation}
where $\xi_T$ is a Euclidean unit simple $m$-vector spanning $T \in \bG(N,m)$.
We identify planes with orthogonal projections and use the Frobenius norm distance.
For an even one-homogeneous norm $F$ on $\bigwedge^m\bR^N$ that is $C^1$ away from the origin, and a unit simple $m$-vector $\xi$ spanning $T$, we let
\begin{equation}\label{eqn:nu-xi}
    \nu_{\xi} := \frac{DF(\xi)}{F(\xi)}, \qquad (\Pi_T)_{ab} := \la \iota_{e_a} \nu_{\xi}, \iota_{e_b} \xi \rg,
\end{equation}
where $\iota_{e_a}$ denotes the vector contraction.
Then,
\begin{equation}\label{eqn:normalized-scalar-kernel}
    \la B_{\Psi_q}(T), B_{\Psi^*_q}(S^{\perp} ) \rg = \Psi_q(T) \Psi_q(S) K_F(T,S), \qquad \text{where} \quad K_F(T,S) := m - \textup{tr}(\Pi_T \Pi_S).
\end{equation}
Here, we recall that $\Psi^*_q(S^{\perp}) = \Psi_q(S)$ and take $F_0(\xi) = \| \xi \|_{\ell^q}$ to be the Pl\"ucker norm.
Our main observation is that, for $2 - \frac{7}{10 \min \{m, N-m \}}\leq q \leq 2$,
\begin{proposition}\label{prop:kernel}
There is a $C_{N,m,q}>0$ such that $K_F(T,S) \geq C_{N,m,q} ( m - \textup{tr}(TS))$.
\end{proposition}

We first record some elementary inequalities.
We write, for brevity,
\begin{equation}\label{eqn:eps-m-q-m-constants}
\ve_m = \frac{7}{10m}, \qquad h_m = \frac{2}{\ve_m} = \frac{20m}{7}, \qquad M_m = \frac{(2m)^m}{m!}.
\end{equation}

\begin{lemma}\label{lemma:elementary-inequalities}
Let $R \geq 1$ satisfy $\bE [ R^{h_m}] \leq M_m$.
For $m \geq 4$, there is a $b_m \geq 1$ with
\begin{equation}\label{eqn:expected-value}
\bE[ (R^{\theta} - b^{\theta}_m)^2 ] \leq ( \tfrac{13}{32} \bigr)^2 \, \theta^2  \qquad \text{for every } \; \theta \in [0,1].
\end{equation}
For $m \in \{ 2, 3\}$, we can find a $b_m$ with $\bE[ (R^{\theta} - b^{\theta}_m)^2] \leq ( \frac{m+5}{20})^2 \theta^2$ for every $\theta \in [0,1]$.
\end{lemma}
\begin{proof}
It suffices to prove the claim for $\theta=1$, namely that $\bE[ (R - b_m)^2] \leq ( \tfrac{13}{32} \bigr)^2$.
Indeed, the map $r \mapsto r^{\theta}$ is $\theta$-Lipschitz on $[1,\infty)$ for $\theta \in [0,1]$, so $\bE[ (R^{\theta} - b^{\theta}_m)^2] \leq \theta^2 \bE[ (R-b_m)^2]$ will imply~\eqref{eqn:expected-value} once the case $\theta=1$ is established.
We observe that $M_m < \frac{1}{4} ( \frac{29}{16})^{h_m}$ for $m \geq 4$; indeed, this holds for $m=4$ since $h_4>11$ and $( \frac{29}{16})^{11} > 4 M_4$, while
\[
\frac{M_{m+1}/ ( \frac{29}{16})^{h_{m+1}}}{M_{m}/ ( \frac{29}{16})^{h_{m}}} = 2 \Bigl( 1 + \frac{1}{m} \Bigr)^m \Bigl( \frac{29}{16} \Bigr)^{- \frac{20}{7}}  < 2 e \Bigl( \frac{29}{16} \Bigr)^{- \frac{20}{7}} < 1.
\]
We now use the fact that $\frac{13}{29} h_m > 2$ to define $b_m \geq 1$ by
\begin{align}
    b_m &:= 1 + \frac{13}{32} \frac{( \frac{29}{16})^{h_m} (\frac{13}{29} h_m - 2) + 2}{( \frac{29}{16})^{h_m} ( \frac{13}{29} h_m - 1) + 1}, \\
    Q(r) &:= 1 + \tfrac{256}{169} \bigl[  \bigl( \tfrac{29}{16} \bigr)^{h_m} \bigl( \tfrac{13}{29} h_m - 1 \bigr) + 1 \bigr] \bigl( (r-b_m)^2 - (1-b_m)^2 \bigr). \label{eqn:Q(r)-quadratic}
\end{align}
Then, $Q(1) = 1$ and $Q(\frac{29}{16}) = ( \frac{29}{16})^{h_m}$, with $Q'( \frac{29}{16}) = h_m ( \frac{29}{16})^{h_m-1}$.
Since $(r^{h_m})''' > 0$ due to $h_m>3$, we deduce that $Q(r) \leq r^{h_m}$ for $r \geq 1$.
Note that
\[
\frac{1}{4} - \Bigl( \frac{16(b_m-1)}{13} \Bigr)^2 = \frac{1}{4} \frac{ ( ( \frac{29}{16})^{h_m}-1) [ ( \frac{29}{16})^{h_m} ( \frac{26}{29} h_m - 3) + 3] }{ \bigl[ ( \frac{29}{16})^{h_m} ( \frac{13}{29} h_m - 1) + 1 \bigr]^2 } \geq \frac{( \frac{29}{16})^{h_m}-1}{ 4 \, \bigl[ ( \frac{29}{16})^{h_m} ( \frac{13}{29} h_m - 1) + 1 \bigr] }.
\]
Recalling that $M_m < \frac{1}{4} (\frac{29}{16})^{h_m}$, we take expectations in~\eqref{eqn:Q(r)-quadratic} and usse $Q(r) \leq r^{h_m}$ to find
\[
\bE [ (R-b_m)^2] \leq \Bigl( \frac{13}{16} \Bigr)^2 \frac{M_m-1}{( \frac{29}{16})^{h_m} ( \frac{13}{29} h_m-1) + 1} + (b_m-1)^2 < \Bigl( \frac{13}{32} \Bigr)^2
\]
as desired.
This covers $m \geq 4$.
For $m=3$, the same property follows from bounding $M_3 < \frac{1}{4}(\frac{9}{5})^{h_3}$ and constructing the analogous quadratic with $\frac{29}{16}$ replaced by $\frac{9}{5}$; the adaptation to $m=2$ is analogous, with final bound $\bE[ (R^{\theta} - b^{\theta}_2)^2] \leq ( \frac{7}{20})^2 \theta^2$.
This completes the proof.
\end{proof}

\begin{lemma}\label{lemma:lots-of-letters}
Consider numbers $\sigma, \tau,b>0$, $u,v\geq 0$, and $s \in [0, \frac{b}{\sqrt{2}}]$ with $u+(\sigma-b)^2 \leq s^2$ and $v + (\tau-b)^2 \leq s^2$.
Then, it holds that $\sqrt{u}\sqrt{\tau/\sigma} + \sqrt{v} \sqrt{\sigma/\tau} \leq 2s$.
\end{lemma}
\begin{proof}
It suffices to consider $s>0$ and let $(B,x,y) := ( \frac{b}{s}, \frac{\sigma-b}{s}, \frac{\tau-b}{s})$, so $\sigma,\tau>0$ and $|\sigma-b|, |\tau-b| \leq s$ makes $|x|,|y| \leq 1$ and $B \geq \sqrt{2}$, with $\max\{ \frac{u}{1-x^2}, \frac{v}{1-y^2} \} \leq s^2$.
Also, $(\sigma, \tau) = s ( B+x, B+y)$.
We express the stated inequality in terms $(B,x,y)$, use Cauchy-Schwarz, and write $x+y$ to arrive at
\begin{align*}
    \Bigl( \sqrt{1-x^2} \textstyle{ \sqrt{\tfrac{B+y}{B+x} }} + \sqrt{1-y^2} \textstyle{ \sqrt{\tfrac{B+x}{B+y} }} \Bigr)^2 &\leq \bigl( \tfrac{1-x^2}{B+x} + \tfrac{1-y^2}{B+y} \bigr) \, [ (B+y) + (B+x)] \\
    &\leq \bigl[ 2B - x-y - 4 \tfrac{B^2-1}{2B+x+y} \bigr] (2B + x+y).
\end{align*}
Here, we used $\frac{1-x^2}{B+x} = B-x - \frac{B^2-1}{B+x}$ and $\frac{1}{z} + \frac{1}{w} = \frac{z+w}{zw} \geq \frac{4}{z+w}$ for $z,w>0$.
Consequently,
\[
\Bigl( \sqrt{1-x^2} \textstyle{ \sqrt{\tfrac{B+y}{B+x} }} + \sqrt{1-y^2} \textstyle{ \sqrt{\tfrac{B+x}{B+y} }} \Bigr)^2 \leq 4B^2-(x+y)^2 - 4(B^2-1) = 4-(x+y)^2 \leq 4.
\]
Therefore, $\sqrt{1-x^2} \sqrt{\frac{B+y}{B+x} } + \sqrt{1-y^2} \sqrt{\frac{B+x}{B+y} } \leq 2$ proves our assertion.
\end{proof}

For simplicity, we will treat $m \geq 4$ in what follows; the cases $m \in \{2,3\}$ follow from straightforward modifications of the constants as in Lemma~\ref{lemma:elementary-inequalities}.
The Hodge star operator permutes Pl\"ucker coordinates up to signs, so the anisotropy and the scalar kernel are preserved under $m \leftrightsquigarrow N-m$.
Hence, it suffices to consider $\min \{ m,N-m \} =m \leq \frac{1}{2}N$.

We will rearrange the Pl\"ucker vectors into a convenient form.
\begin{lemma}\label{lemma:balanced-stress}
Suppose that $m \geq 4$ and $q \in [ 2 - \frac{7}{10 m}, 2]$, and define
\[
F(z)^q = \sum_I A_I |\la b_I, z \rg|^q, \qquad \text{where } \quad z = \sum_I A_I \la b_I, z \rg b_I, \qquad |b_I| \leq 1, \qquad \sum_I A_I \leq M_m.
\]
For a unit simple vector $\zeta$, we denote $\nu_{\zeta} := \frac{DF(\zeta)}{F(\zeta)}$ and $r_{\zeta} := \nu_{\zeta}-\zeta$.
Then, $F^q - (q-1) |\cdot|^q$ is convex.
Moreover, for all unit simple vectors $\xi, \eta$, we have $F(\xi)^q , F(\eta)^q \in [ 1, \frac{29}{16}]$ and
\begin{equation}\label{eqn:r-xi-r-eta-inequalities}
|r_{\xi}| \, |r_{\eta}| \leq ( \tfrac{13}{32})^2 F(\xi)^{- \frac{q}{2}} F(\eta)^{- \frac{q}{2}}, \qquad |r_{\xi}| + |r_{\eta}| \leq \tfrac{13}{16} F(\xi)^{- \frac{q}{2}} F(\eta)^{- \frac{q}{2}}.
\end{equation}
\end{lemma}
\begin{proof}
Let $\zeta = v_1 \wedge \cdots \wedge v_m \in \bigwedge^{m}W$ be a simple unit length vector where $v_i$ form an orthonormal frame, so $|\zeta| = 1$.
For $t_I := \la b_I, \zeta\rg$, the defining properties of $z$ imply that $\sum_I A_I t_I^2 = 1$ and $F(\zeta)^q = \sum_I A_I t_I^2 |t_I|^{q-2}$.
Throughout, we exclude indices where $t_I = 0$ for any sum containing a negative power of $|t_I|$.
Let $\theta = \frac{2-q}{\ve_m} \in [0,1]$, recalling that $h_m \ve_m = 2$.
Then, we have
    \begin{equation}\label{eq:F-bound}
    1 \leq F(\zeta)^q \leq \Bigl(\sum_{t_I\ne0} A_It_I^2|t_I|^{-\ve_mh_m} \Bigr)^{\frac{\theta}{h_m}} = \Bigl( \sum_{t_I \neq 0} A_I \Bigr)^{\frac{\theta}{h_m}} \leq M_m^{\frac{\theta}{h_m}} \leq M_m^{\frac{1}{h_m}}.
    \end{equation}
    In particular, the computational estimates of Lemma~\ref{lemma:elementary-inequalities} imply that $1 \leq F(\zeta)^q < \frac{29}{16}$ for $m \geq 4$.
Also, combining the inequality $\sum_{t_I \neq 0} A_It_I^2 (|t_I|^{-\ve_m} )^{h_m}=\sum_{t_I \neq 0}A_I \leq M_m$ and Lemma~\ref{lemma:elementary-inequalities} gives
\begin{equation}\label{eq:weighted-moment}
    \sum_I A_It_I^2 (|t_I|^{q-2}-b_m^\theta)^2 \leq ( \tfrac{13}{32})^2, \qquad \text{ for } m \geq 4.
    \end{equation}
For any unit-length vector $v$, we can apply the Cauchy-Schwarz inequality to obtain
    \begin{align*}
        F(\zeta)^q \bigl|\la \nu_\zeta-\zeta,v\rg\bigr| &\leq \Bigl[ \sum_I A_I t_I^2 \bigl( |t_I|^{q-2}-F(\zeta)^q \bigr)^2 \Bigr]^{\frac{1}{2}} \Bigl[ \sum_I A_I \la b_I,v\rg^2 \Bigr]^{\frac{1}{2}} \\
        &= \Bigl[ \sum_I A_I t_I^2 \bigl(|t_I|^{q-2}-F(\zeta)^q \bigr)^2 \Bigr]^{\frac{1}{2}}.
    \end{align*}
Taking the supremum over all such unit-length vectors $v$, we obtain
\begin{equation}\label{eq:nu-error-variance}
    F(\zeta)^{2q}|\nu_\zeta-\zeta|^2 \leq \sum_I A_I t_I^2 (|t_I|^{q-2}-F(\zeta)^q)^2.
\end{equation}
 Since the $A_It_I^2$ are positive weights summing to $1$, the formula for $F(z)^q$ defines their weighted average of $|t_I|^{q-2}$.
    Consequently, for any $b$, 
    \[
    \sum_I A_I t_I^2 \bigl(|t_I|^{q-2}-b \bigr)^2 = \sum_I A_I t_I^2 \bigl(|t_I|^{q-2}-F(\zeta)^q \bigr)^2 + \bigl( F(\zeta)^q-b \bigr)^2.
    \]
    Specializing to $b = b_m^\theta$ as in Lemma~\ref{lemma:elementary-inequalities} and applying~\eqref{eq:weighted-moment}, we compute
    \[
    \sum_I A_I t_I^2 (|t_I|^{q-2}-F(\zeta)^q)^2 +(F(\zeta)^q-b_m^\theta)^2 \leq ( \tfrac{13}{32})^2.
    \]
    Now applying~\eqref{eq:nu-error-variance} proves that
    \begin{equation}\label{eq:scalar-short}
        F(\zeta)^{2q}|\nu_\zeta-\zeta|^2 + (F(\zeta)^q-b_m^\theta)^2 \leq ( \tfrac{13}{32})^2.
    \end{equation}
Next, we obtain similar estimates for $DF$.
Differentiating the definition of $F(z)^q$ at $\zeta$, we find
    \[
    F(\zeta)^{q-1} DF(\zeta) = \sum A_I t_I|t_I|^{q-2}b_I,
    \]
    and plugging in the definition of $\nu_\zeta$, we again obtain $\nu_\zeta:=\frac{DF(\zeta)}{F(\zeta)} =\frac1{F(\zeta)^q} \sum_I A_It_I|t_I|^{q-2}b_I$.
    Using $\zeta = \sum_I A_It_I b_I$ from the identity~\eqref{eqn:parseval-short} for $F(z)$, we arrive at
    \[
    \nu_\zeta-\zeta = \frac1{F(\zeta)^q} \sum_I A_It_I (|t_I|^{q-2}-F(\zeta)^q)b_I .
    \]
    For any $\xi, \eta$, we use this expression with $r_{\xi} := \nu_{\xi} - \xi$ and $r_{\eta} := \nu_{\eta} - \eta$.
    Euler's identity for $F$ implies
\[
    \la \nu_{\xi} , \xi \rg = \la \nu_{\eta}, \eta \rg = 1, \qquad \la r_{\xi}, \xi \rg = \la r_{\eta}, \eta \rg = 0.
\]
Consequently, we can apply Lemma~\ref{lemma:lots-of-letters} to the inequality~\eqref{eq:scalar-short}, with $\xi, \eta$ in place of $\zeta$.
This establishes the claimed inequalities~\eqref{eqn:r-xi-r-eta-inequalities} for $|r_{\xi}| , |r_{\eta}|$.

Finally, to prove the desired convexity $F^q - (q-1) |\cdot|^q$, we consider the Hessian 
    \[
    D^2 (F^q)(z) =q(q-1) \sum_I A_I|\langle b_I,z\rangle|^{q-2}b_I\otimes b_I \geq q(q-1)|z|^{q-2}\mr{Id},
    \]
    where this exists on the complement of finitely many hyperplanes.
    We can compute
    \[
    D^2((q-1)|z|^q) = q(q-1)|z|^{q-2}\mr{Id} + q(q-1)(q-2)|z|^{q-4}z\otimes z \leq q(q-1)|z|^{q-2}\mr{Id}.
    \]
    Therefore, $D^2(F^q-(q-1)|z|^q) \geq 0$ away from these hyperplanes.
    Restricting to any line, one can see that the derivative is non-decreasing, so $F^q-(q-1)|\cdot|^q$ is $C^1$ everywhere and convex.
\end{proof}

\begin{lemma}\label{lemma:two-plane-estimate}
Consider $m$-planes $T,S \in \bG(W,m)$ and let $\xi,\eta$ be oriented unit simple vectors, where $\xi$ spans $T$ and $\eta$ spans $S$, chosen so that $\la \xi, \eta \rg \geq 0$.
Then,
\begin{align*}
    K_F(T,S) &\geq 1- \la \nu_{\xi}, \eta \rg \la \nu_{\eta}, \xi \rg + (m - \textup{tr}(TS)) - (1 - \la \xi, \eta \rg^2) \\
    &\quad - \bigl[ \tfrac{3}{4} ( |r_{\xi}| + |r_{\eta}|) + |r_{\xi}| \, |r_{\eta}| \bigr] \, ( m- \textup{tr}(TS)).
\end{align*}
\end{lemma}
\begin{proof}
It suffices to work in a frame $\{ \xi,\eta \}$ for $T,S$ such that 
\[
\xi = e_1 \wedge \cdots \wedge e_m, \qquad v_i = \alpha_i e_i + \beta_i f_i, \qquad \eta = v_1 \wedge \cdots \wedge v_m,
\]
where $\alpha_i, \beta_i \geq 0$ and $\alpha_i^2+\beta_i^2=1$.
Thus, $c := \prod_i \alpha_i = \la \xi, \eta \rg$ and $d := \sum_i \beta_i^2 = m - \textup{tr}(TS)$.
Let
\[
\xi_{ij} := e_1 \wedge \cdots \wedge f_j \wedge \cdots \wedge e_m,
\]
i.e., $\xi$ with $e_i$ replaced by $f_j$.
We define $\eta_{ij}$ analogously using $w_j = \beta_j e_j - \alpha_j f_j$ in place of $v_j$.
Let
\[
g_{\xi} := \sum_i \alpha_i \beta_i \xi_{ii} - c( \eta - c\xi), \qquad g_{\eta} := \sum_i \alpha_i \beta_i \eta_{ii} - c(\xi - c \eta).
\]
We first claim that $|g_{\xi}|, |g_{\eta}| \leq \frac{3}{4}d$.
Indeed, by the orthogonality of the basis $\{ e_1, \dots, e_m , f_1, \dots, f_m \}$, we can use the inequality $1 - x \leq e^{-x}$ to bound
\begin{align}
    |g_{\xi}|^2 = |g_{\eta}|^2 &= d - \textstyle{ \sum_i} \beta_i^4 + c^2 ( 1 - c^2 - 2d), \label{eqn:g-xi-g-eta} \\
    c^2 &= \textstyle{ \prod_i} (1 - \beta_i^2) \leq e^{-d}, \notag \\
    e^d c^2 &= \textstyle{\prod_i} e^{\beta_i^2} (1 - \beta_i^2) \geq \textstyle{ \prod_i} (1 - \beta_i^4) \geq  1 - \textstyle{\sum_i} \beta_i^4. \notag
\end{align}
Hence, the inequality~\eqref{eqn:g-xi-g-eta} gives 
\[
|g_{\xi}|^2 \leq d-1 + (e^d+1 - 2d) c^2 - c^4.
\]
The right-hand side is increasing as a function of $c^2$ on $c^2 \in [0, e^{-d}]$, since $e^d+1-2d - 2 e^{-d} \geq 0$.
We therefore obtain 
\[
|g_{\xi}|^2 \leq  d + e^{-d} - e^{-2d} - 2d e^{-d} \leq \tfrac{9}{16} d^2.
\]
The latter inequality is obtained by observing that $\phi(d) := \frac{9}{16} d^2 - d - e^{-d} + e^{-2d} + 2de^{-d}$ satisfies $\phi(0) = \phi'(0) = 0$ and $\phi''(d) \geq 2 ( e^{-d} - \frac{3}{4})^2 \geq 0$, due to $d \geq 1 - e^{-d}$.

We expand the definition $(\Pi_T)_{ab} = \la \iota_{e_a} \nu_{\xi}, \iota_{e_b} \xi \rg$ in the principal frames $\{ e_i\}_{i=1}^m, \{ f_j \}_{j=1}^m$ to find
\begin{equation}\label{eqn:KF(T,S)-d-c}
    K_F(T,S) = d - c \la r_{\xi}, \eta \rg - c \la r_{\eta}, \xi \rg - \la r_{\xi}, g_{\xi} \rg - \la r_{\eta}, g_{\eta} \rg - \sum_{i,j} \beta_i \beta_j \la r_{\xi}, \xi_{ij} \rg \la r_{\eta}, \eta_{ji} \rg.
\end{equation}
Since $\nu_{\xi} = \xi + r_{\xi}$ and $\nu_{\eta} = \eta + r_{\eta}$, we obtain
\[
1 - \la \nu_{\xi}, \eta \rg \la \nu_{\eta}, \xi \rg - (1-c^2) = - c \la r_{\xi}, \eta \rg - c \la r_{\eta}, \xi \rg - \la r_{\xi}, \eta \rg \la r_{\eta}, \xi \rg.
\]
This allows us to simplify~\eqref{eqn:KF(T,S)-d-c} into
\begin{align*}
    K_F(T,S) &= 1 - \la \nu_{\xi}, \eta \rg \la \nu_{\eta}, \xi \rg + d - (1-c^2) - \la r_{\xi}, g_{\xi} \rg - \la r_{\eta}, g_{\eta} \rg \\
    & \quad + \la r_{\xi}, \eta \rg \la r_{\eta}, \xi \rg - \textstyle{\sum_{ij}} \beta_i \beta_j \la r_{\xi}, \xi_{ij} \rg \la r_{\eta}, \eta_{ji} \rg.
\end{align*}
Using the bounds $|g_{\xi}|, |g_{\eta}| \leq \frac{3}{4}d$, we see that the claimed inequality for $K_F(T,S)$ is reduced to showing that the term $\cR$ of the second line above is bounded by $d \, |r_{\xi}| \, |r_{\eta}|$.
To prove this, we identify the exterior basis associated with $(v_i, w_i)$ with the corresponding exterior basis associated with $(e_i, f_i)$.
Using $\la r_{\xi}, \xi \rg = \la r_{\eta}, \xi \rg = 0$ after the identification, while $\eta = c \xi + (\eta - c \xi)$, we express the first term of $\cR$ as the bilinear form
\[
(r_{\xi}, r_{\eta}) \mapsto \la r_{\xi}, \eta - c \xi \rg \la r_{\eta}, \eta - c \xi \rg.
\]
The associated rank-one operator is positive semidefinite, with norm $| \eta - c \xi|^2 = 1 - c^2 \leq d$.
Moreover, the second term has matrix $\begin{psmallmatrix}
    0 & \beta_i \beta_j \\
    \beta_i \beta_j & 0
\end{psmallmatrix}$ on each off-diagonal two-dimensional space $\textup{span} \{ \xi_{ij}, \xi_{ji} \}_{i \neq j}$, hence its eigenvalues are $\pm \beta_i \beta_j$.
Since $|\beta_i \beta_j| \leq \frac{1}{2} (\beta_i^2+\beta_j^2) \leq \frac{1}{2}d$, its norm is at most $d$.
Moreover, the previous rank-one form $(v,w) \mapsto \la v, \eta-c \xi \rg \la w, \eta - c \xi \rg$ vanishes on these spaces, since $\eta - c \xi$ has no off-diagonal component there.
Finally, on the orthogonal complement of these spaces, the second form $B$ is diagonal and positive semidefinite, with non-zero eigenvalues $\beta_i^2$.
Thus, again $0 \leq B \leq dI$, and alos $0 \leq (\eta - c \xi ) \otimes (\eta - c \xi) \leq dI$.
Therefore, the difference of these expressions again has norm at most $d$, and $|\cR| \leq d \, |r_{\xi}| \, |r_{\eta}|$.
This completes the proof.
\end{proof}

\begin{corollary}\label{cor:kernel}
Let $F$ be an even one-homogeneous  $C^1$ norm on $\bigwedge^m W$ such that $F^q - (q-1) |\cdot|^q$ is convex.
Moreover, suppose that, for some $q \in (1,2]$ and $\sigma>0$, we have $q - 1 - \tfrac{3}{2} \sigma - \sigma^2\geq 0$, $F(\zeta)^q \in [1,\Lambda]$ and, for every pair of unit simple vectors $\xi,\eta$,
\[
|r_{\xi}| + |r_{\eta}| \leq 2 \sigma ( F(\xi) F(\eta))^{- \frac{q}{2}}, \qquad |r_{\xi}| \, |r_{\eta}| \leq \sigma^2 \, (F(\xi) F(\eta))^{- \frac{q}{2}}.
\]
Then, for every pair of $m$-planes $T,S$, it holds that 
\[
K_F(T,S) \geq \min \{ \Lambda^{-1} (q - 1 - \tfrac{3}{2} \sigma - \sigma^2), 1 - \tfrac{3}{2} \sigma - 2 \sigma^2 \} \, (m - \textup{tr}(TS)).
\]
\end{corollary}
\begin{proof}
We fix unit simple vectors $\xi, \eta$ spanning $T,S$, oriented so that $\la \xi, \eta \rg \geq 0$.
Then, $1 - \la \xi, \eta \rg^2 \leq m - \textup{tr}(TS)$ and $\Lambda^{-1} \leq (F(\xi) F(\eta))^{- \frac{q}{2}} \leq 1$.
Then, applying Lemma~\ref{lemma:two-plane-estimate}, we obtain
\begin{equation}\label{eqn:KF(T,S)-inequality}
    \begin{split}
    K_F(T,S) &\geq 1 - \la \nu_{\xi}, \eta \rg \la \nu_{\eta}, \xi \rg + ( m - \textup{tr}(TS)) - (1 - \la \xi, \eta \rg^2) \\
    & \quad - \bigl[ \tfrac{3}{4} ( |r_{\xi}| + |r_{\eta}|) + |r_{\xi}| \, |r_{\eta}| \bigr] \, ( m - \textup{tr}(TS)).
    \end{split}
\end{equation}
Suppose first that $\la \nu_{\xi}, \eta \rg > 0$ and $\la \nu_{\eta}, \xi \rg> 0$.
Then, applying the convexity of $F^q - (q-1) |\cdot|^q$ at the vectors $\frac{\xi}{F(\xi)}$ and $\frac{\eta}{F(\eta)}$, we obtain
\[
1 - F(\eta)^{-1} F(\xi) \la \nu_{\xi}, \eta \rg \geq (q-1) F(\xi)^{1-q} F(\eta)^{-1} (1 - \la \xi, \eta \rg)
\]
as well as the symmetric inequality under $\xi \leftrightarrow \eta$.
Multiplying the corresponding upper bounds and applying the arithmetic-geometric inequality, we obtain
\begin{align*}
    & 1 - \la \nu_{\xi}, \eta \rg \la \nu_{\eta}, \xi \rg \geq (q-1) ( F(\xi) F(\eta))^{- \frac{q}{2}} (1 - \la \xi, \eta \rg^2) \\
    & \implies 1 - \la \nu_{\xi}, \eta \rg \la \nu_{\eta}, \xi \rg + (m - \textup{tr}(TS)) - (1 - \la \xi, \eta \rg^2) \geq (q-1) (F(\xi) F(\eta))^{- \frac{q}{2}} ( m- \textup{tr}(TS)).
\end{align*}
Using the assume bounds on the vectors $r_{\xi}, r_{\eta}$ in the inequality~\eqref{eqn:KF(T,S)-inequality}, we obtain
\begin{align*}
    K_F(T,S) &\geq (F(\xi) F(\eta))^{- \frac{q}{2}} (q - 1 - \tfrac{3}{2} \sigma - \sigma^2) (m - \textup{tr}(TS)) \\
    &\geq \Lambda^{-1} ( q - 1 - \tfrac{3}{2} \sigma - \sigma^2) ( m - \textup{tr}(TS)).
\end{align*}
Now, suppose that the two pairings are not both positive.
If they have opposite signs, their product is non-positive.
If both are negative, then using $\la r_{\xi}, \xi \rg = \la r_{\eta}, \eta \rg = 0$, we obtain
\[
\la \nu_{\xi}, \eta \rg \la \nu_{\eta}, \xi \rg \leq |r_{\xi}| \, |r_{\eta}| ( 1 - \la \xi, \eta \rg^2), \qquad 1 - \la \xi, \eta \rg^2 \leq \min \{ 1 , m - \textup{tr}(TS) \}.
\]
Thus, we can combine the bounds in both cases to obtain
\[
1 - \la \nu_{\xi}, \eta \rg \la \nu_{\eta}, \xi \rg + (m - \textup{tr}(TS)) - (1 - \la \xi, \eta \rg^2) \geq (1 - |r_{\xi}| \, |r_{\eta}|) \, ( m - \textup{tr}(TS)).
\]
Using the bound~\eqref{eqn:KF(T,S)-inequality} once more, together with $(F(\xi) F(\eta))^{- \frac{q}{2}} \leq 1$, we obtain $K_F(T,S) \geq (1 - \frac{3}{2} \sigma - 2 \sigma^2) ( m - \textup{tr}(TS))$.
The assertion follows from these two bounds on $K_F(T,S)$.
\end{proof}

\begin{proof}[Proof of Proposition~\ref{prop:kernel}]
    Fix $T_0, S_0 \in \bG(N,m)$ for $m \leq N/2$ and $W := T_0 + S_0$. We may project the basis $e_i$ to $a_i := \pi_W e_i$ which is a spanning set which satisfies $\sum a_i \otimes a_i =\mr{Id}\vert_W$ since $e_i$ is an orthonormal basis.     
    Throughout, we ignore all $i$ where $a_i = 0$. 
     Consider the function
    \[
    \Phi(Q) := \tfrac{2}{q}\sum_i (a_i^t Qa_i)^{\frac{q}{2}} - \log \det Q \qquad \text{for }Q \in \mr{Sym}^+(W).
    \]
    Since $\sum_i(a_i^t Qa_i)^{\frac{q}{2}} \geq \left(\sum_i a_i^t Qa_i\right)^{\frac{q}{2}} =(\on{tr}Q)^{\frac{q}{2}}$, this function goes to $+\infty$ as $\det Q \to 0$ or $\|Q\| \to \infty$; thus, it has a minimizer. 
    Differentiating in a symmetric direction $H$ shows
    \[
    D\Phi(Q)[H] = \sum_i(a_i^t Qa_i)^{ \frac{q}{2} -1}a_i^t Ha_i -\operatorname{tr}(Q^{-1}H)
    \]
    which shows that at the minimum $Q$, we have $\sum_i(a_i^t Qa_i)^{ \frac{q}{2} -1}a_i\otimes a_i=Q^{-1}$.

    We now define a useful change of variables. 
    Let $L:=Q^{\frac{1}{2}}$ and define the frame
    \[
    La_i=\lambda_i^{\frac{1}{q}}u_i,\qquad |u_i|=1,\qquad \lambda_i:=|La_i|^q
    \]
    which satisfies $\sum \lambda_i u_i \otimes u_i = \textup{Id}|_W$, with $\sum_i \lambda_i = \dim W$ and $\lambda_i \in (0,1]$.
    For every multi-index $I = (i_1 < \cdots < i_m)$, let $A_I := \lambda_{i_1} \cdots \lambda_{i_m}$ and $b_I := u_{i_1}\wedge \cdots \wedge u_{i_m}$.
    For $F_0$ the original Pl\"ucker norm, we define $F(z) := F_0((\bigwedge^m L)z)$.
    We then compute
    \begin{equation}\label{eqn:parseval-short}
    F(z)^q = \sum_I A_I| \langle b_I,z\rangle|^q, \qquad z=\sum_I A_I\langle b_I,z\rangle b_I, \qquad |z|^2=\sum_I A_I|\langle b_I,z\rangle|^2
    \end{equation}
     where $|b_I| \leq 1$ and $\sum_I A_I \leq \frac{(2m)^m}{m!} =: M_m$.
    We define $T := L^{-1}T_0$ and $S:= L^{-1}S_0$, and let $\xi, \eta$ be vectors representing the $m$-planes $T,S$, namely $\xi = \bigwedge_i e_i$ and $\eta = \bigwedge_i v_i$ where $T = \textup{span} (e_i)$ and $S = \textup{span}(v_i)$.
    Thus, we can place ourselves in the setting of Lemma~\ref{lemma:balanced-stress} and obtain the inequalities~\eqref{eqn:r-xi-r-eta-inequalities} for $r_{\xi}, r_{\eta}$, with the notation established therein.
    This, in turn, allows us to apply Lemma~\ref{lemma:two-plane-estimate} and Corollary~\ref{cor:kernel}, with $\Lambda = \frac{29}{16}$ and $\sigma = \frac{13}{32}$, and obtain
    \[
    K_F(T,S) \geq \min \{ \tfrac{16}{29} \bigl(q - 1 - \tfrac{3}{2} \cdot \tfrac{13}{32} - ( \tfrac{13}{32})^2 \bigr) , 1 - \tfrac{3}{2} \cdot \tfrac{13}{32} - 2 ( \tfrac{13}{32})^2 \} \, ( m - \textup{tr}(TS)).
    \]
    We note that $1 - \tfrac{3}{2} \cdot \tfrac{13}{32} - 2 ( \tfrac{13}{32})^2 = \frac{31}{512}$ and $1 + \tfrac{3}{2} \cdot \tfrac{13}{32} + ( \tfrac{13}{32})^2 = \frac{1817}{1024}$.
    Thus, the above minimum is at least $10^{-3}$ for all $q \geq \frac{9}{5}$, hence $K_F(T,S) \geq 10^{-3} ( m- \textup{tr}(TS))$.
    
    Finally, we undo the original change of coordinates $T=L^{-1}T_0$ and $S=L^{-1}S_0$.
    For $|v|=1$, we have $\sum_i\lambda_i|\la u_i,  v\rg|^2 =1$ and $|Lv|^2 = \sum_i \lambda_i|\la u_i, v\rg |^2\lambda_i^{(2-q)/q}$.
    Since $\lambda_i|\la u_i, v \rg|^2\leq \lambda_i\leq 1$,
    \[
    |Lv|^2
 =\sum_i\lambda_i^{(2-q)/q}
        \lambda_i|\langle u_i,v\rangle|^2
 \leq\sum_i\lambda_i|\langle u_i,v\rangle|^2
 =1, \quad 
    \mbox{and} \quad   |Lv|^2
 =\sum_i (\lambda_i|\la u_i, v \rg|^2)^\frac{2}{q} \geq N^{1 - \frac{2}{q}},
    \]
    by the identities~\eqref{eqn:parseval-short}.
    Therefore, $N^{1 - \frac{2}{q}} \leq s_{\min}(L)^2 \leq s_{\max}(L)^2 \leq 1$ which shows $\kappa(L)^2\leq N^{(2-q)/q}$, where $\kappa(L)$ denotes the ratio of the largest to the smallest singular value of $L$.
    Furthermore, for any invertible $A$ and $m$-planes $P,Q$, the singular value ratio $\kappa(A)$ bounds the distance between the plane projections, yielding
    \[
    m-\mr{tr}(\pi_{AP}\pi_{AQ}) \geq \kappa(A)^{-2} (m-\mr{tr}(PQ)).
    \]
    Because $\Pi_T$ transforms by conjugation, i.e.~$\Pi^F_T =L \Pi^{F_0}_{T_0}L^{-1}$ and $\Pi^F_S = L \Pi^{F_0}_{S_0} L^{-1}$, we have
    \[
    K_F(T,S) = K_{F_0}(T_0,S_0) \geq \tfrac{1}{270} N^{1-\frac{2}{q}} \bigl(m-\operatorname{tr}(T_0S_0)\bigr)
    \]
    by the chain rule.
    This establishes our claim and completes the proof.
\end{proof}

\begin{proof}[Proof of Theorem~\ref{thm:ell-p-on-G(N,m)}]
For $q\leq2$ and Euclidean unit Pl\"ucker vectors,
$\Psi_q(T)=\|\xi_T\|_{\ell^q}\geq1$.
By~\eqref{eqn:normalized-scalar-kernel},
Proposition~\ref{prop:kernel}, and
$\|T-S\|^2=2(m-\operatorname{tr}(TS))$, we obtain
\[
 \langle B_{\Psi_q}(T),B_{\Psi_q^*}(S^\perp)\rangle
 \geq C_{N,m,q} \|T-S\|^2.
\]
This proves USAC in every stated dimension and codimension.
\end{proof}

\begin{remark}\label{rmk:when-usac-fails}
For $2 \leq m \leq N-2$ and every $q>4$, the $\ell^q$ anisotropy $\Psi_q$ does not satisfy the scalar atomic conditions.
Concretely, there are families of distinct $m$-planes $T_t, S_t \in \bG(N,m)$ converging to the same plane as $t \downarrow 0$, such that
\begin{equation}\label{eqn:sac-failure}
m - \textup{tr}(\Pi_{T_t} \Pi_{S_t}) = \frac{\la B_{\Psi_q}(T_t), B_{\Psi_q^*}(S_t^\perp)\rg}{\Psi_q(T_t)\Psi_q(S_t)} = -\frac{q-1}{4}\,t^4+o(t^4), \qquad \text{where } \; \Pi_P := \frac{B_{\Psi_q}(P)}{\Psi_q(P)}.
\end{equation}
For simplicity, we construct examples on $\bG(4,2)$; these extend to $\bG(N,2)$, by inclusion, and then to $\bG(N,m)$, via $( \tilde{T}_t, \tilde{S}_t) := (T_t \oplus E, S_t \oplus E)$ with $E := \textup{span} \{ e_5, \dots, e_{m+2} \}$.
For $A \in \bR^{2 \times 2}$, the non-parametric graph integrand is $F_q(A) := ( 1 + \sum_{i,j=1}^2 |A_{ij}|^q + |\det A|^q)^{\frac{1}{q}}$, and $G_A := \frac{D F_q(A)}{F_q(A)}$ has
\begin{align}
&\Pi_{\textup{graph} \, A} = \begin{pmatrix}
    I_2 - A^t G_A \\
    G_A
\end{pmatrix} \begin{pmatrix}
    I_2 & A^t
\end{pmatrix}, \notag \\
&2 - \textup{tr}(\Pi_{\textup{graph} \, A} \Pi_{\textup{graph} \, B}) = \textup{tr} ( (A-B)^t (G_A - G_B) ) + \textup{tr} ( (A-B)^t G_B (A-B)^t G_A). \label{eqn:m-tr-sac}
\end{align}
Consequently, we can define planes $T_t = \textup{graph} \, A_t$ and $S_t = \textup{graph} \, B_t$, for $t \in (0,\frac{1}{2})$, where
\[
A_t =  \begin{pmatrix}
    \sqrt{1-2t^2} & t \\
    - 2t & \sqrt{1-2t^2}
\end{pmatrix}, \qquad B_t = \begin{pmatrix}
    \sqrt{1-2t^2} & 2t \\
    -t & \sqrt{1-2t^2}
\end{pmatrix}.
\]
These planes are distinct for every $t>0$ and converge to the graph of $I_2$ as $t \downarrow 0$, and
\begin{align*}
G_{A_t} &= F_q(A_t)^{-q} \begin{pmatrix}
    (1-2t^2)^{\frac{1}{2}} + (1-2t^2)^{\frac{q-1}{2}} & 2t + t^{q-1} \\
    - (t+2^{q-1} t^{q-1}) & (1-2t^2)^{\frac{1}{2}} + (1-2t^2)^{\frac{q-1}{2}}
\end{pmatrix}, \\
G_{B_t} &= F_q(B_t)^{-q} \begin{pmatrix}
    (1-2t^2)^{\frac{1}{2}} + (1-2t^2)^{\frac{q-1}{2}} & t+2^{q-1} t^{q-1}  \\
    - (2t + t^{q-1}) & (1-2t^2)^{\frac{1}{2}} + (1-2t^2)^{\frac{q-1}{2}}
\end{pmatrix},
\end{align*}
where $F_q(A_t)^q = F_q(B_t)^q = 2 + 2 (1-2t^2)^{\frac{q}{2}} + (1+2^q) t^q$.
Moreover, $A_t - B_t = -t \begin{psmallmatrix}
    0 & 1 \\ 1 & 0
\end{psmallmatrix}$, so substituting into the expression~\eqref{eqn:m-tr-sac} and expanding terms gives $2 - \textup{tr}( \Pi_{T_t} \Pi_{S_t}) = - \frac{q-1}{4} t^4 + o(t^4)$ for $t \downarrow 0$.
\end{remark}

\subsection{Families of USAC integrands}\label{sec:axisymmetric}

We now prove Theorem~\ref{thm:axisymmetric-usac} and Corollary~\ref{cor:ha-integrands}: we construct axisymmetric USAC integrands that are strictly separated from any ellipsoidal area integrand.

\begin{proof}[Proof of Theorem~\ref{thm:axisymmetric-usac}]

Let us recall our setup: we define an anisotropy $\Psi_h \in C^2(\bG(N,m))$ by $\Psi_h(T) := h( |Te|, |T^{\perp}e|)$, where $h \in C^2 (\bR^2 \setminus \{ 0 \}, \bR_+)$.
We set $H(\theta) := h(\cos(\theta), \sin(\theta))$ for $0 \leq \theta \leq \frac{\pi}{2}$, and will show that $\Psi_h$ satisfies USAC if and only if
\begin{equation}\label{eqn:axisymmetric-usac-criterion}
H(\theta)+H''(\theta)>0 \quad\text{for every }\theta\in[0,\tfrac{\pi}{2}],
\end{equation}
in which case it also supports a Michael-Simon inequality.
Let $W=e^\perp$.
For $T\in\bG(N,m)$ such that $|Te|=\cos(\theta)$ and $|T^\perp e|=\sin(\theta) $, write $T = \bR t \oplus Z$ and $t = \cos(\theta) \, e + \sin(\theta) \, r$, where $r\in W$ is a unit vector and $Z\subset W\cap r^\perp$ has dimension $m-1$.
The normal direction is $n=-\sin(\theta)  e+\cos(\theta) r$.
The anisotropic first variation identity~\eqref{eqn:B-psi-identity} computes the stress tensor as
\begin{equation}\label{eqn:axisymmetric-stress}
    B_{\Psi_h}(T)
    =H(\theta)T+H'(\theta)n\otimes t.
\end{equation}
Therefore, $\Pi_T:=\frac{B_{\Psi_h}(T)^t}{\Psi_h(T)} = T+\frac{H'(\theta)}{H(\theta)} t\otimes n$.
Since $B_{\Psi_h^*}(S^\perp)=\Psi_h(S)(I-\Pi_S)$, we have
\begin{equation}\label{eqn:axisymmetric-kernel}
    \frac{\la B_{\Psi_h}(T),B_{\Psi_h^*}(S^\perp)\rg}
    {\Psi_h(T)\Psi_h(S)}
    =
    m-\mr{tr}(\Pi_T\Pi_S).
\end{equation}
First, we suppose that $H + H'' > 0$ and we will show that $\mr{tr}(\Pi_T \Pi_S) \leq m$ with equality if and only if $T = S$.
Decompose $S = \bR u \oplus Y$ and use $\varphi$ as the associated angle function and unit vector $r'$ analogous to that for $T$.
We use subscripts to denote partial derivatives of $h$ and prime to denote derivatives with respect to $\theta$.
From Euler's identity, $\cos(\theta) h_1 + \sin(\theta)  h_2 = H(\theta)$ using, we record the following useful equations which we will repeatedly use:
\begin{equation}\label{eqn:DerivativeEuler}
\begin{aligned}
h_1(\theta)&=\cos(\theta)\,H(\theta)-\sin(\theta) \,H'(\theta),&
    h_2(\theta)&=\sin(\theta) \,H(\theta)+\cos(\theta)\,H'(\theta),\\
    h_1'(\theta)&=-\sin(\theta) \,[H(\theta)+H''(\theta)],&
    h_2'(\theta)&=\cos(\theta)\,[H(\theta)+H''(\theta)].
\end{aligned}
\end{equation}
We compute
\begin{align*}
    \Pi_T &= t\otimes \frac{h_1(\theta)e+h_2(\theta)r}{H(\theta)} +\pi_Z \qquad \text{and}\qquad \Pi_S = u \otimes \frac{h_1(\varphi)e + h_2(\varphi)r'}{H(\varphi)} + \pi_Y, \\
    \Pi_T\Pi_S &= \left( t\otimes \frac{h_1(\theta)e+h_2(\theta)r}{H(\theta)} \right) \left( u\otimes \frac{h_1(\varphi)e+h_2(\varphi)r'}{H(\varphi)} \right) \\ 
    &\quad + \left( t\otimes \frac{h_1(\theta)e+h_2(\theta)r}{H(\theta)} \right)\pi_Y  + \pi_Z \left( u\otimes \frac{h_1(\varphi)e+h_2(\varphi)r'}{H(\varphi)} \right) +\pi_Z\pi_Y.
\end{align*}
Since $Z \subset e^\perp \cap r^\perp$ and $Y \subset e^\perp \cap (r')^\perp$, we can express this as
\begin{align*}
    \Pi_T\Pi_S =& \frac{ \cos(\varphi) h_1(\theta) +\sin(\varphi)(\la r, r'\rg)h_2(\theta)} {H(\theta)}  t\otimes \frac{h_1(\varphi)e+h_2(\varphi)r'}{H(\varphi)} \\ 
    &\quad + \frac{h_2(\theta)}{H(\theta)} t\otimes\pi_Yr+ \frac{\sin(\varphi)}{H(\varphi)} \pi_Z r'\otimes (h_1(\varphi)e+h_2(\varphi)r') + \pi_Z\pi_Y 
\end{align*}
so taking the trace gives
\begin{align*} 
    \mr{tr}(\Pi_T\Pi_S) =& \frac{ \bigl(\cos(\varphi) h_1(\theta) +\sin(\varphi)(\la r, r'\rg)h_2(\theta)\bigr) \bigl(\cos(\theta) h_1(\varphi) +\sin(\theta) (\la r, r'\rg)h_2(\varphi)\bigr)} {H(\theta)H(\varphi)} \\
    &\quad + \frac{\sin(\theta)  h_2(\theta)}{H(\theta)} |\pi_Yr|^2 + \frac{\sin(\varphi) h_2(\varphi)}{H(\varphi)} |\pi_Z r'|^2 + \mr{tr}(\pi_Y\pi_Z).
\end{align*} 
Since $h$ is an absolute norm, it is non-decreasing in each coordinate, so $h_1,h_2\geq 0$.
Euler's identity shows $0\leq\frac{\sin\theta\,h_2(\theta)}{H(\theta)}\leq 1$ and the same holds for $\varphi$.
We have the inequalities
\[
    \mr{tr}(\pi_Y\pi_Z)
    \leq m-1-
    \max\{|\pi_Z r'|^2,|\pi_Yr|^2\},
    \qquad
    |\pi_Z r'|^2,\ |\pi_Yr|^2\leq1-(\la r, r'\rg)^2.
\]
Applying these inequalities to the trace, we obtain
\begin{align*}
    \mr{tr}(\Pi_T\Pi_S) &\leq m-1 + \frac{ \bigl(\cos(\varphi) h_1(\theta)
      +\sin(\varphi) \la r,r' \rg h_2(\theta)\bigr) \bigl(\cos(\theta) h_1(\varphi)
      +\sin(\theta)  \la r,r' \rg h_2(\varphi)\bigr)} {H(\theta)H(\varphi)} \\
    &\quad+ \max\left\{ 0,  \frac{\sin(\theta)  h_2(\theta)}{H(\theta)} + \frac{\sin(\varphi) h_2(\varphi)}{H(\varphi)} -1 \right\}(1- \la r,r' \rg^2).
\end{align*}
The right-hand side, after subtracting $m-1$, is a convex quadratic function of $\la r, r'\rg$.  
If the max term is 0, then the quadratic coefficient is $\frac{\sin(\theta) \, h_2(\theta)}{H(\theta)} \frac{\sin(\varphi)h_2(\varphi)}{H(\varphi)} \geq 0$.

Otherwise, the coefficient is
\begin{align*}
\tfrac{\sin(\theta) \,h_2(\theta)}{H(\theta)} \tfrac{\sin\varphi\,h_2(\varphi)}{H(\varphi)} -\tfrac{\sin(\theta) \,h_2(\theta)}{H(\theta)} -\tfrac{\sin\varphi\,h_2(\varphi)}{H(\varphi)} +1 = \bigl( 1-\tfrac{\sin(\theta) \,h_2(\theta)}{H(\theta)} \bigr) \bigl( 1-\tfrac{\sin\varphi\,h_2(\varphi)}{H(\varphi)} \bigr) \geq 0
\end{align*}
using $0\leq \frac{\sin(\theta) \, h_2(\theta)}{H(\theta)}\leq 1$ from Euler's identity.

If  $0 < \theta,\varphi< \frac{\pi}{2}$, absolute symmetry gives $h_2(0)=h_1(\frac{\pi}{2})=0$, and $\frac{h_1(\theta)}{\cos(\theta)}, \frac{h_2(\theta)}{\sin(\theta)}$ extend continuously to their boundaries on $[0,\frac{\pi}{2}]$, as do all other identities.
The above derivative identities~\eqref{eqn:DerivativeEuler}, combined with the Euler identity and the assumption $H+H''>0$, imply that $h_1, h_2>0$.
Therefore, the quadratic coefficient is strictly positive in either case, so the right-hand side is maximized for $\la r,r'\rg\in\{-1,1\}$.
Using the convexity and $1$-homogeneity of $h$ and equation~\eqref{eqn:DerivativeEuler}, we can bound
\[
|\cos(\varphi)h_1(\theta) \pm\sin(\varphi)h_2(\theta)| \leq H(\varphi)  \qquad \text{and}\qquad | \cos(\theta) \, h_1(\varphi) \pm \sin(\theta) \, h_2(\varphi) |\leq H(\theta),
\]
so at the extremal values, $\la r,r'\rg=\pm1$, the final term containing $1-\la r,r'\rg^2$ and the preceding trace estimate gives $\mr{tr}(\Pi_T\Pi_S)\leq m$.
Because $H+H''>0$, the sublevel set $\{h < 1\}$ is strictly convex, so any supporting line meets the boundary at a unique point.
Equality in 
\[
|\cos(\varphi)h_1(\theta)\pm\sin(\varphi)h_2(\theta)|\leq H(\varphi)
\]
can only occur when $(\cos(\varphi),\pm\sin(\varphi))=\pm(\cos(\theta),\sin(\theta))$.
Together with the equality case in the projection estimate, equality can occur only when $T=S$, thus
\[
m-\mr{tr}(\Pi_T\Pi_S)>0 \qquad \text{for }T\neq S.
\]
To prove the uniform USAC estimate, we now consider $T_\ve = \mr{graph}(\ve X) \to T$ for some $X \in T_T\bG(N,m)$, which is a linear map $X:T\to T^\perp$.
For any rank-$m$ projections $P,Q$, $m-\mr{tr}(PQ)=\tfrac{1}{2}\mr{tr}((P-Q)^2)$.
We can then apply the above estimate to $\Pi_T$ and $\Pi_{T_\ve}$ showing
\[
 m-\mr{tr}(\Pi_T\Pi_{T_\ve})
 =\tfrac12\ve^2\mr{tr}((D\Pi_T[X])^2)+o(\ve^2).
\]
For an interior angle set $k=H'/H$ and
$U=W\cap(\bR r\oplus Z)^\perp$.
We decompose $X$ orthogonally into
\[
 X = \langle Xt,n\rangle n \otimes t + \pi_UXt\otimes t + n \otimes \pi_ZX^tn +\pi_UX|_Z.
\]
From our coordinate definitions, we compute $\dot\theta=\la Xt, n\rg$, $\dot t=\la Xt, n\rg n+\pi_UXt-\tan\theta\,\pi_ZX^tn$, and $\dot n=-\la Xt, n\rg t-\pi_ZX^tn+\cot\theta\,\pi_UXt$ and therefore
\[
D \Pi_T[X] = X + X^t + ( \tfrac{H'}{H})' \la X t, n \rg \, t \otimes n + \tfrac{H'}{H} \dot{t} \otimes n + \tfrac{H'}{H} t \otimes \dot{n}.
\]
Expanding the stress tensor~\eqref{eqn:axisymmetric-stress} for $T_\ve=\mr{graph}(\ve X)$ yields
\begin{equation}\label{eqn:axisymmetric-second-variation}
\begin{split}
m - \textup{tr} ( \Pi_T \Pi_{T_{\ve}}) &= \ve^2 \bigl( \tfrac{H + H''}{H} |\la X t, n \rg|^2 + \tfrac{h_2}{H \, \sin \theta} |\pi_{W \cap ( \bR r \oplus Z)^{\perp} } Xt|^2 \bigr) \\
& \quad + \ve^2 \bigl( \tfrac{h_1}{H \, \cos \theta} |\pi_Z X^t n|^2 + \| \pi_{W \cap ( \bR r \oplus Z)^{\perp}} X|_Z \|^2  \bigr) + o (\ve^2).
\end{split}
\end{equation}
Each coefficient can be bounded uniformly using the property that $h$ is an absolute norm, meaning $h(x,y)=h(x,-y)=h(-x,y)$, and therefore $h_2(0)=0$ and $h_1(\frac{\pi}{2})=0$.
We therefore obtain
\begin{align*}
    h_1(\theta) &= \int_\theta^{\frac{\pi}{2}}
    \sin u\,[H(u)+H''(u)]\,du \geq \min_{[0,\pi/2]}(H+H'')\,\cos(\theta), \\
    h_2(\theta) &= \int_0^\theta \cos u\,[H(u)+H''(u)]\,du \geq \min_{[0,\pi/2]}(H+H'')\,\sin(\theta) .
\end{align*}
Thus, if $H+H''>0$ on $[0,\pi/2]$, each term in~\eqref{eqn:axisymmetric-second-variation} is non-negative, and $m - \textup{tr}(\Pi_T \Pi_{T_{\ve}}) \geq c \ve^2 \|X\|^2 + o(\ve^2)$, for some $c>0$ independent of $T$ and $X$.
In particular, we obtain $m - \textup{tr}( \Pi_T \Pi_S) \geq c \|T-S\|^2$ for $S$ is sufficiently close to $T$, uniformly in $T$.
By the compactness of $\bG(N,m)$, we conclude that $\Psi$ is USAC.
The converse is analogous, upon computing $m - \textup{tr}( \Pi_T \Pi_{T_{\ve}})$ for an angular variation
\[
T_{\ve} = \bR ( \cos (\theta+\ve) e + \sin (\theta+\ve) r) \oplus Z
\]
where $r, Z$ are arbitrary.
Thus, USAC implies $H(\theta)+H''(\theta)>0$ as claimed.

Next, these integrands support a Michael-Simon inequality, by applying~\cite{anisotropic-michael-simon}*{Theorem 4.1} with the concave barrier $\lambda_m(A):=\min_{|v|=1}\la Av,v\rg$.
For $U\in\bG(W,m)$, the property~\eqref{eqn:axisymmetric-stress} yields
\begin{equation}\label{eqn:horizontal-compression}
    \pi_U B_{\Psi_h}(T)|_U=\sin(\theta) \, h_2(\theta) (\pi_Ur)\otimes(\pi_Ur) + H(\theta)\pi_U\pi_Z|_U \geq 0 .
\end{equation}
For $\sin(\theta) >0$, the preceding inequalities give $h_2(\theta)>0$, and~\eqref{eqn:horizontal-compression} is positive definite for
$U=\bR r\oplus Z$. 
Thus, $\lambda_m$ has a strictly positive Haar average over $\bG(W,m)$ unless $e\in T$, where it vanishes.

We therefore must glue in a different concave barrier for the set $Z := \{T : e \in T\}$.
Consider the probability measure $\rho_e$ which is the pushforward of the Haar measure on $\bG(e^\perp, m-1)$ via the map $U \mapsto \bR e \oplus U$. 
On a relatively compact neighborhood of $Z$, we have
\[
B(T) := \int \lambda_m(E_P^t B_{\Psi_h}(T) E_P)\, d\rho_e(P)
\]
does not vanishes since $B_{\Psi_{h}}(T) = H(0) T$ for $T \in Z$, which is strictly positive.
Therefore,~\cite{anisotropic-michael-simon}*{Theorem 4.1 and Remark 1} demonstrates that we may combine these two region to a global measure and barrier providing a Michael-Simon inequality for $\Psi_h$.
\end{proof}

\begin{corollary}\label{cor:ha-integrands}
For every $a>1$, the norm
\begin{equation}\label{eqn:ha-example}
    h_a(x,y) := \tfrac{1}{a+1} \textstyle( \sqrt{a^2x^2 + y^2}  +\sqrt{x^2 + a^2 y^2})
\end{equation}
satisfies the conditions of Theorem~\ref{thm:axisymmetric-usac}.
Notably, the integrand $\Psi_{h_a}$ is USAC and supports a Michael-Simon inequality for every $2\leq m<N$.

Let $\Psi_g$ be any ellipsoidal area anisotropy induced by an arbitrary inner product $g$ on $\bR^N$.
Then,
\begin{equation}\label{eqn:far-from-ellipsoidal}
    \inf_{g = g^t > 0}\|\Psi_{h_a}-\Psi_g\|_{C^0(\bG(N,m))} = \tfrac{1}{2} \, \bigl( \tfrac{1}{a+1} \textstyle{\sqrt{2(a^2+1)}} - 1 \bigr).
\end{equation}
In particular, we have $\lim_{a \to \infty} \inf_{g = g^t} \| \Psi_{h_a} - \Psi_g \|_{C^0(\bG(N,m))} = \frac{\sqrt{2}-1}{2}$.
Consequently, $\Psi_a$ is far from the area functional under every ellipsoidal transformation.
\end{corollary}

\begin{proof}
    For $\Psi_{h_a}$ and $H_a(\theta):=h_a(\cos\theta,\sin\theta)$, we can compute
\[
    H_a+H_a'' = \frac{a^2}{a+1} \left[ (a^2\cos^2 \theta+\sin^2 \theta)^{- \frac{3}{2}} + (\cos^2 \theta+a^2\sin^2 \theta)^{-\frac{3}{2}}\right]>0,
\]
so $\Psi_{h_a}$ is USAC and satisfies the Michael-Simon inequality, by Theorem~\ref{thm:axisymmetric-usac}.
Given a positive-definite symmetric matrix
$g\in\bR^{N\times N}$, let $\Psi_g(T) := \sqrt{\det\bigl(\la g v_i,v_j\rg\bigr)}$ where $v_1,\ldots,v_m$ is any Euclidean orthonormal basis of $T$.
We fix a unit-length $r\perp e$ and $Z\in\bG((\bR e\oplus\bR r)^\perp,m-1)$.
Choose an orthonormal basis $z_2,\ldots,z_m$ of $Z$.
For every unit vector $v \in \bR e\oplus\bR r$, we have
\begin{align*}
    \Psi_g ( \mathbb{R} v \oplus Z)^2 &= \det \begin{pmatrix}
        \la gv,v rg & \la gv, z_2 \rg & \cdots & \la gv,z_m \rg \\
        \la gz_2, v \rg & \la gz_2, z_2 \rg & \cdots & \la g z_2, z_m \rg \\
        \cdots & \cdots & \ddots & \vdots \\
        \la gz_m, v \rg & \la gz_m , z_2 \rg & \cdots & \la gz_m, z_m \rg
    \end{pmatrix}.
\end{align*}
Applying the parallelogram identity to $v_\pm=\frac{e\pm r}{\sqrt2}$ yields
\begin{align*}
    \Psi_g ( \bR v_+ \oplus Z)^2 + \Psi_g(\bR v_- \oplus Z)^2 = \Psi_g (\bR e \oplus Z)^2 + \Psi_g(\bR r \oplus Z)^2.
\end{align*}
For $\Psi_{h_a}$, we have $\Psi_{h_a}(\bR e \oplus Z) = \Psi_{h_a}(\bR r\oplus Z) =1$.
This implies that
\[
\Psi_{h_a} ( \bR v_{\pm} \oplus Z) = h_a ( \tfrac{1}{\sqrt{2}}, \tfrac{1}{\sqrt{2}}) = \tfrac{\sqrt{2(a^2+1)}}{a+1} = : M_a.
\]
For any positive-definite symmetric matrix $g$, the conclusion is immediate if $\| \Psi_{h_a} - \Psi_g \|_{C^0} \geq M_a$, so let us assume that $\| \Psi_{h_a} - \Psi_g \|_{C^0} > M_a$.
Applying the parallelogram identity to $\frac{e \pm r}{\sqrt{2}}$ shows that
\[
\Psi_g( \bR v_{\pm} \oplus Z) \geq M_a - \| \Psi_{h_a} - \Psi_g \|_{C^0}, \qquad \max \{ \Psi_g (\bR e \oplus Z), \Psi_g(\bR r \oplus Z) \} \leq 1+ \| \Psi_{h_a} - \Psi_g \|_{C^0}.
\]
Combining these properties, we deduce that
\[
2 ( M_a - \| \Psi_{h_a} - \Psi_g \|_{C^0})^2 \leq 2 ( 1 + \| \Psi_{h_a} - \Psi_g \|_{C^0})^2.
\]
Thus, we arrive at
\[
\inf_{g = g^t > 0} \| \Psi_{h_a} - \Psi_g \|_{C^0} \geq \tfrac{1}{2} (M_a - 1).
\]
Conversely, for any plane $T$ and $\theta$ the angle between $T$ and $e$, we compute
\[
\Psi_{h_a}(T) = \frac{\sqrt{a^2\cos^2 \theta + \sin^2 \theta} + \sqrt{\cos^2 \theta + a^2\sin^2 \theta}}{a  + 1}.
\]
The right-hand side, viewed as a function of $\cos^2 \theta$, is concave and symmetric about $\frac{1}{2}$, so it is minimized at $\theta = 0,\frac{\pi}{2}$ and maximized at $\theta = \frac{\pi}{4}$. 
Consequently, letting $g := \left(\frac{1 + M_a}{2}\right)^\frac{2}{m}\mr{Id}$ yields 
\[
\| \Psi_{h_a} - \Psi_g\|_{C^0} = \max \bigl\{\tfrac{1 + M_a}{2} - 1, M_a - \tfrac{1 + M_a}{2} \bigr\} = \tfrac{M_a - 1}{2}.
\]
Finally, taking this value of $g(a)$ as $a \to \infty$ proves the asymptotic distance claim.
\end{proof}

\smallskip \noindent {\bf Data Availability Statement.} No data was used for the research described in the
article.

\smallskip \noindent {\bf Conflict of interest.} The authors have no known conflict of interest for this work.

\bibliography{ref}

\end{document}